\documentclass[11pt,reqno]{amsart}
\usepackage{txfonts}
\usepackage{color}
\usepackage{arydshln}
\usepackage{extarrows}
\usepackage{fourier}
\usepackage{bbm}
\usepackage{amsthm}
\usepackage{amsfonts}
\usepackage{amsxtra}
\usepackage[backref=page]{hyperref}
\usepackage{amssymb}
\usepackage{amscd}
\usepackage{amsthm}
\usepackage{mathrsfs}
\usepackage{leqno}
\usepackage{multirow}
\usepackage[all]{xypic}
\usepackage{graphicx}
\usepackage{amstext, amscd, setspace, mathtools, enumerate, graphics, latexsym, siunitx}

\usepackage{pst-node}
\usepackage{tikz-cd} 

\usepackage{amssymb}

\usepackage{PTSerif} 
\usepackage{graphicx}

\usepackage[cmtip,all]{xy}
\hypersetup{linkcolor=blue, urlcolor=green, pdfstartview=FitH }
\makeatletter
\def\resetMathstrut@{%
  \setbox\z@\hbox{%
    \mathchardef\@tempa\mathcode`\(\relax
    \def\@tempb##1"##2##3{\the\textfont"##3\char"}%
    \expandafter\@tempb\meaning\@tempa \relax
  }%
  \ht\Mathstrutbox@1.2\ht\z@ \dp\Mathstrutbox@1.2\dp\z@
}

\numberwithin{equation}{section}
\newtheorem*{theorem*}{Theorem}
\newtheorem{lemma}{Lemma}[section]
\newtheorem{proposition}[lemma]{Proposition}
\newtheorem{remark}[lemma]{Remark}

\newtheorem{theorem}[lemma]{Theorem}

\newtheorem{corollary}[lemma]{Corollary}

\newtheorem*{question*}{Question}
\newtheorem*{assumption*}{Assumption}
\newtheorem*{axiom*}{Axiom}

\newcommand{\End}{\operatorname{End}}
\newcommand{\Hom}{\operatorname{Hom}}

\newcommand{\Irr}{\operatorname{Irr}}

\newcommand{\Span}{\operatorname{Span}}
\newcommand{\cInd}{\operatorname{c-Ind}}
\newcommand{\Ind}{\operatorname{Ind}}

\newcommand{\Ha}{\operatorname{H}}

\newcommand{\C}{\mathbb C}
\newcommand{\F}{\mathbb F}

\newcommand{\Q}{\mathbb Q}
\newcommand{\R}{\mathbb R}
\newcommand{\Z}{\mathbb Z}

\newcommand{\GL}{\operatorname{GL}}

\newcommand{\GSp}{\operatorname{GSp}}
\newcommand{\Mp}{\operatorname{Mp}}
\newcommand{\SL}{\operatorname{SL}}

\newcommand{\Sp}{\operatorname{Sp}}

\newcommand{\M}{\operatorname{M}}

\newcommand{\Ps}{\operatorname{Ps}}

\newcommand{\diag}{\operatorname{diag}}

\newcommand{\id}{\operatorname{Id}}
\newcommand{\Za}{{\mathbb Z}}

\newcommand{\supp}{\operatorname{supp}}

\usepackage[top=1in, bottom=1in, left=1.25in, right=1.25in]{geometry}

\begin{document}
\title{On the lattice model of the  Weil representation}
\date{\today}
\author{Chun-Hui Wang}
\address{School of Mathematics and Statistics\\Wuhan University \\Wuhan, 430072,
P.R. CHINA}
\email{cwang2014@whu.edu.cn}
\keywords{ Metaplectic group, Theta correspondence, Weil representation}
\subjclass[2010]{11F27 (Primary), 20G25 (Secondary).}
\maketitle
\setcounter{tocdepth}{1}
\begin{abstract}
Let $F$ be a local field. In the case of $F$ being  the real field,  Pierre Cartier constructed Heisenberg-Weil representations of a Heisenberg group in families  using non-self-dual lattices.  This result was later reformulated by Jae-Hyun Yang in another paper.  We extend this family of representations to a representation of  a Jacobi subgroup by incorporating a rational Metaplectic group. In the case of $F$ being a non-archimedean local field of odd residual characteristic or a finite unramified extension of the field $\Q_2$ of $2$-adic numbers, we obtain similar results for a Jacobi group by incorporating a Metaplectic group.
\end{abstract}
\maketitle
\setcounter{secnumdepth}{7}
\tableofcontents{}
\section{Introduction}
\subsection{Notations and conventions}
In the whole paper, we will use the following notations and conventions:
\begin{itemize}
\item We will write  $F$ for  a local field. We only consider the following three cases:
\begin{itemize}
\item $F$ is a non-archimedean local field of odd residual characteristic.
\item $F$ is a finite unramified extension  of the field  of $2$-adic numbers.
\item $F$ is the  real field. 
\end{itemize}
If $F$ is a non-archimedean local field,  we denote its residue field as $k_F$ and choose a uniformizer $\varpiup$ for $F$. Let $q$ denote the cardinality of $k_F$.    $\mathcal{O}$ represents the ring of integers, with $\mathcal{P}$ being its prime ideal.  The group of units in  $\mathcal{O}$ is denoted as $U$, whereas $U_n$ is the subgroup containing those    $u \in F^{\times}$  such that   $ u\equiv 1 \mod \mathcal{P}^{n}$.
\item We will let $(W, \langle, \rangle)$ denote  a symplectic vector space  of  dimension $2m$   over  $F$.  Let  $\{e_1, \cdots, e_m; e_1^{\ast}, \cdots, e_m^{\ast}\}$  be a symplectic basis  of $W$ so that $\langle e_i, e_j\rangle=0=\langle e_i^{\ast}, e_j^{\ast}\rangle$, $\langle e_i,e_j^{\ast}\rangle=\delta_{ij}$. Let $X=\Span_{F}\{e_1, \cdots, e_m\}$, $X^{\ast}=\Span_{F}\{e^{\ast}_1, \cdots, e^{\ast}_m\}$.  Let $\Ha(W)=W\oplus F$ denote the usual Heisenberg group,  defined as follows:
 $$(w, t)\cdot (w',t')=(w+w', t+t'+\tfrac{1}{2}\langle w,w'\rangle),$$
 for $w, w'\in W$, $t, t'\in F$. 
 \item If $F$ is a non-archimedean local field, we will employ the notations and results from \cite{BushH} concerning smooth representations of locally profinite groups, without detailed reference. When addressing the real field $F$, we will apply the notations and results from \cite{KaTa}  on unitary representations of locally compact groups, again without extensive elaboration. 
 \item Let $\psi$ be a non-trivial, continuous, and  unitary character of $F$. The Stone-von Neumann theorem asserts that only one unitary, irreducible, and complex representation of $\Ha(W)$ with central character $\psi$ exists, up to unitary equivalence. We refer to this representation as the Heisenberg representation, denoted  by $\pi_{\psi}$.  If $F$ is a non-archimedean local field, $\pi_{\psi}$ will  be a smooth representation.  
 \item Let  $\Sp(W)$ denote the corresponding symplectic group.  By Weil's seminal work,  $\pi_{\psi}$ can give rise  to a projective representation of $\Sp(W)$ and, consequently, to an actual representation of a $\C^{\times}$-covering group over $\Sp(W)$. It is customary to refer to such a group as a Metaplectic group and the actual representation as a Weil representation.  By  Perrin and Rao (cf. \cite{Per, Rao}),  this special  central cover can descend  to an  $8$-degree or a $2$-degree cover over $\Sp(W)$.  Let $c_{PR,X^{\ast}}(-,-)$ denote a  correspondence $8$-degree cocycle, called a Perrin-Rao cocycle.
 \item Let $\Z$ and  $\R$  denote the sets of integer numbers and real numbers, respectively. Let $\F_2$ denote the field of $2$ elements. Furthermore, let $\Q_2$ denote the field of 2-adic numbers. Let $T=\{ e^{it\pi}\}$ be the circle group.
\end{itemize}
\subsection{Main result}
If $F$ is a  non-archimedean local field,  a set $L$ is a \emph{lattice} in $W$ if  it is a free $\mathcal{O}$-module of full rank in $W$.  If $F=\R$, a set $L$ is a \emph{lattice} in $W$ if  it is a free $\Z$-module of full rank in $W$.  Let us define its dual lattice with respect to $\psi$ as follows:
$$L^{\ast}=\{ w\in W\mid \psi(\langle w, l\rangle)=1, \textrm{ for all } l\in L\}.$$
If  $L=L^{\ast}$,  $L$  is called a self-dual lattice with respect to $\psi$.  Let $\Ha(L)=L\times F$, $\Ha(L^{\ast})=L^{\ast}\times F$,  denote  the correspondence Heisenberg  subgroups of $\Ha(W)$. It is well-known that when $L$ is a self-dual lattice, one can  use  the induced operator from  $\Ha(L)$ to $\Ha(W)$  to construct the Heisenberg representation of $\Ha(W)$. Furthermore, this representation can be extended by incorporating the Metaplectic group, yielding the widely recognized Weil representation. (See \cite{AuPr, Bl, Kud2,  LiVe, MVW, Mu} for the details.)  In the present study, we consider the case that $L$ is only a lattice without the self-dual condition. As we found that  when  $F$ is  the real field,  P. Cartier  has already  constructed  Heisenberg representations of $\Ha(W)$ in families by utilizing a non-self-dual lattice in \cite{Ca}.  As such,   Jae-Hyun Yang generalizes them  to   a matrix Heisenberg group in his paper  \cite{Ya}. Motivated by \cite{Ho, MVW,  Ta, Wa}, we generalize Cartier and Yang's results by adding the symplectic group $\Sp(W)$ or its  subgroup. 

  \subsubsection{} In the case of a non-archimedean local field with odd residual characteristic, our result can be summarized as follows:
\begin{itemize}
\item[(1)] Let $\psi_L$ be an extended representation of $\Ha(L )$ from $\psi$ of $F$.  Let $(\sigma=\cInd_{\Ha(L )}^{\Ha(L^{\ast})}\psi_{L }$, $\mathcal{W}=\cInd_{\Ha(L )}^{\Ha(L^{\ast})}\C)$, $\pi_{L ,\psi}=\cInd_{\Ha(L )}^{\Ha(W)} \psi_{L }$, and   $\mathcal{V}_{L , \psi}=\cInd_{\Ha(L^{\ast})}^{\Ha(W)} \mathcal{W}$.
\item[(2)] The representation $\pi_{L ,\psi}$ of $\Ha(W)$ can be realized on $\mathcal{V}_{L , \psi}$. 
\item[(3)]  For $g\in \Sp(W)$, we define 
$$M[g]f(w)=\int_{a\in L^{\ast}} \psi(\tfrac{1}{2}\langle a, w\rangle) \sigma(-a)f((a+w)g) da, \quad  f\in \mathcal{V}_{L , \psi}.$$
\item[(4)] Then there exists a cocycle $c$ such that $M[g]M[g'] =c(g, g')M[gg']$.
Let $1\longrightarrow T \longrightarrow \Mp(W) \longrightarrow \Sp(W) \longrightarrow 1$
be the central extension of $\Sp(W)$ by  $T$  associated to $c(-,-)$.   
\end{itemize}
\begin{theorem*}
 For $[g,t]\in \Mp(W)$, $h\in \Ha(W)$, $f\in \mathcal{V}_{L , \psi}$, let 
$$\pi_{L ,\psi}([(g,t),h])f=tM[g]\pi_{L ,\psi}(h) f.$$
Then $\pi_{L ,\psi}$ defines a representation of $\Mp(W)\ltimes \Ha(W)$.  The restriction of $\pi_{L ,\psi}$  to $\Ha(W)$ contains $\sqrt{|L^{\ast}/L |}$-number of irreducible components and every  component is a Heisenberg representation associated to $\psi$. 
\end{theorem*}

\subsubsection{} In the case of $F$ being  a finite unramified extension of the field $\Q_2$ of $2$-adic numbers, we obtain the  analogous results for certain special lattices 
$L$ within $W$ (see  Theorem \ref{theoremnon2}). 
 However,  the situation is a bit complicated for  the definition of $M[g]$.  In the following, we explain these two points.
\paragraph{\emph{Special lattice}} Let $\psi$ be a fixed non-trivial character of $F$ of order $e$.    We  let $L_1=\mathcal{P}^{[\tfrac{e+1}{2}]}e_1 \oplus\cdots \oplus\mathcal{P}^{[\tfrac{e+1}{2}]}e_m  \oplus \mathcal{P}^{[\tfrac{e}{2}]}e_1^{\ast} \oplus \cdots \oplus  \mathcal{P}^{[\tfrac{e}{2}]}e_m^{\ast} $ be a  fixed self-dual lattice of $W$ with respect to $\psi$. We consider a special sublattice $L$ of $L_1$.
Let $d_L=\begin{bmatrix}
g_1& 0\\
0& g_2\end{bmatrix}\in \GSp_{2m}(F)$, such that $(L_1\cap X)g_1 \subseteq L_1\cap X$, $(L_1\cap X^{\ast})g_2 \subseteq L_1\cap X^{\ast}$. Let $L=L_1d_L$.
\paragraph{\emph{The definition of $M[g]$}} 
 Let $d=\diag(d_1, \cdots, d_m; \tfrac{1}{d_1}, \cdots, \tfrac{1}{d_m})$. Let $L''=Ld$. Similar to  the above (1)(2),  we define:
\begin{itemize}
\item $(\sigma=\cInd_{\Ha(L)}^{\Ha(L^{\ast})}\psi_{L}, \mathcal{W}=\cInd_{\Ha(L)}^{\Ha(L^{\ast})}\C)$.
\item $(\sigma''=\cInd_{\Ha(L'')}^{\Ha(L^{''\ast})}\psi_{L''}, \mathcal{W}''=\cInd_{\Ha(L'')}^{\Ha(L^{''\ast})}\C)$.
\item $(\pi_{L ,\psi}=\cInd_{\Ha(L^{\ast})}^{\Ha(W)} \sigma, \mathcal{V}_{L,\psi}= \cInd_{\Ha(L^{\ast})}^{\Ha(W)} \mathcal{W})$.
\item $(\pi_{L'' ,\psi}=\cInd_{\Ha(L^{''\ast})}^{\Ha(W)} \sigma'', \mathcal{V}_{L'',\psi}= \cInd_{\Ha(L^{''\ast})}^{\Ha(W)} \mathcal{W}'')$.
\item $\mathcal{D}:  \mathcal{W}\to  \mathcal{W}''; f \longmapsto \mathcal{D}(f)$, where $\mathcal{D}(f)(l^{''\ast})=f(l^{''\ast} d^{-1})$.
\end{itemize}
For  $g\in \Sp(W)$, we define a $\C$-linear map  $ \mathfrak{i}^g_{L,L''}$ from $\mathcal{V}_{L,\psi}$ to $ \mathcal{V}_{L'',\psi}$ as follows:
\begin{align}\label{interLL''}
\mathfrak{i}^g_{L,L''}(f)(w)&=\int_{a\in L^{''\ast}}\psi(\tfrac{1}{2}\langle a, w\rangle) \sigma''(-a)\mathcal{D}[f((a+w)g)] da.
\end{align}
Under the basis $\{e_1, \cdots, e_m; e_1^{\ast}, \cdots, e_m^{\ast}\}$, we identity $\Sp(W)$ with $\Sp_{2m}(F)$. Let $I$ be the Iwahori subgroup of $\Sp_{2m}(F)$ as given in \cite{HeOi} or \cite{TaWo}. For a subset $S\subseteq \{1, \cdots, m\}$, let us define an element   $\omega_S$ of  $ \Sp(W)$ as follows: $ (e_i)\omega_S=\left\{\begin{array}{lr}
-e_i^{\ast}& i\in S,\\
 e_i & i\notin S,
 \end{array}\right.$ and $ (e_i^{\ast})\omega_S=\left\{\begin{array}{lr}
e_i^{\ast}& i\notin S,\\
 e_i & i\in S.
 \end{array}\right.$ Let $S_m$ denote the permutation group on $m$ elements. For $s\in S_m$, let $\omega_s$ be an element of $\Sp(W)$ defined as $(e_i)\omega_s=e_{s(i)}$ and $(e_j^{\ast}) \omega_s=e_{s(j)}^{\ast}$. Let  $\widetilde{\mathfrak{W}}=\{ \omega_s \omega_S^{k} \mid S\subseteq  \{1, \cdots, m\}, s\in S_m\}$ and $\mathfrak{W}=\{  \omega_s \omega_S \mid S\subseteq  \{1, \cdots, m\}, s\in S_m\}$. Then  $\widetilde{\mathfrak{W}}/[\widetilde{\mathfrak{W}} \cap T(F)]$ is  a Weyl group of $\Sp(W)$ and  $\mathfrak{W}$ is a  representative for  this  Weyl group.  Let 
 $$D=\{ \begin{bmatrix}
d& 0 \\0 &  d^{-1} \end{bmatrix} \mid d = \diag(2^{k_1}, \cdots, 2^{k_m}) \textrm{ with integers } k_1, \cdots,   k_m \}.$$
Let $\widetilde{\mathfrak{W}^{\textrm{eaff}}}=D\rtimes \widetilde{\mathfrak{W}}$ and $\mathfrak{W}^{\textrm{eaff}}=D\rtimes \mathfrak{W}$. By \cite{Vi}, 
$ \mathfrak{W}^{\textrm{eaff}}$ gives a representative for  the Iwahori-Weyl group of $\Sp_{2m}(F)$ and 
$\Sp_{2m}(F)= \sqcup_{w\in \mathfrak{W}^{\textrm{eaff}}} IwI$. Then $\Sp(W)=\sqcup_{\omega\in \mathfrak{W}^{\textrm{eaff}}} [d_L^{-1}Id_L] \cdot [d_L^{-1}\omega d_L]\cdot [d_L^{-1} Id_L]$.

\begin{itemize}
\item For $g\in d_L^{-1}Id_L$,  let 
 \begin{itemize}
\item $L'=\left\{ \begin{array}{cl}
2 L\cap X\oplus \tfrac{1}{2} L\cap X^{\ast} & \textrm{ if } 2\mid e,\\
\tfrac{1}{2} L\cap X\oplus 2 L\cap X^{\ast} & \textrm{ if } 2\nmid e.\end{array}\right.$ 
\end{itemize} We define $M[g]\in \End(\mathcal{V}_{L,\psi})$ as follows:
\begin{align}
M[g]=  \mathfrak{i}_{L',L} \circ\mathfrak{i}^g_{L,L'} .
\end{align}
\item For $g\in d_L^{-1}\mathfrak{W}^{\textrm{eaff}}d_L$, we define a set $\mathcal{F}_g$ and a lattice  $L''_g$(cf. Section \ref{thegeneralcase}).   Similarly,  we define $M[g]\in \End(\mathcal{V}_{L,\psi})$ as follows:
$$ M[g]=\left\{ \begin{array}{cc} 
\mathfrak{i}^{g}_{L,L} & \textrm{ if }  \mathcal{F}_g=\emptyset,\\
  \mathfrak{i}_{L''_g,L} \circ\mathfrak{i}^g_{L,L''_g}   & \textrm{ if }  \mathcal{F}_g\neq \emptyset.
  \end{array}\right.$$
 \item  For an element $g\in \Sp(W)$, we fix a decomposition:
$$g=i_1\omega_g i_2$$
for $i_1, i_2\in [d_L^{-1}Id_L]$ and $\omega_g\in [d_L^{-1}\mathfrak{W}^{\textrm{eaff}}d_L]$.   Define: $M[g]=M[i_1]M[\omega_g]M[i_2]$.
\end{itemize} 

\subsubsection{} In the case of $F$ being  the real field, we get the same result (see Theorem \ref{main3}) for the rational symplectic group. The methodology applied here closely follows that of the preceding case. For brevity, we only explain the lattice that was chosen in this case. 
Let $\kappa=1$ or $-1$. We consider   $\psi$ to be  a  character of $F$ defined as: $ t \longmapsto e^{2\kappa\pi it}$, for $t\in F$.  Let $L_1=\Za  e_1 \oplus  \cdots \oplus \Za  e_m \oplus \Za  e_1^{\ast} \oplus \cdots \oplus \Za e_m^{\ast}$.  Analogous to  the non-archimedean case, we also  consider a special sublattice $L$ of $L_1$.   Let $d_L=\begin{bmatrix}
g_1& 0\\
0& g_2\end{bmatrix}\in \GSp_{2m}(\Q)$, such that $(L_1\cap X)g_1 \subseteq L_1\cap X$, $(L_1\cap X^{\ast})g_2 \subseteq L_1\cap X^{\ast}$. Let $L=L_1d_L$.

\section{The non-archimedean local field case I }\label{pn1}
In this section, we consider $F$ to be a non-archimedean local field of odd residue characteristic. We will let $\psi$ be a fixed non-trivial character of $F$ with order $e$, meaning that $\psi|_{\mathcal{P}^e}=1$, and  $\psi|_{\mathcal{P}^{e-1}}\neq 1$. 
\subsection{Self-dual lattice}
In this subsection, we will let $L_1$ be a self-dual lattice of $W$ with respect to $\psi$. Let $\chi\in \Irr(L_1)$. Let $\psi_{L_1, \chi}$ denote the extended  character of $\Ha(L_1)$  from $\psi$ of  $F$ defined as:
$$\psi_{L_1, \chi}: \Ha(L_1) \longrightarrow \C^{\times}; (l,t) \longmapsto \psi(t)\chi(l),$$ 
for $l\in L_1$. 
Let us define the representation:
 $$(\pi_{L_1,\psi_{\chi}}=\cInd_{\Ha(L_1)}^{\Ha(W)} \psi_{L_1,\chi}, \quad  \mathcal{S}_{L_1,\psi_{\chi}}=\cInd_{\Ha(L_1)}^{\Ha(W)} \C).$$
\begin{lemma}
$\pi_{L_1,\psi_{\chi}}$ defines a Heisenberg representation of $\Ha(W)$ associated to  the center character $\psi$.
\end{lemma} 
\begin{proof}
See \cite[p.42, II.8]{MVW}.
\end{proof}
The space $\mathcal{S}_{L_1,\psi_{\chi}}$ consists  of locally constant, compactly supported functions $f$ on $W$  such that
 $$f(l+w)=\psi(-\tfrac{1}{2}\langle l, w\rangle)\chi(l)f(w),$$
  for $l\in L_1$, $w\in W$.  For $h=(w',t)\in \Ha(W)$,
\begin{equation}\label{BB'}
\pi_{L_1,\psi_{\chi}}(h) f(w)=f([w,0]\cdot [w',t])=f([w+w', t+\tfrac{1}{2}\langle w, w'\rangle])=\psi( t+\tfrac{1}{2}\langle w, w'\rangle) f(w+w').
\end{equation}
\begin{remark}
If $\chi$ is the trivial character, we will  write $\pi_{L_1,\psi}$.  Then $\pi_{L_1,\psi_{\chi}}\simeq \pi_{L_1,\psi}$, for any $\chi$.
\end{remark}
\subsection{Non-self-dual lattice}
Let us consider a non-self-dual lattice.  Keep the above notations. Let $L $ be a sublattice of $L_1$.  Write  $[L_1:L ]=q^r$, for    some $r$. Let $\Ha(L )=L \times F$ denote the corresponding subgroup of $\Ha(W)$.  Let $\psi_{L }$ denote the extended  character of $\Ha(L )$  from $\psi$ of  $F$ defined as:
$$\psi_{L }: \Ha(L ) \longrightarrow \C^{\times}; (l,t) \longmapsto \psi(t).$$ 
   Let us define the representation:
 $$(\pi_{L ,\psi}=\cInd_{\Ha(L )}^{\Ha(W)} \psi_{L }, \quad  \mathcal{S}_{L ,\psi}=\cInd_{\Ha(L )}^{\Ha(W)} \C).$$
 Note that $\Ha(L ) \subseteq \Ha(L_1)$.  Let $\chi$ be a character of $L_1/L $. Then $\psi_{L_1, \chi}|_{\Ha(L )}= \psi_{L }$. 
By Clifford theory, $\cInd_{\Ha(L ) }^{\Ha(L_1) } \psi_{L }\simeq  \oplus_{\chi \in \Irr(L_1/L )} \psi_{L_1, \chi}$. As a consequence, we obtain:
\begin{lemma}
$\pi_{L, \psi}\simeq q^r\pi_{L_1, \psi}.$
\end{lemma}
\begin{proof}
$\pi_{L,\psi}=\cInd_{\Ha(L)}^{\Ha(W)} \psi_{L}\simeq \cInd_{\Ha(L_1)}^{\Ha(W)}  (\cInd_{\Ha(L) }^{\Ha(L_1) } \psi_{L})$
$\simeq \cInd_{\Ha(L_1)}^{\Ha(W)}( \oplus_{\chi \in \Irr(L_1/L)} \psi_{L_1, \chi}) \simeq \oplus_{\chi \in \Irr(L_1/L)}  \pi_{L_1, \psi_{\chi}}$
$\simeq q^r\pi_{L_1, \psi}$.
\end{proof}
The space $\mathcal{S}_{L, \psi}$ consists  of locally constant, compactly supported functions $f$ on $W$  such that
 $$f(l+w)=\psi(-\tfrac{1}{2}\langle l, w\rangle)f(w),$$
  for $l\in L$, $w\in W$.  For $h=(w',t)\in \Ha(W)$,
\begin{equation}\label{BB'1}
\pi_{L,\psi}(h) f(w)=\psi( t+\tfrac{1}{2}\langle w, w'\rangle) f(w+w').
\end{equation}
Let $\mathcal{S}_{L_1,\psi_{\chi}}$  denote the subspace of the elements $f$ in $\mathcal{S}_{L,\psi}$ such that 
$$f(l+w)=\psi(-\tfrac{1}{2}\langle l, w\rangle)\chi(l)f(w),$$
for  $l\in L_1$.  Let $\Lambda_{L_1}$( resp. $\Lambda_{L_1/L}$) be a set of representatives for $W/L_1$( resp. $L_1/L$). Then 
$$\Lambda_{L}=\{ w_0+l_0\mid w_0\in  \Lambda_{L_1}, l_0\in  \Lambda_{L_1/L}\}$$
serves as a representative set for  $W/L$. For an element   $w_0\in\Lambda_{L_1}$, and an element $w_1=w_0+l_0\in \Lambda_{L}$,  we let  $s^{\chi}_{w_0}$ and  $s_{w_1}$  be the functions on $W$ such that 
\begin{itemize}
\item $\supp s^{\chi}_{w_0} \subseteq L_1+w_0$,
\item  $\supp s_{w_1} \subseteq L+w_1$,
\item $s^{\chi}_{w_0}(l+w_0)=\psi(-\tfrac{1}{2}\langle l, w_0\rangle)\chi(l)$, for $l\in L_1$, 
\item $s_{w_1} (l+w_1)=\psi(-\tfrac{1}{2}\langle l, w_1\rangle)$, for $l\in L$.
\end{itemize}
Then $\{ s_{w_1}\}_{w_1\in \Lambda_{L}}$ and $\{ s^{\chi}_{w_0}\}_{w_0\in \Lambda_{L_1}}$ serve as the  bases of $\mathcal{S}_{L, \psi}$ and  $\mathcal{S}_{L_1,\psi_{\chi}}$, respectively.  Let $V_{L_1, w_0}$ be the subspace of $\mathcal{S}_{L, \psi}$, spanned by $\{s_{w_1}, w_1\in w_0+\Lambda_{L_1/L}\}$. For any $w_1=w_0+l_0\in w_0+ \Lambda_{L_1/L}$, $l\in L$, we have:
$$s^{\chi}_{w_0}(w_1)=s^{\chi}_{w_0}(l_0+w_0)=\psi(-\tfrac{1}{2}\langle l_0, w_0\rangle)\chi(l_0),$$
 \begin{align}
 &s^{\chi}_{w_0}(l+w_1)\\
 &=s^{\chi}_{w_0}(l+l_0+w_0)\\
 &=\psi(-\tfrac{1}{2}\langle l+l_0, w_0\rangle)\chi(l+l_0)\\
 &=\psi(-\tfrac{1}{2}\langle l_0, w_0\rangle)\chi(l_0)\psi(-\tfrac{1}{2}\langle l, w_0\rangle) \\
 &=\psi(-\tfrac{1}{2}\langle l, w_0\rangle)s^{\chi}_{w_0}(w_1)\\
 &=\psi(-\tfrac{1}{2}\langle l, w_1\rangle)s^{\chi}_{w_0}(w_1)\\
   &=s_{w_1}(l+w_1) s^{\chi}_{w_0}(w_1).
 \end{align}
   Hence 
 $$s^{\chi}_{w_0}=\sum_{w_1\in w_0+ \Lambda_{L_1/L}} s^{\chi }_{w_0}(w_1) \cdot s_{w_1}\in V_{L_1, w_0}.$$
 Moreover, these $s^{\chi}_{w_0}$ are linear independence as $\chi$ runs through the characters of $L_1/L$. From dimension considerations it follows that $V_{L_1, w_0}$ is spanned by $\{s^{\chi}_{w_0}, \chi \in \Irr(L_1/L)\}$. Hence:
 $$ \mathcal{S}_{L,\psi}=\oplus_{\chi\in \Irr(L_1/L)} \mathcal{S}_{L_1,\psi_{\chi}}.$$
 Note that $\mathcal{S}_{L_1,\psi_{\chi}}$  is an irreducible $\Ha(W)$-module. 
      
  \subsubsection{} 
  For any $\chi \in \Irr(L_1/L)$,  $ (\pi_{L_1,\psi_{\chi}}, \mathcal{S}_{L_1,\psi_{\chi}}) $ is an irreducible representation of $\Ha(W)$. 
 Hence 
 $$(\sigma_{\chi}=\cInd_{\Ha(L_1)}^{\Ha(L^{\ast})} \psi_{L_1,\chi}, \mathcal{W}_{\chi}=\cInd_{\Ha(L_1)}^{\Ha(L^{\ast})} \C)$$ is an irreducible representation of $\Ha(L^{\ast})$.  The vector space $\mathcal{W}_{\chi}$ consists of the  functions $f$ on $L^{\ast}$ such that 
 $$f(l+l^{\ast})= \chi(l) \psi(-\tfrac{1}{2} \langle l, l^{\ast}\rangle)f(l^{\ast}),$$
 for $l\in L_1$, $l^{\ast}\in L^{\ast}$. 
 \begin{lemma}
 $\sigma_{\chi}(l)f(l^{\ast})= \chi(l) \psi(- \langle l, l^{\ast}\rangle)f(l^{\ast})$, for $l\in L_1$.
 \end{lemma}
 \begin{proof}
 $\sigma_{\chi}(l)f(l^{\ast})=f(l+l^{\ast})\psi(- \tfrac{1}{2}\langle l, l^{\ast}\rangle)=\chi(l) \psi(- \langle l, l^{\ast}\rangle)f(l^{\ast}).$
 \end{proof}

Let $$\sigma=\cInd_{\Ha(L)}^{\Ha(L^{\ast})}\psi_{L},  \mathcal{W}=\cInd_{\Ha(L)}^{\Ha(L^{\ast})}\C.$$
The vector space $\mathcal{W}$ consists of the functions $f:L^{\ast}/L\longrightarrow \C$.  The action is given as follows:
$$[\sigma(l^{\ast}) f](\dot{l}^{\ast}_1)= \psi(\tfrac{1}{2} \langle l_1^{\ast}, l^{\ast}\rangle) f(\dot{l^{\ast}}+ \dot{l}^{\ast}_1).$$
\begin{lemma}
For $l\in L$,  $[\sigma(l) f](\dot{l}^{\ast})= f(\dot{l}^{\ast})$. 
\end{lemma}
\begin{proof}
 $[\sigma(l) f](\dot{l}^{\ast})=\psi(\tfrac{1}{2} \langle l^{\ast},l\rangle) f(\dot{l}^{\ast})= f(\dot{l}^{\ast})$.
\end{proof}
We can identity  $(\sigma_{\chi}, \mathcal{W}_{\chi})$ as a sub-representation of $(\sigma, \mathcal{W})$ such that $ \mathcal{W}_{\chi}$ consists of the elements 
$f$ satisfying  $f(\dot{l}+\dot{l}^{\ast})= \chi(l) \psi(-\tfrac{1}{2} \langle l, l^{\ast}\rangle)f(\dot{l}^{\ast})$ for $l\in L_1$.
\subsubsection{}
 Note that $$\psi: L_1/L \times L^{\ast}/L_1 \longrightarrow T; (x,y^{\ast}) \longrightarrow \psi(\langle x, y^{\ast}\rangle ),$$
 defines a non-degenerate bilinear map. So for any $\chi \in \Irr(L_1/L)$, there exists an element $y^{\ast}_{\chi} \in L^{\ast}/L_1 $ such that 
 $$\psi(\langle  x, y_{\chi}^{\ast}\rangle )=\chi(x), \quad\quad \quad x\in L_1/L.$$
 Let us define a function:
  $$\mathcal{A}_{\chi_1, \chi_2}: \mathcal{W}_{\chi_1} \longrightarrow \mathcal{W}_{\chi_2},$$ 
  given by 
 \begin{align}
 \mathcal{A}_{\chi_1, \chi_2}(f)(l^{\ast})=f(l^{\ast}+y^{\ast}_{\chi_1}-y^{\ast}_{\chi_2}) \psi(-\tfrac{1}{2}\langle l^{\ast}, y^{\ast}_{\chi_1}-y^{\ast}_{\chi_2}\rangle).
 \end{align}
  For $f\in  \mathcal{W}_{\chi_1}$, $l\in L_1$, we have:
 \begin{align}
 &\mathcal{A}_{\chi_1, \chi_2}(f)(l+l^{\ast})\\
 &=f(l+l^{\ast}+y^{\ast}_{\chi_1}-y^{\ast}_{\chi_2})\psi(-\tfrac{1}{2}\langle l+l^{\ast}, y^{\ast}_{\chi_1}-y^{\ast}_{\chi_2}\rangle)\\
 &=\chi_1(l) \psi(-\tfrac{1}{2} \langle l, l^{\ast}+y^{\ast}_{\chi_1}-y^{\ast}_{\chi_2}\rangle)f(l^{\ast}+y^{\ast}_{\chi_1}-y^{\ast}_{\chi_2})\psi(-\tfrac{1}{2}\langle l+l^{\ast}, y^{\ast}_{\chi_1}-y^{\ast}_{\chi_2}\rangle)\\
 &=\chi_2(l) \psi(-\tfrac{1}{2} \langle l, l^{\ast}\rangle)\mathcal{A}_{\chi_1, \chi_2}(f)(l^{\ast}).
 \end{align}
 Hence $ \mathcal{A}_{\chi_1, \chi_2}(f) \in\mathcal{W}_{\chi_2}$.  So $\mathcal{A}_{\chi_1, \chi_2}$ is well-defined.
 \begin{lemma}\label{iter}
 $\mathcal{A}_{\chi_1, \chi_2}$ defines an intertwining operator from $\sigma_{\chi_1}$ to $\sigma_{\chi_2}$.
 \end{lemma}
 \begin{proof}
For $f \in \mathcal{W}_{\chi_1}$, $l^{\ast}\in L^{\ast}$,  $h=[l_1^{\ast}, t]\in \Ha(L^{\ast})$, we have:
\begin{align}
\mathcal{A}_{\chi_1, \chi_2}[\sigma_{\chi_1}(h) f](l^{\ast})&=[\sigma_{\chi_1}(h) f](l^{\ast}+y^{\ast}_{\chi_1}-y^{\ast}_{\chi_2}) \psi(-\tfrac{1}{2}\langle l^{\ast}, y^{\ast}_{\chi_1}-y^{\ast}_{\chi_2}\rangle)\\
&= f(l_1^{\ast}+l^{\ast}+y^{\ast}_{\chi_1}-y^{\ast}_{\chi_2}) \psi( t+\tfrac{1}{2}\langle l^{\ast}+y^{\ast}_{\chi_1}-y^{\ast}_{\chi_2}, l_1^{\ast}\rangle) \psi(-\tfrac{1}{2}\langle l^{\ast}, y^{\ast}_{\chi_1}-y^{\ast}_{\chi_2}\rangle)\\
&=f(l_1^{\ast}+l^{\ast}+y^{\ast}_{\chi_1}-y^{\ast}_{\chi_2}) \psi( t+\tfrac{1}{2}\langle l^{\ast}, l_1^{\ast}\rangle) \psi(-\tfrac{1}{2}\langle l^{\ast}+l_1^{\ast}, y^{\ast}_{\chi_1}-y^{\ast}_{\chi_2}\rangle);
\end{align}
\begin{align}
\sigma_{\chi_2}(h) [\mathcal{A}_{\chi_1, \chi_2}f](l^{\ast})&= \psi( t+\tfrac{1}{2}\langle l^{\ast}, l_1^{\ast}\rangle)[\mathcal{A}_{\chi_1, \chi_2}f](l^{\ast}+l_1^{\ast})\\
&= \psi( t+\tfrac{1}{2}\langle l^{\ast}, l_1^{\ast}\rangle)f(l_1^{\ast}+l^{\ast}+y^{\ast}_{\chi_1}-y^{\ast}_{\chi_2})\psi(-\tfrac{1}{2}\langle l^{\ast}+l_1^{\ast}, y^{\ast}_{\chi_1}-y^{\ast}_{\chi_2}\rangle).
\end{align}
 \end{proof} 

 \subsubsection{}Let us denote  $\mathcal{V}_{L,\psi}= \cInd_{\Ha(L^{\ast})}^{\Ha(W)} \mathcal{W}$. Then $\pi_{L,\psi}$ can be realized on $\mathcal{V}_{L,\psi}$. Note that $L^{\ast}$ is a compact group and $\dim \mathcal{W} <+\infty$. So the vector space  $\mathcal{V}_{L,\psi}$ consists  of locally constant, compactly supported functions $f: W \longrightarrow \mathcal{W}$  such that
 $$f(l^{\ast}_1+w)=\psi(\tfrac{1}{2}\langle w, l_1^{\ast}\rangle)\sigma(l_1^{\ast})f(w),$$
  for $l^{\ast}_1\in L^{\ast}$,  $w\in W$.   For $h=(w',t)\in \Ha(W)$,
\begin{equation}\label{BB'2}
\pi_{L,\psi}(h) f(w)=f([w,0]\cdot [w',t])=f([w+w', t+\tfrac{1}{2}\langle w, w'\rangle])=\psi( t+\tfrac{1}{2}\langle w, w'\rangle) f(w+w').
\end{equation}
 Let us denote  $\mathcal{V}_{L_1 ,\psi_{\chi}}= \cInd_{\Ha(L^{\ast} )}^{\Ha(W)} \mathcal{W}_{\chi}$. This vector space  consists  of locally constant, compactly supported functions $f: W \longrightarrow \mathcal{W}_{\chi}$  such that
 $$f(l^{\ast}+w)=\psi(\tfrac{1}{2}\langle w, l^{\ast}\rangle)\sigma_{\chi}(l^{\ast})f(w),$$
  for $l^{\ast}\in L^{\ast}$,  $w\in W$.  Moreover, $\pi_{L_1,\psi_{\chi}}$ can be realized on $\mathcal{V}_{L_1 ,\psi_{\chi}}$.  Consequently, 
  $$\mathcal{V}_{L ,\psi} \simeq \oplus_{\chi \in \Irr(L_1/L )} \mathcal{V}_{L_1 ,\psi_{\chi}}.$$

\subsubsection{}\label{MgM} 
Following \cite[p.42]{MVW}, for $g\in \Sp(W)$, let us define $M[g]\in \End(\mathcal{V}_{L,\psi})$ as follows:
\begin{align}\label{gactionsigma}
M[g]f(w)=\int_{a\in L^{\ast}} \psi(\tfrac{1}{2}\langle a, w\rangle) \sigma(-a)f((a+w)g) da.
\end{align}
\begin{itemize}
\item[(1)] Let $l\in L^{\ast}$, $f\in   \mathcal{V}_{L,\psi}$. 
\begin{align}\label{gactionsigma1}
M[g]f(l+w)&=\int_{a\in L^{\ast}} \psi(\tfrac{1}{2}\langle a, l+w\rangle) \sigma(-a)f((a+l+w)g) da\\
&=\int_{a\in L^{\ast}} \psi(\tfrac{1}{2}\langle a-l, l+w\rangle) \sigma(-a+l)f((a+w)g) da\\
&= \psi(\tfrac{1}{2}\langle -l, w\rangle)\sigma(l)\int_{a\in L^{\ast}} \psi(\tfrac{1}{2}\langle a,w\rangle) \sigma(-a)f((a+w)g) da\\
&=\psi(\tfrac{1}{2}\langle w, l\rangle)\sigma(l)M[g]f(w).
\end{align}
Hence $M[g]f \in \mathcal{V}_{L,\psi}$, for any  $f\in \mathcal{V}_{L,\psi} $.  This means that $M[g]$ is well-defined.
\end{itemize}
\begin{itemize}
\item[(2)] For $h=(w',t)\in \Ha(W)$,  $hg^{-1}=(w'g^{-1},t)$.
\begin{align*}
&M[g] \pi_{L, \psi}(h) f(w)\\
&=\int_{a\in L^{\ast}} \psi(\tfrac{1}{2}\langle a, w\rangle) \sigma(-a) [\pi_{L, \psi}(h) f]((a+w)g) da\\
&=\int_{a\in L^{\ast}} \psi(\tfrac{1}{2}\langle a, w\rangle) \psi( t+\tfrac{1}{2}\langle (a+w)g, w'\rangle) \sigma(-a) f((a+w)g+w')da;
\end{align*}
\begin{align*}
&\pi_{L, \psi}(hg^{-1}) (M[g]f)(w)\\
&=\psi( t+\tfrac{1}{2}\langle w, w'g^{-1}\rangle) (M[g]f)(w+w'g^{-1})\\
 &=\int_{a\in L^{\ast}} \psi(\tfrac{1}{2}\langle a, w+w'g^{-1}\rangle)\psi( t+\tfrac{1}{2}\langle w, w'g^{-1}\rangle) \sigma(-a) f((a+w+w'g^{-1})g)da\\
 &=\int_{a\in L^{\ast}} \psi(\tfrac{1}{2}\langle a, w\rangle)\psi( t+\tfrac{1}{2}\langle w+a, w'g^{-1}\rangle) \sigma(-a)  f((a+w)g+w')da.
\end{align*}
Hence $\pi_{L, \psi}(hg^{-1}) M[g]= M[g] \pi_{L, \psi}(h) $, for  $g\in \Sp(W)$, $h\in \Ha(W)$.
\end{itemize}
\begin{itemize}
\item[(3)] 
\begin{align}\label{muti}
&M[g]M[g^{-1}](f)(w)\\
&=\int_{a'\in L^{\ast}}\psi(\tfrac{1}{2}\langle a', w\rangle)\sigma(-a')M[g^{-1}](f)((a'+w)g)da'\\
&=\int_{a'\in L^{\ast}}\int_{a \in L^{\ast}}\psi(\tfrac{1}{2}\langle a', w\rangle)\psi(\tfrac{1}{2}\langle a,(a'+w)g\rangle)\sigma(-a')\sigma(-a)f(ag^{-1}+a'+w)da' da\\
&\xlongequal[b=ag^{-1}]{ c_{g}=\tfrac{da}{dag^{-1}}}c_g\int_{a'\in L^{\ast}}\int_{b\in L^{\ast}g^{-1}}\psi(\tfrac{1}{2}\langle a', w\rangle)\psi(\tfrac{1}{2}\langle b,a'+w\rangle)\sigma(-a')\sigma(-bg)f(b+a'+w)da' db\\
&=c_g\int_{a'\in L^{\ast}}\int_{b\in L^{\ast}g^{-1}}\psi(\tfrac{1}{2}\langle a', w\rangle)\psi(\tfrac{1}{2}\langle b,a'+w\rangle)\psi(\tfrac{1}{2}\langle a', bg\rangle) \sigma(-a'-bg)f(b+a'+w)da' db\\
&=c_g\int_{a'\in L^{\ast}}\int_{b\in L^{\ast}g^{-1}}\psi(\tfrac{1}{2}\langle a', w\rangle)f(b-bg+w)\\
&\quad\quad\quad\quad\quad\quad \cdot \psi(\tfrac{1}{2}\langle b,a'+w\rangle)\psi(\tfrac{1}{2}\langle a', bg\rangle) \psi(\tfrac{1}{2}\langle -a'-bg,b+a'+w\rangle) da' db\\
&=c_g\int_{a'\in L^{\ast}}\int_{b\in L^{\ast}g^{-1}}f(b-bg+w)
 \psi(\langle -a', b-bg\rangle) \psi(\tfrac{1}{2}\langle b-bg, w\rangle)  \psi(\tfrac{1}{2}\langle-bg, b\rangle) da' db\\
 &=c_g\mu(L^{\ast})\int_{b\in L^{\ast}g^{-1}, b-bg\in L}f(b-bg+w)
  \psi(\tfrac{1}{2}\langle b-bg, w\rangle)  \psi(\tfrac{1}{2}\langle-bg, b-bg\rangle) db\\ 
 &=c_g\mu(L^{\ast})\mu(\{b\in L^{\ast}g^{-1}, b-bg\in L\}) f(w).
\end{align}
\end{itemize}
 Hence,  $M[g]$ defines an automorphism of $\mathcal{V}_{L ,\psi}$. 
 \subsubsection{}\label{ext}
 We  extend $\mathcal{A}_{\chi_1, \chi_2}$ from  $\sigma_{\chi_1}$  to $\pi_{L_1,\psi}$ as follows:
 $$\mathcal{A}_{\chi_1, \chi_2}: \cInd_{\Ha(L^{\ast} )}^{\Ha(W)} \sigma_{\chi_1} \longrightarrow \cInd_{\Ha(L^{\ast} )}^{\Ha(W)} \sigma_{\chi_2};$$
 $$\quad\quad\quad  f\longmapsto \mathcal{A}_{\chi_1, \chi_2}(f).$$
 \begin{lemma}\label{IntMg}
 $M[g]\mathcal{A}_{\chi_1, \chi_2}(f)=\mathcal{A}_{\chi_1, \chi_2}M[g](f)$, for $f\in \mathcal{W}_{L ,\psi_{\chi_1}}$.
 \end{lemma}
 \begin{proof}
 \begin{align}
M[g]\mathcal{A}_{\chi_1, \chi_2}(f)(w)&=\int_{a\in L^{\ast}} \psi(\tfrac{1}{2}\langle a, w\rangle) \sigma(-a)\mathcal{A}_{\chi_1, \chi_2}(f)((a+w)g) da\\
&=\int_{a\in L^{\ast}} \psi(\tfrac{1}{2}\langle a, w\rangle) \sigma(-a)\mathcal{A}_{\chi_1, \chi_2}(f((a+w)g)) da\\
&=\mathcal{A}_{\chi_1, \chi_2}\Big(\int_{a\in L^{\ast}} \psi(\tfrac{1}{2}\langle a, w\rangle) \sigma(-a)f((a+w)g) da\Big)\\
&=\mathcal{A}_{\chi_1, \chi_2}M[g](f)(w).
\end{align}
\end{proof}
According to  the above item (2), $M[g]$ sends  $\mathcal{V}_{L_1 ,\psi_{\chi}}$ onto itself. Let us define 
$$\pi_{L ,\psi}([g,h])=M[g]\pi_{L ,\psi}(h), \quad \quad g\in \Sp(W), h\in \Ha(W).$$
 Note that $$[g,h]=[g,0][1,h]=[1,hg^{-1}][g,0].$$  By the above (2), $\pi_{L,\psi}([g,h])$ is well-defined.
By Weil's result,  $\pi_{L ,\psi}$ defines a projective representation of $\Sp(W)\ltimes \Ha(W)$ on each $\mathcal{V}_{L_1 ,\psi_{\chi}}$.  For $[g,h]$,  $[g',h']\in \Sp(W)\ltimes \Ha(W)$, $f\in \mathcal{V}_{L_1 ,\psi_{\chi}}$, there exists a non-trivial $2$-cocycle $c_{\chi}$ of order $2$ in $\Ha^2(\Sp(W), T)$ such that 
\begin{align}
\pi_{L,\psi_{\chi}}([g,h])\pi_{L,\psi_{\chi}}([g',h'])=c_{\chi}(g, g')\pi_{L,\psi_{\chi}}([g,h][g',h']).
\end{align}
By the above lemma \ref{IntMg}, we can see that all the cocycles $c_{\chi}$ are the same. Let 
$$1\longrightarrow T \longrightarrow \Mp(W) \longrightarrow \Sp(W) \longrightarrow 1$$
be the central extension of $\Sp(W)$ by  $T$  associated to $c_{\chi}(-,-)$.  Note that $ \Mp(W)$ is isomorphic to the usual Metaplectic group associated to the Perrin-Rao's cocycle. Finally, we  conclude:
\begin{itemize}
\item[(1)] Let $\sigma=\cInd_{\Ha(L )}^{\Ha(L^{\ast})}\psi_{L }$, $\mathcal{W}=\cInd_{\Ha(L )}^{\Ha(L^{\ast})}\C$, $\mathcal{V}_{L , \psi}=\cInd_{\Ha(L^{\ast})}^{\Ha(W)} \mathcal{W}$.
\item[(2)] $\pi_{L ,\psi}=\cInd_{\Ha(L )}^{\Ha(W)} \psi_{L } \simeq \cInd_{\Ha(L^{\ast})}^{\Ha(W)} \sigma$. Then the representation $\pi_{L ,\psi}$ of $\Ha(W)$ can be realized on $\mathcal{V}_{L , \psi}$. 
\item[(3)]  For $g\in \Sp(W)$, we define 
$$M[g]f(w)=\int_{a\in L^{\ast}} \psi(\tfrac{1}{2}\langle a, w\rangle) \sigma(-a)f((a+w)g) da, \quad  f\in \mathcal{V}_{L , \psi}.$$
\item[(4)] For $[g,t]\in \Mp(W)$, $h\in \Ha(W)$, $f\in \mathcal{V}_{L , \psi}$, let 
$$\pi_{L ,\psi}([(g,t),h])f=tM[g]\pi_{L ,\psi}(h) f.$$
Then $\pi_{L ,\psi}$ defines a representation of $\Mp(W)\ltimes \Ha(W)$.  The restriction of $\pi_{L ,\psi}$  to $\Ha(W)$ contains $\sqrt{|L^{\ast}/L |}$-number of irreducible components and every  component is a Heisenberg representation associated to $\psi$. 
\end{itemize}
\begin{lemma}
If $L^{\ast} g \subseteq L^{\ast}$, then $M[g] f(w)=\int_{a\in L^{\ast}} \psi(\tfrac{1}{2}\langle -a,ag\rangle) \sigma(-a+ag) f(wg) da$.
\end{lemma}
\begin{proof}
$ag \in L^{\ast}$. So 
\begin{align}
M[g]f(w)&=\int_{a\in L^{\ast}} \psi(\tfrac{1}{2}\langle a, w\rangle) \sigma(-a)f((a+w)g) da\\
&=\int_{a\in L^{\ast}}  \sigma(-a)\sigma(ag) f(wg) da\\
&=\int_{a\in L^{\ast}} \psi(\tfrac{1}{2}\langle -a, ag\rangle) \sigma(-a+ag) f(wg) da.
\end{align}
\end{proof}
\subsubsection{} Following \cite[Chap.5]{MVW}, let 
$$J_1(L)=\{g\in \Sp(W) \mid L^{\ast} (g-1) \subseteq L\},$$
$$\Ha_1(L)=\{ g\in \Sp(W) \mid L^{\ast} (g-1) \subseteq L_1\}.$$
By  \cite[Chap.5]{MVW}, they are two groups, $  J_1(L) \unlhd \Ha_1(L)$, $\Ha_1(L)/ J_1(L)$ is an abelian group. Moreover, 
$$\Ha_1(L)=\{ g\in \Sp(W) \mid L_1 (g-1) \subseteq L\},$$
$$\Ha_1(L) \subseteq K(L_1)=\{ g\in \Sp(W)\mid gL_1\subseteq L_1\}.$$
Let $\overline{J_1(L)}$, $\overline{\Ha_1(L)}$  be the inverse images of $J_1(L)$ and $\Ha_1(L)$ in $\Mp(W)$, respectively.  Then the exact sequences 
$$1\longrightarrow T\longrightarrow \overline{J_1(L)} \longrightarrow J_1(L) \longrightarrow  1$$
$$1\longrightarrow T \longrightarrow \overline{\Ha_1(L)} \longrightarrow \Ha_1(L) \longrightarrow  1$$
 both are splitting. Recall that $\psi_{L_1,\chi}$ is a character of $\Ha(L_1)$. It can be extended to $\Ha_1(L)\ltimes \Ha(L_1)$ by defining 
$$\psi_{L_1,\chi}([g,(l,t)])= \psi(t) \chi(l).$$
Let us check that it is well-defined. Note:
$$[g,(l,t)]=[1, (lg^{-1},t)] \cdot [g, (0, 1)], $$
$$\psi_{L_1,\chi}([1, (lg^{-1},t)] )\psi_{L_1,\chi}([g, (0, 1)])=\psi(t)\chi(lg^{-1}),$$
$$\chi(lg^{-1}-l)=1.$$
Thus, we can treat  $(\sigma_{\chi}, \mathcal{W}_{\chi})$ as an irreducible representation of $\Ha_1(L) \ltimes \Ha(L^{\ast})$. 
 The action of an element  $g$ in $ \Ha_1(L)$  is given by
$$[\sigma(g) f](l^{\ast})=   f(l^{\ast}g).$$
\begin{lemma}\label{stabl}
 Assume $\mu(L^{\ast})=1$. For $g\in J_1(L)$, $f\in \mathcal{V}_{L_1, \chi}$, $M[g] f(w)=f(wg)$.
\end{lemma}
\begin{proof}
  For $g\in J_1(L)$,
\begin{align}
M[g]f(w)&=\int_{a\in L^{\ast}} \psi(\tfrac{1}{2}\langle -a,ag\rangle) \sigma(-a+ag) f(wg) da\\
&=\int_{a\in L^{\ast}}  f(wg) da=f(wg).
\end{align}
\end{proof}
Write 
$$(\pi'_{L_1,\psi_{\chi}}=\cInd_{\Ha_1(L) \ltimes \Ha(L^{\ast} )}^{\Ha_1(L) \ltimes \Ha(W)}\sigma_{\chi},   \mathcal{V}'_{L_1 ,\psi_{\chi}}= \cInd_{\Ha_1(L) \ltimes \Ha(L^{\ast} )}^{\Ha_1(L) \ltimes  \Ha(W)} \mathcal{W}_{\chi}).$$
Note that each $\pi'_{L_1,\psi_{\chi}}$ is an irreducible representation of $\Ha_1(L) \ltimes \Ha(W)$ and we can identity $\mathcal{V}'_{L_1 ,\psi_{\chi}}$ with $\mathcal{V}_{L_1 ,\psi_{\chi}}$. Note that $M[g]$  acts stably on each $\mathcal{V}_{L_1 ,\psi_{\chi}}$, for $g\in \Ha_1(L)$. So $\pi'_{L_1,\psi_{\chi}}(g)$ and $M[g]$ will differ by a constant on $\mathcal{V}_{L_1 ,\psi_{\chi}}$.  By Lemma \ref{stabl}, on $J_1(L)$, they are the same. Moreover, 
$$\pi'_{L_1,\psi_{\chi}} \simeq  \cInd_{\Ha_1(L) \ltimes\Ha(L_1)}^{ \Ha_1(L) \ltimes \Ha(W)} \psi_{L_1, \chi},$$
and it can also be realized on $\mathcal{S}_{L_1 ,\psi_{\chi}}$.

Let us write 
$$\pi'_{L,\psi}=\cInd_{\Ha_1(L) \ltimes \Ha(L )}^{\Ha_1(L) \ltimes \Ha(W)}\psi_L \simeq \cInd_{\Ha_1(L) \ltimes \Ha(L^{\ast} )}^{\Ha_1(L) \ltimes \Ha(W)}\sigma \simeq \oplus_{\chi \in \Irr(L_1/L)} \pi'_{L_1,\psi_{\chi}}. $$ 
The representation $\pi'_{L,\psi}$ can be realized on $\mathcal{V}_{L ,\psi}$ and $\mathcal{S}_{L ,\psi}$. 
 \subsubsection{}
Let us consider the representation $(\pi'_{L_1,\psi_{\chi}}, \mathcal{S}_{L_1 ,\psi_{\chi}})$. Let $w$ be an element of $ L^{\ast}$. Then $wg\in L^{\ast}$, for $g\in \Ha_1(L)$.  Recall that  $s_w^{\chi}$ is an element of $\mathcal{S}_{L_1 , \psi_{\chi}}$ such that it's supported  on $L_1+w$ and $s_w^{\chi}(w)=1$.
\begin{lemma}
\begin{itemize}
\item[(1)] For $g\in J_1(L)$, $\pi'_{L_1,\psi_{\chi}}(g)s_w^{\chi}=s_w^{\chi}$.
\item[(2)] For $g\in \Ha_1(L)$, $\pi'_{L_1,\psi_{\chi}}(g)s_w^{\chi}=\psi(-\tfrac{\langle wg, w\rangle}{2}) \chi(wg-w) s_w^{\chi}$.
\end{itemize}
\end{lemma}
\begin{proof}
1) $\pi'_{L_1,\psi_{\chi}}(g) s_w^{\chi}(w_1)=s_w^{\chi}(w_1g)$. So $\supp \pi'_{L_1,\psi_{\chi}}(g) s_w^{\chi}\subseteq L_1+wg$. Since $wg-w\subseteq L$, $L_1+wg=L_1+w$. Moreover, 
$$ \pi'_{L_1,\psi_{\chi}}(g)s_w^{\chi}(w)=s_w^{\chi}(wg)=s_w^{\chi}(wg-w+w)=s_w^{\chi}(w).$$
Hence $\pi'_{L_1,\psi_{\chi}}(g)s_w^{\chi}=s_w^{\chi}$.\\
2) Similarly, $wg-w\subseteq L_1$, $L_1+wg=L_1+w$. Then
 \begin{align}
 \pi'_{L_1,\psi_{\chi}}(g)s_w^{\chi}(w)&=s_w^{\chi}(wg)=s_w^{\chi}(wg-w+w)\\
 &=\psi(-\tfrac{1}{2}\langle wg-w, w\rangle) \chi(wg-w)s_w^{\chi}(w)\\
 &=\psi(-\tfrac{1}{2}\langle wg, w\rangle) \chi(wg-w)s_w^{\chi}(w).
 \end{align}
\end{proof}
\section{The non-archimedean local field case II$\tfrac{1}{2}$ }\label{pn21/2}
In this section, $F$ is a finite unramified extension  of the field $\Q_2$ of $2$-adic numbers. We will let $\psi$ be a fixed non-trivial character of $F$ with order $e$, meaning that $\psi|_{\mathcal{P}^e}=1$, and  $\psi|_{\mathcal{P}^{e-1}}\neq 1$. Retain the notations from the introduction.  

\subsection{Weil representation}
  For $w=x+x^{\ast}$, $w'=y+y^{\ast}\in W $,  with $x, y\in X$, $x^{\ast}, y^{\ast}\in X^{\ast}$,  let us define
   $$B(w, w')=\langle x, y^{\ast}\rangle.$$ Then $$\langle w, w'\rangle=B(w, w')-B(w', w).$$ Let $H_{B}(W)=W\times F$ denote   the corresponding Heisenberg group, defined as follows:
 $$(w, t)\cdot (w',t')=(w+w', t+t'+B(w,w')).$$

   For $g\in \Sp(W)$, following \cite{GeLy, GuHa, Rao, Ta}, we let $\Sigma_g$  be the set  of  continuous   functions $q$ from $W$ to $F$, such that
 \begin{equation}\label{equiv}
 q(w+w')-q(w)-q (w')=B(wg, w'g)-B(w, w'), \quad \quad w, w'\in W.
 \end{equation}

   Following \cite{Ta}, \cite{We}, let $$Ps(W)=\{(g, q)\mid g\in \Sp(W), q\in \Sigma_g\},$$ which is called the linear pseudosymplectic group with respect to $X, X^{\ast}$. The group law is given as follows: $$( g, q) ( g', q')=(gg', q''),$$ where $q^{''}(w)=q(w)+q'(wg)$, for $w\in W$. The group  $\Ps(W)$ then acts on $H_B(W)$ as follows:
   $$(w, t)\cdot(g,q) = (wg, t + q(w)).$$
   \begin{lemma}\label{split}
   There exists a split exact sequence: $1\longrightarrow W^{\vee} \longrightarrow \Ps(W) \longrightarrow \Sp(W) \longrightarrow 1$, where $W^{\vee}=\Hom_{group}(W, F)$.
   \end{lemma}
   \begin{proof}
   See \cite[p.351, Lem.3.1]{Rao}, or \cite[Lem.2.3]{Ta}.
   \end{proof}
    Moreover, by \cite[p.351, Lem.3.1]{Rao}, \cite[Lem.2.3]{Ta}, \cite[Sect.4]{We}, there exists an explicit group homomorphism
   $\alpha: \Sp(W) \longrightarrow \Ps(W); g\longmapsto (g, q_g)$, where
   \begin{equation}\label{B}
   \begin{split}
   q_g(x+x^{\ast})&=\tfrac{1}{2}B(
xa, xb)+\tfrac{1}{2}B( x^{\ast}c, x^{\ast}d)+B(x^{\ast}c, xb),
\end{split}
\end{equation} for $g=\begin{pmatrix} a& b\\c & d\end{pmatrix}\in \Sp(W)$.
 Let $\Ha(W)$ be the usual Heisenberg group.  Then there exists an isomorphism:
$$\alpha_B:    \Sp(W)\ltimes \Ha(W) \longrightarrow  \Sp(W) \ltimes H_B(W)   ,$$
$$ \qquad  [g,(x,x^{\ast}; t)]  \longmapsto [ \alpha(g), (x, x^{\ast};t+\tfrac{1}{2}B( x, x^{\ast}))].$$
\subsection{Self-dual lattice}   We  let $L_1$ be a self-dual lattice of $W$ with respect to $\psi$ such that $L_1=L_1\cap X\oplus L_1\cap X^{\ast}$. For simplicity, we \emph{assume} that $L_1\cap X=\mathcal{P}^{[\tfrac{e+1}{2}]}e_1 + \cdots + \mathcal{P}^{[\tfrac{e+1}{2}]}e_m $, $L_1\cap X^{\ast}=\mathcal{P}^{[\tfrac{e}{2}]}e_1^{\ast} + \cdots +  \mathcal{P}^{[\tfrac{e}{2}]}e_m^{\ast} $. \emph{ From now on,we will fix such lattice $L_1$  in this  section \ref{pn21/2}.}

Let $\Ha(L_1)=L_1\times F$ denote the corresponding subgroup of $\Ha(W)$.   Let $\chi\in \Irr(L_1)$.
Let $\psi_{L_1, \chi}$ denote the extended  character of $\Ha(L_1)$  from $\psi$ of  $F$ defined as:
$$\psi_{L_1, \chi}: \Ha(L_1) \longrightarrow \C^{\times}; (l,t) \longmapsto \psi(t-\tfrac{1}{2}B(l,l))\chi(l),$$ 
for $l\in L_1$.  Let us check that $\psi_{L_1, \chi}$ is well-defined. 
For two elements $(l,t),(l',t') \in \Ha(L_1)$, 
\begin{align*}
\psi_{L_1, \chi}([l,t])\psi_{L_1, \chi}([l',t'])&=\psi(t-\tfrac{1}{2}B(l,l))\chi(l)\psi(t'-\tfrac{1}{2}B(l',l'))\chi(l')\\
&=\psi(t+t'-\tfrac{1}{2}B(l,l)-\tfrac{1}{2}B(l',l'))\chi(l+l').
\end{align*}
\begin{align*}
\psi_{L_1, \chi}([l,t][l',t'])&=\psi_{L_1, \chi}([l+l', t+t'+\tfrac{\langle l,l'\rangle}{2}])\\
&=\psi(t+t'+\tfrac{1}{2}\langle l,l'\rangle-\tfrac{1}{2}B(l+l',l+l'))\chi(l+l').
\end{align*}
\begin{align*}
&\psi(\tfrac{1}{2}\langle l,l'\rangle-\tfrac{1}{2}B(l+l',l+l'))\\
&=\psi(\tfrac{1}{2}B(l,l')-\tfrac{1}{2}B(l',l)-\tfrac{1}{2}B(l+l',l+l'))\\
&=\psi(-\tfrac{1}{2}B(l,l)-\tfrac{1}{2}B(l',l')-B(l',l))\\
&=\psi(-\tfrac{1}{2}B(l,l)-\tfrac{1}{2}B(l',l')).
\end{align*}
  Let us define the representation:
 $$(\pi_{L_1,\psi_{\chi}}=\cInd_{\Ha(L_1)}^{\Ha(W)} \psi_{L_1,\chi}, \quad  \mathcal{S}_{L_1,\psi_{\chi}}=\cInd_{\Ha(L_1)}^{\Ha(W)} \C).$$
\begin{lemma}
$\pi_{L_1,\psi_{\chi}}$ defines a Heisenberg representation of $\Ha(W)$ associated to  the center character $\psi$.
\end{lemma} 
\begin{proof}
See \cite[p.42, II.8]{MVW}.
\end{proof}
\begin{remark}
If $\chi$ is the trivial character, we will  write $\pi_{L_1,\psi}$.  Then $\pi_{L_1,\psi_{\chi}}\simeq \pi_{L_1,\psi}$, for any $\chi$.
\end{remark}
The space $\mathcal{S}_{L_1,\psi_{\chi}}$ consists  of locally constant, compactly supported functions $f$ on $W$  such that
 $$f(l+w)=\psi(-\tfrac{1}{2}\langle l, w\rangle-\tfrac{1}{2}B(l,l))\chi(l)f(w),$$
  for $l=x_l+x_{l}^{\ast}\in L_1=L_1\cap X\oplus L_1\cap X^{\ast}$, $w\in W$.  For $h=(w',t)\in \Ha(W)$,
\begin{equation}\label{BB'}
\pi_{L_1,\psi_{\chi}}(h) f(w)=f([w,0][w',t])=f([w+w', t+\tfrac{1}{2}\langle w, w'\rangle])=\psi( t+\tfrac{1}{2}\langle w, w'\rangle) f(w+w').
\end{equation}
\subsection{Self-dual lattice II}
Let $L_2$ be another self-dual lattice of $W$ with respect to $\psi$ such that $L_2=L_2\cap X \oplus L_2\cap X^{\ast}$.
Similarly, we can define the Heisenberg  representation:
$$(\pi_{L_2,\psi}=\cInd_{\Ha(L_2)}^{\Ha(W)} \psi_{L_2}, \quad  \mathcal{S}_{L_2,\psi}=\cInd_{\Ha(L_2)}^{\Ha(W)} \C).$$
Then $\pi_{L_1,\psi} \simeq \pi_{L_2,\psi}$. For  $g\in \Sp(W)$, let us define a $\C$-linear map  $ \mathfrak{i}^g_{L_1,L_2}$ from $\mathcal{S}_{L_1,\psi}$ to $ \mathcal{S}_{L_2,\psi}$ as follows:
\begin{align}\label{interL1L2}
\mathfrak{i}^g_{L_1,L_2}(f)(w)=\int_{a\in L_2}\psi(\tfrac{1}{2}\langle a, w\rangle+\tfrac{1}{2}B(a,a))f((a+w)g) da.
\end{align}
\begin{lemma}\label{welldL1L2}
$\mathfrak{i}^g_{L_1,L_2}$ is well-defined.
\end{lemma}
\begin{proof}
  Let  $f\in \mathcal{S}_{L_1,\psi}$.\\ 
1) For $l\in L_2$, we have:
 \begin{align*}
&\mathfrak{i}^g_{L_1,L_2}(f)(l+w)\\
&=\int_{a\in L_2}  \psi(\tfrac{1}{2}\langle a, w+l\rangle+\tfrac{1}{2}B(a,a))f((a+w+l)g) da\\
&=\int_{a\in L_2}  \psi(\tfrac{1}{2}\langle a-l, w+l\rangle+\tfrac{1}{2}B(a-l,a-l))f((a+w)g) da\\
&=\int_{a\in L_2}\psi(\tfrac{1}{2}\langle a, l\rangle)\psi(-\tfrac{1}{2}B(l,a)-\tfrac{1}{2}B(a,l))\psi(-\tfrac{1}{2}\langle l, w\rangle+\tfrac{1}{2}B(l,l)) \psi(\tfrac{1}{2}\langle a, w\rangle+\tfrac{1}{2}B(a,a))f((a+w)g) da\\
&=\psi(-\tfrac{1}{2}\langle l, w\rangle+\tfrac{1}{2}B(l,l))\mathfrak{i}^g_{L_1,L_2}(f)(w)
.
 \end{align*}
 2) Assume $\supp(f) \subseteq C$, for some compact set $C$ of $W$. If $w\notin Cg^{-1}+L_2$, then $f((a+w)g)=0$,  for any  $a\in L_2$. Hence $\supp (\mathfrak{i}^g_{L_1,L_2}(f)) \subseteq L_2+Cg^{-1}$.\\
 3)  If $f(a+w)=f(w)$, for some open compact subgroup $K$ of $W$ and $a\in W$.  Then for any $l\in 2(L_2\cap Kg^{-1})$, we have: 
 \begin{align*}
&\mathfrak{i}^g_{L_1,L_2}(f)(l+w)\\
&=\int_{a\in L_2}  \psi(\tfrac{1}{2}\langle a, w+l\rangle+\tfrac{1}{2}B(a,a))f((a+w+l)g) da\\
 &=\int_{a\in L_2}  \psi(\tfrac{1}{2}\langle a, l\rangle)\psi(\tfrac{1}{2}\langle a, w\rangle+\tfrac{1}{2}B(a,a))f((a+w)g) da\\
 &=\mathfrak{i}^g_{L_1,L_2}(f)(w).
 \end{align*}
As  $2(L_2\cap Kg^{-1})$ is an open set that contains the identity element, it at least contains an open compact subgroup of $W$.
\end{proof}
\begin{lemma}\label{coomL1L2}
Let $h=(w',t)\in \Ha(W)$. The following diagram is commutative:
\begin{equation}\label{eq1}
\begin{CD}
 \mathcal{S}_{L_1,\psi}@>\mathfrak{i}^g_{L_1,L_2}>> \mathcal{S}_{L_2,\psi}  \\
       @V\pi_{L_1,\psi}(h) VV @VV\pi_{L_2,\psi}(h^{g^{-1}}) V  \\
 \mathcal{S}_{L_1,\psi}@>\mathfrak{i}^g_{L_1,L_2}>> \mathcal{S}_{L_2,\psi} 
\end{CD}
\end{equation} 
\end{lemma}
\begin{proof}
  \begin{align*}
 & \pi_{L_2,\psi}(h^{g^{-1}})\mathfrak{i}^g_{L_1,L_2}(f)(w)\\
 &=\pi_{L_2,\psi}(w'g^{-1},t)\mathfrak{i}^g_{L_1,L_2}(f)(w)\\
  &=\psi( t+\tfrac{1}{2}\langle w, w'g^{-1}\rangle)\mathfrak{i}^g_{L_1,L_2}(f)(w+w'g^{-1})\\
 &=\psi( t+\tfrac{1}{2}\langle w, w'g^{-1}\rangle)\int_{a\in L_2}\psi(\tfrac{1}{2}\langle a, w+w'g^{-1}\rangle+\tfrac{1}{2}B(a,a))f((a+w+w'g^{-1})g) da\\
 &=\int_{a\in L_2}\psi( t+\tfrac{1}{2}\langle w, w'g^{-1}\rangle+\tfrac{1}{2}\langle a, w+w'g^{-1}\rangle+\tfrac{1}{2}B(a,a))f(ag+wg+w') da;
  \end{align*}
  \begin{align*}
 &\mathfrak{i}^g_{L_1,L_2}[\pi_{L_2,\psi}(h)f](w)\\
 &=\int_{a\in L_2}\psi(\tfrac{1}{2}\langle a, w\rangle+\tfrac{1}{2}B(a,a))[\pi_{L_2,\psi}(h)f]((a+w)g) da \\
 &=\int_{a\in L_2}\psi(\tfrac{1}{2}\langle a, w\rangle+\tfrac{1}{2}B(a,a))\psi( t+\tfrac{1}{2}\langle (a+w)g, w'\rangle)f((a+w)g+w')da\\
  &=\int_{a\in L_2}\psi( t+\tfrac{1}{2}\langle a, w\rangle+\tfrac{1}{2}B(a,a)+\tfrac{1}{2}\langle a+w, w'g^{-1}\rangle)f((a+w)g+w')da\\
  &=\pi_{L_2,\psi}(h^{g^{-1}})\mathfrak{i}^g_{L_1,L_2}(f)(w).
  \end{align*}
\end{proof}
If $g=1\in \Sp(W)$, we write  $\mathfrak{i}_{L_1,L_2}$ for $\mathfrak{i}^g_{L_1,L_2}$.
\begin{lemma}\label{nonzeroL1L2}
\begin{itemize} 
\item[(1)] $\mathfrak{i}_{L_1,L_2}$ is  a non-zero map.
\item[(2)] $\mathfrak{i}_{L_1,L_2}$ is bijective.
\end{itemize}
\end{lemma}
\begin{proof}
1) Let $f_0\in  \mathcal{S}_{L_1,\psi}$, with support on $L_1$. Then $f_0(l)=\psi(-\tfrac{1}{2}B(l,l))f_0(0)\neq 0$, for $l\in L_1$.  So
 \begin{align*}
 \mathfrak{i}_{L_1,L_2}(f_0)(0)&=\int_{a\in L_2}\psi(\tfrac{1}{2}B(a,a))f_0(a) da\\
 &=\int_{a\in L_2\cap L_1} \psi(\tfrac{1}{2}B(a,a))\psi(-\tfrac{1}{2}B(a,a))f_0(0)da\\
 &=\mu(L_1\cap L_2)f_0(0)\neq 0.
 \end{align*}
2) Considering that both $\mathcal{S}_{L_1,\psi}$, $\mathcal{S}_{L_2,\psi}$ are irreducible $\Ha(W)$-modules, and by virtue of the preceding lemma,  $\mathfrak{i}_{L_1,L_2}$ is established as a non-zero intertwining operator between these two modules. Consequently,  $\mathfrak{i}_{L_1,L_2}$ is a bijective map.
\end{proof}
Similarly, we can define $\mathfrak{i}_{L_2,L_1}$. 
\begin{lemma}
Let $c=\mu(L_2)\mu(L_1\cap L_2)$. Then  $\mathfrak{i}_{L_1,L_2}\circ \mathfrak{i}_{L_2,L_1}=c\id$.
\end{lemma}
\begin{proof}
Let $f\in \mathcal{S}_{L_2,\psi}$. Then:
\begin{align*}
 & \mathfrak{i}_{L_1,L_2}\circ \mathfrak{i}_{L_2,L_1}(f)(w)\\
 &=\int_{a\in L_2}\psi(\tfrac{1}{2}\langle a, w\rangle+\tfrac{1}{2}B(a,a))[ \mathfrak{i}_{L_2,L_1}(f)](a+w) da\\
  &=\int_{a\in L_2}\int_{a'\in L_1}\psi(\tfrac{1}{2}\langle a, w\rangle+\tfrac{1}{2}B(a,a)) \psi(\tfrac{1}{2}\langle a', a+w\rangle+\tfrac{1}{2}B(a',a'))f(a'+a+w) dada'\\
    &=\int_{a'\in L_1}\int_{a\in L_2}\psi(\langle a', a\rangle) \psi(\tfrac{1}{2}\langle a', w\rangle+\tfrac{1}{2}B(a',a'))f(a'+w) dada'\\
     &=\mu(L_2)\int_{a'\in L_1\cap L_2} \psi(\tfrac{1}{2}\langle a', w\rangle+\tfrac{1}{2}B(a',a'))f(a'+w)  da'\\
 &=\mu(L_2)\mu(L_1\cap L_2)f(w).
  \end{align*}
\end{proof}
\subsection{ Iwahori decomposition}\label{decom} For our purpose, let us recall some results of Iwahori decomposition from  \cite{Vi}.  For a subset $S\subseteq \{1, \cdots, m\}$, let us define an element   $\omega_S$ of  $ \Sp(W)$ as follows: $ (e_i)\omega_S=\left\{\begin{array}{lr}
-e_i^{\ast}& i\in S\\
 e_i & i\notin S
 \end{array}\right.$ and $ (e_i^{\ast})\omega_S=\left\{\begin{array}{lr}
e_i^{\ast}& i\notin S\\
 e_i & i\in S
 \end{array}\right.$. Let $S_m$ denote the permutation group on $m$ elements. For $s\in S_m$, let $\omega_s$ be an element of $\Sp(W)$ defined as $(e_i)\omega_s=e_{s(i)}$ and $(e_j^{\ast}) \omega_s=e_{s(j)}^{\ast}$. Let  $\mathfrak{W}=\{ \omega_s \omega_S \mid S\subseteq  \{1, \cdots, m\}, s\in S_m\}$, $\mathfrak{W}_0=\{\omega_S \mid S\subseteq  \{1, \cdots, m\}\}$ . Then  $\mathfrak{W}$ represents   a Weyl group of $\Sp(W)$. 
 
 Under the basis $\{e_1, \cdots, e_m; e_1^{\ast}, \cdots, e_m^{\ast}\}$, let us identity $\Sp(W)$ with $\Sp_{2m}(F)$. Let $I$ be the Iwahori subgroup of $\Sp_{2m}(F)$ as given in \cite{HeOi} or \cite{TaWo}. Let 
 $$D=\{ \begin{bmatrix}
d& 0 \\0 &  d^{-1} \end{bmatrix} \mid d = \diag(2^{k_1}, \cdots, 2^{k_m}) \textrm{ with integers } k_1, \cdots,   k_m \}.$$
Let $\mathfrak{W}^{\textrm{eaff}}=D\rtimes \mathfrak{W}$ and $\mathfrak{W}_0^{\textrm{aff}}=D\rtimes \mathfrak{W}_0$. By \cite{Vi}, $\mathfrak{W}^{\textrm{eaff}}$ represents  the Iwahori-Weyl group of $\Sp_{2m}(F)$ and 
$\Sp_{2m}(F)= \cup_{w\in \mathfrak{W}^{\textrm{eaff}}} IwI$.
 \subsection{On the Iwahori group $I$}
 Let us define:
\begin{itemize}
\item $L_1'=\left\{ \begin{array}{cl}
2 L_1\cap X\oplus \tfrac{1}{2} L_1\cap X^{\ast} & \textrm{ if } 2\mid e,\\
\tfrac{1}{2} L_1\cap X\oplus 2 L_1\cap X^{\ast} & \textrm{ if } 2\nmid e.\end{array}\right.$ 
\item $L_1^{0}= L_1\cap L_1'$.
 \end{itemize}
 \begin{lemma}\label{trans}
 Let $g\in I$.
 \begin{itemize}
 \item[(1)] $L_1^{0}g=L_1^0$ and $L_1g=L_1$.
 \item[(2)] $\psi(q_g(l))=1$, for $l\in L_1^0$.
 \end{itemize}
 \end{lemma}
 \begin{proof}
 1) See \cite[Sect.1.2]{TaWo}, or \cite{Wo}.\\
 2) For $l\in L_1^0$, $\psi(\tfrac{1}{2}B(l,l))=1$. Since $lg\in L_1^0$,  $\psi(\tfrac{1}{2}B(lg,lg))=1$. So $\psi(q_g(l))=\psi(\tfrac{1}{2}B(lg,lg)-\tfrac{1}{2}B(l,l))= 1$. 
 \end{proof}

\begin{lemma}\label{nonzeromap12}
\begin{itemize} 
\item[(1)] $\mathfrak{i}^g_{L_1',L_1}$ is  a non-zero map.
\item[(2)] $\mathfrak{i}^g_{L_1',L_1}$ is bijective.
\end{itemize}
\end{lemma}
\begin{proof}
1) Let $f_0\in  \mathcal{S}_{L_1',\psi}$, with support on $L_1'$. Then $f_0(l)=\psi(-\tfrac{1}{2}B(l,l))f_0(0)\neq 0$, for $l\in L_1'$.    So
 \begin{align*}
 \mathfrak{i}^g_{L_1',L_1}(f_0)(0)&=\int_{a\in L_1}\psi(\tfrac{1}{2}B(a,a))f_0(ag) da\\
 &=\int_{a\in L_1'g^{-1}\cap L_1}\psi(\tfrac{1}{2}B(a,a)-\tfrac{1}{2}B(ag,ag)) f_0(0)da\\
 &=\int_{a\in L_1^0g^{-1}=L_1^0}\psi(\tfrac{1}{2}B(a,a)-\tfrac{1}{2}B(ag,ag)) f_0(0)da\\
 &=\mu(L_1^g\cap L_1g^{-1})\psi(-q_{g}(a))f_0(0)\neq 0.
 \end{align*}
 2) Note that  $\mathcal{S}_{L_1,\psi}$, $\mathcal{S}_{L_1',\psi}$ both are  irreducible $\Ha(W)$-modules. By Lemma \ref{coomL1L2}, $\mathfrak{i}^g_{L_1',L_1}$ is  bijective.
\end{proof}

For $g\in I$, let us define $M[g]\in \End(\mathcal{S}_{L_1,\psi})$ as follows:
\begin{align}\label{gactionsigma22}
M[g]=  \mathfrak{i}_{L_1',L_1} \circ\mathfrak{i}^g_{L_1,L_1'} .
\end{align}
\begin{lemma}
$M[g]$ is bijective. 
\end{lemma}
\begin{proof}
It follows from Lemmas \ref{nonzeroL1L2}, \ref{nonzeromap12}.
\end{proof}

\subsection{On the  Iwahori-Weyl group $\mathfrak{W}^{\textrm{eaff}}$}
\begin{lemma}\label{affweylgroup}
Let $g\in \mathfrak{W}^{\textrm{eaff}}$. For $l\in L_1$, $\psi(q_g(l))=1$.
\end{lemma}
\begin{proof}
1) If $g=\omega_s  \in \mathfrak{W}$, for some $s\in S_m$,  then $q_g(w)=0$, for any $w\in W$.\\
2) If $g=\omega_S \in  \mathfrak{W}$, then by (\ref{B}), $\psi(q_g(l))=1$, for any  $l\in L_1$.\\
3) If $g=\omega_S\omega_s   \in  \mathfrak{W}$, then 
$\psi(q_g(l))=\psi(q_{\omega_S}(l)) \psi(q_{\omega_s}(l \omega_S))=1$.\\
4) If $g\in D$, then $q_g(w)=0$, for any $w\in W$.\\
5) If  $g=g_1g_2$, for $g_1\in  \mathfrak{W}$, $g_2\in D$, then $\psi(q_g(l))=\psi(q_{g_1}(l)) \psi(q_{g_2}(l g_1))=1$.
\end{proof}
For $g\in \mathfrak{W}^{\textrm{eaff}}$, let us define $M[g]\in \End(\mathcal{S}_{L_1,\psi})$ as follows:
\begin{align}\label{gactionsigmaafff}
M[g]=   \mathfrak{i}^g_{L_1,L_1}.
\end{align}
\begin{lemma}\label{nonzeromap1}
 $M[g]$ is bijective.
\end{lemma}
\begin{proof}
1) Let $f_0\in  \mathcal{S}_{L_1,\psi}$, with support on $L_1$. Then $f_0(l)=\psi(-\tfrac{1}{2}B(l,l))f_0(0)\neq 0$, for $l\in L_1$.    
 \begin{align*}
 \mathfrak{i}^g_{L_1,L_1}(f_0)(0)&=\int_{a\in L_1}\psi(\tfrac{1}{2}B(a,a))f_0(ag) da\\
 &=\int_{a\in L_1g^{-1}\cap L_1}\psi(\tfrac{1}{2}B(a,a)-\tfrac{1}{2}B(ag,ag)) f_0(0)da\\
 &=\mu(L_1g^{-1}\cap L_1)f_0(0)\neq 0.
 \end{align*}
 2) Note that  $\mathcal{S}_{L_1,\psi}$ is an irreducible  $\Ha(W)$-module. By Lemma \ref{coomL1L2}, $M[g]$ is  bijective.
\end{proof}
\subsection{On the whole group $\Sp(W)$}
\begin{lemma}\label{coomL2}
Let $h=(w',t)\in \Ha(W)$, $g\in I$ or $\mathfrak{W}^{\textrm{eaff}}$. The following diagram is commutative:
\begin{equation}\label{eq2}
\begin{CD}
 \mathcal{S}_{L_1,\psi}@>M[g]>> \mathcal{S}_{L_1,\psi}  \\
      @V\pi_{L_1,\psi}(h) VV @VV\pi_{L_1,\psi}(h^{g^{-1}}) V  \\
\mathcal{S}_{L_1,\psi}@>M[g]>> \mathcal{S}_{L_1,\psi} 
\end{CD}
\end{equation} 
\end{lemma}
\begin{proof}
It follows from (\ref{eq1}).           
\end{proof}
Recall the  Iwahori decomposition: $\Sp(W)=I \mathfrak{W}^{\textrm{eaff}}I.$ For an element $g\in \Sp(W)$, let us \emph{fix } a decomposition:
$$g=i_1\omega_g i_2$$
for $i_1, i_2\in I$ and $\omega_g\in \mathfrak{W}^{\textrm{eaff}}$.  In this context, it is stipulated that for   $g=1\in \Sp(W)$, $i_1=i_2=1$. 
Then
$$M[1]= \mu(L_1) [\mu(L_1)\mu(L_1\cap L_1') ]\mu(L_1)Id.$$
\begin{remark}
In our analysis of the self-dual lattice $L_1$, we shall select the Haar measure $\mu$ on $W$  with the property that $\mu(L_1) [\mu(L_1)\mu(L_1\cap L_1') ]\mu(L_1)$ equals $1$. 
\end{remark}
 Let us define: 
$$M[g]=M[i_1]M[\omega_g]M[i_2].$$
Then $M[g]\in \End(\mathcal{S}_{L_1,\psi})$ and $M[g]$ is bijective. Moreover, $M[g]$ satisfies the diagram (\ref{eq2}). Hence there exists a non-trivial $2$-cocycle $c$ of order $2$ in $\Ha^2(\Sp(W), T)$ such that 
$$M[g_1]M[g_2]=c(g_1,g_2)M[g_1g_2].$$

Let us define                               
$$\pi_{L_1,\psi}([g,h])=M[g]\pi_{L_1,\psi}(h)$$
for $g\in \Sp(W)$, $h\in \Ha(W)$. Then:
\begin{align}
\pi_{L_1,\psi}([g,h])\pi_{L,\psi}([g',h'])f =c(g, g')\pi_{L_1,\psi}([g,h][g',h']) f.
\end{align}
for $[g,h]$,  $[g',h']\in \Sp(W)\ltimes \Ha(W)$, $f\in \mathcal{S}_{L_1,\psi}$.

 Our next purpose is to extend the above result to a non-self-dual lattice.
\subsection{Non-self-dual lattice I}
Let us consider a non-self-dual lattice.  Keep the above notations.   Let $L$ be a sublattice of $L_1$ such that $L=L\cap X\oplus L\cap X^{\ast}$.   Let  $[L_1:L ]=q^r$, for some $r$. Let $\Ha(L )=L \times F$ denote the corresponding subgroup of $\Ha(W)$. Let $\psi_{L }$ denote the extended  character of $\Ha(L )$  from $\psi$ of  $F$ defined as:
$$\psi_{L }: \Ha(L ) \longrightarrow \C^{\times}; (l,t) \longmapsto \psi(t-\tfrac{1}{2}B(l,l)).$$ 
   Let us define the representation:
 $$(\pi_{L ,\psi}=\cInd_{\Ha(L )}^{\Ha(W)} \psi_{L }, \quad  \mathcal{S}_{L ,\psi}=\cInd_{\Ha(L )}^{\Ha(W)} \C).$$
 Note that $\Ha(L ) \subseteq \Ha(L_1)$.  Let $\chi$ be a character of $L_1/L $. Then $\psi_{L_1, \chi}|_{\Ha(L )}= \psi_{L }$. 
By Clifford theory, $\cInd_{\Ha(L ) }^{\Ha(L_1) } \psi_{L }\simeq  \oplus_{\chi \in \Irr(L_1/L )} \psi_{L_1, \chi}$. As a consequence, we obtain:
\begin{lemma}
$\pi_{L, \psi}\simeq q^r\pi_{L_1, \psi}.$
\end{lemma}
\begin{proof}
$\pi_{L,\psi}=\cInd_{\Ha(L)}^{\Ha(W)} \psi_{L}\simeq \cInd_{\Ha(L_1)}^{\Ha(W)}  (\cInd_{\Ha(L) }^{\Ha(L_1) } \psi_{L})$
$\simeq \cInd_{\Ha(L_1)}^{\Ha(W)}( \oplus_{\chi \in \Irr(L_1/L)} \psi_{L_1, \chi}) \simeq \oplus_{\chi \in \Irr(L_1/L)}  \pi_{L_1, \psi_{\chi}}$
$\simeq q^r\pi_{L_1, \psi}$.
\end{proof}
The space $\mathcal{S}_{L, \psi}$ consists  of locally constant, compactly supported functions $f$ on $W$  such that
 $$f(l+w)=\psi(-\tfrac{1}{2}\langle l, w\rangle-\tfrac{1}{2}B(l,l))f(w),$$
  for $l\in L$, $w\in W$.  For $h=(w',t)\in \Ha(W)$,
\begin{equation}\label{BB'1}
\pi_{L,\psi}(h) f(w)=\psi( t+\tfrac{1}{2}\langle w, w'\rangle) f(w+w').
\end{equation}
Let $\mathcal{S}_{L_1,\psi_{\chi}}$  denote the subspace of the elements $f$ in $\mathcal{S}_{L,\psi}$ such that 
$$f(l+w)=\psi(-\tfrac{1}{2}\langle l, w\rangle-\tfrac{1}{2}B(l,l))\chi(l)f(w),$$
for all $l\in L_1$.  Let $\Lambda_{L_1}$( resp. $\Lambda_{L_1/L}$) be a set of representatives for $W/L_1$( resp. $L_1/L$). Then 
$$\Lambda_{L}=\{ w_0+l_0\mid w_0\in  \Lambda_{L_1}, l_0\in  \Lambda_{L_1/L}\}$$
serves as a representative set for  $W/L$. For any   $w_0\in\Lambda_{L_1}$, $w_1=w_0+l_0\in \Lambda_{L}$,  we let  $s^{\chi}_{w_0}$ and  $s_{w_1}$  be the functions on $W$ such that 
\begin{itemize}
\item $\supp s^{\chi}_{w_0} \subseteq L_1+w_0$,
\item  $\supp s_{w_1} \subseteq L+w_1$,
\item $s^{\chi}_{w_0}(l+w_0)=\psi(-\tfrac{1}{2}\langle l, w_0\rangle-\tfrac{1}{2}B(l,l))\chi(l)$, for $l\in L_1$, 
\item $s_{w_1} (l+w_1)=\psi(-\tfrac{1}{2}\langle l, w_1\rangle-\tfrac{1}{2}B(l,l))$, for $l\in L$.
\end{itemize}
Then $\{ s_{w_1}\}_{w_1\in \Lambda_{L}}$ and $\{ s^{\chi}_{w_0}\}_{w_0\in \Lambda_{L_1}}$ serve as the  bases of $\mathcal{S}_{L, \psi}$ and  $\mathcal{S}_{L_1,\psi_{\chi}}$, respectively.  Let $V_{L_1, w_0}$ be the subspace  spanned by $\{s_{w_1}, w_1\in w_0+\Lambda_{L_1/L}\}$. For any $w_1=w_0+l_0\in w_0+ \Lambda_{L_1/L}$, $l\in L$, we have:
$$s^{\chi}_{w_0}(w_1)=s^{\chi}_{w_0}(l_0+w_0)=\psi(-\tfrac{1}{2}\langle l_0, w_0\rangle-\tfrac{1}{2}B(l_0,l_0))\chi(l_0),$$
 \begin{align}
 &s^{\chi}_{w_0}(l+w_1)\\
 &=s^{\chi}_{w_0}(l+l_0+w_0)\\
 &=\psi(-\tfrac{1}{2}\langle l+l_0, w_0\rangle-\tfrac{1}{2}B(l+l_0,l+l_0))\chi(l+l_0)\\
 &=\psi(-\tfrac{1}{2}\langle l, w_0\rangle-\tfrac{1}{2}B( l, l)-\tfrac{1}{2}B(l,l_0)-\tfrac{1}{2}B(l_0,l))\psi(-\tfrac{1}{2}\langle l_0, w_0\rangle-\tfrac{1}{2}B(l_0,l_0))\chi(l_0)\\
 &=\psi(-\tfrac{1}{2}\langle l, l_0+w_0\rangle-\tfrac{1}{2}B( l, l))s^{\chi}_{w_0}(w_1)\\
   &=s_{w_1}(l+w_1) s^{\chi}_{w_0}(w_1).
 \end{align}
   Hence 
 $$s^{\chi}_{w_0}=\sum_{w_1\in w_0+ \Lambda_{L_1/L}} s^{\chi }_{w_0}(w_1) \cdot s_{w_1}\in V_{L_1, w_0}.$$
 Moreover, these $s^{\chi}_{w_0}$ are linear independence as $\chi$ runs through the characters of $L_1/L$. From dimension considerations it follows that $V_{L_1, w_0}$ is spanned by $\{s^{\chi}_{w_0}, \chi \in \Irr(L_1/L)\}$. Hence:
 $$ \mathcal{S}_{L,\psi}=\oplus_{\chi\in \Irr(L_1/L)} \mathcal{S}_{L_1,\psi_{\chi}}.$$
 Note that $\mathcal{S}_{L_1,\psi_{\chi}}$  is an irreducible $\Ha(W)$-module. 
      
  \subsubsection{} 
  For any $\chi \in \Irr(L_1/L)$,  $ (\pi_{L_1,\psi_{\chi}}, \mathcal{S}_{L_1,\psi_{\chi}}) $ is an irreducible representation of $\Ha(W)$. 
 Hence 
 $$(\sigma_{\chi}=\cInd_{\Ha(L_1)}^{\Ha(L^{\ast})} \psi_{L_1,\chi}, \mathcal{W}_{\chi}=\cInd_{\Ha(L_1)}^{\Ha(L^{\ast})} \C)$$ is an irreducible representation of $\Ha(L^{\ast})$.  The vector space $\mathcal{W}_{\chi}$ consists of the  functions $f$ on $L^{\ast}$ such that 
 $$f(l+l^{\ast})= \chi(l) \psi(-\tfrac{1}{2}\langle l, l^{\ast}\rangle-\tfrac{1}{2}B(l,l))f(l^{\ast}),$$
 for $l\in L_1$, $l^{\ast}\in L^{\ast}$. Note that 
 \begin{align}\label{LL_1}
 \psi: L_1/L \times L^{\ast}/L_1 \longrightarrow T; (x,y^{\ast}) \longrightarrow \psi(\langle x, y^{\ast}\rangle ),
 \end{align}
 defines a non-degenerate bilinear map. So for any $\chi \in \Irr(L_1/L)$, there exists an element $y^{\ast}_{\chi} \in L^{\ast}/L_1 $ such that 
 $$
 \psi(\langle  x, y_{\chi}^{\ast}\rangle )=\chi(x), \quad\quad \quad x\in L_1/L.
 $$
 Let us define a function:
  $$\mathcal{A}_{\chi_1, \chi_2}: \mathcal{W}_{\chi_1} \longrightarrow \mathcal{W}_{\chi_2},$$ 
  given by 
 \begin{align}
 \mathcal{A}_{\chi_1, \chi_2}(f)(l^{\ast})=f(l^{\ast}+y^{\ast}_{\chi_1}-y^{\ast}_{\chi_2}) \psi(\tfrac{1}{2}\langle y^{\ast}_{\chi_1}-y^{\ast}_{\chi_2}, l^{\ast}\rangle).
 \end{align}
  For $f\in  \mathcal{W}_{\chi_1}$, $l\in L_1$, we have:
 \begin{align}
 &\mathcal{A}_{\chi_1, \chi_2}(f)(l+l^{\ast})\\
 &=f(l+l^{\ast}+y^{\ast}_{\chi_1}-y^{\ast}_{\chi_2})\psi(\tfrac{1}{2}\langle  y^{\ast}_{\chi_1}-y^{\ast}_{\chi_2},  l+l^{\ast}\rangle)\\
 &=\chi_1(l) \psi(-\tfrac{1}{2}\langle l, l^{\ast}+y^{\ast}_{\chi_1}-y^{\ast}_{\chi_2}\rangle-\tfrac{1}{2}B(l,l))f(l^{\ast}+y^{\ast}_{\chi_1}-y^{\ast}_{\chi_2})\psi(\tfrac{1}{2}\langle  y^{\ast}_{\chi_1}-y^{\ast}_{\chi_2},  l+l^{\ast}\rangle)\\
 &=\chi_1(l) \psi(-\langle l, y^{\ast}_{\chi_1}-y^{\ast}_{\chi_2}\rangle-\tfrac{1}{2}B(l,l)) \psi(-\tfrac{1}{2}\langle l, l^{\ast}\rangle)f(l^{\ast}+y^{\ast}_{\chi_1}-y^{\ast}_{\chi_2})\psi(\tfrac{1}{2}\langle  y^{\ast}_{\chi_1}-y^{\ast}_{\chi_2},  l^{\ast}\rangle)\\
 &=\chi_2(l) \psi(-\tfrac{1}{2}\langle l, l^{\ast}\rangle-\tfrac{1}{2}B(l,l))\mathcal{A}_{\chi_1, \chi_2}(f)(l^{\ast}).
 \end{align}
 Hence $ \mathcal{A}_{\chi_1, \chi_2}(f) \in\mathcal{W}_{\chi_2}$.  So $\mathcal{A}_{\chi_1, \chi_2}$ is well-defined.
 \begin{lemma}\label{iter}
 $\mathcal{A}_{\chi_1, \chi_2}$ defines an intertwining operator from $\sigma_{\chi_1}$ to $\sigma_{\chi_2}$.
 \end{lemma}
 \begin{proof}
For $f \in \mathcal{W}_{\chi_1}$, $l^{\ast}\in L^{\ast}$,  $h=[l_1^{\ast}, t]\in \Ha(L^{\ast})$, we have:
\begin{align}
\mathcal{A}_{\chi_1, \chi_2}[\sigma_{\chi_1}(h) f](l^{\ast})&=[\sigma_{\chi_1}(h) f](l^{\ast}+y^{\ast}_{\chi_1}-y^{\ast}_{\chi_2})  \psi(\tfrac{1}{2}\langle y^{\ast}_{\chi_1}-y^{\ast}_{\chi_2}, l^{\ast}\rangle)\\
&= f(l_1^{\ast}+l^{\ast}+y^{\ast}_{\chi_1}-y^{\ast}_{\chi_2}) \psi( t+\tfrac{1}{2}\langle  l^{\ast}+y^{\ast}_{\chi_1}-y^{\ast}_{\chi_2}, l_1^{\ast}\rangle) \psi(\tfrac{1}{2}\langle y^{\ast}_{\chi_1}-y^{\ast}_{\chi_2}, l^{\ast}\rangle)\\
&=f(l_1^{\ast}+l^{\ast}+y^{\ast}_{\chi_1}-y^{\ast}_{\chi_2}) \psi( t+\tfrac{1}{2}\langle l^{\ast}, l_1^{\ast}\rangle) \psi(\tfrac{1}{2}\langle  y^{\ast}_{\chi_1}-y^{\ast}_{\chi_2}, l^{\ast}+l_1^{\ast}\rangle);
\end{align}
\begin{align}
\sigma_{\chi_2}(h) [\mathcal{A}_{\chi_1, \chi_2}f](l^{\ast})&= \psi( t+\tfrac{1}{2} \langle l^{\ast}, l_1^{\ast}\rangle)[\mathcal{A}_{\chi_1, \chi_2}f](l^{\ast}+l_1^{\ast})\\
&= \psi( t+\tfrac{1}{2} \langle l^{\ast}, l_1^{\ast}\rangle)f(l_1^{\ast}+l^{\ast}+y^{\ast}_{\chi_1}-y^{\ast}_{\chi_2}) \psi(\tfrac{1}{2}\langle y^{\ast}_{\chi_1}-y^{\ast}_{\chi_2}, l^{\ast}+l_1^{\ast}\rangle).
\end{align}
 \end{proof} 
  \subsubsection{}
Let $$\sigma=\cInd_{\Ha(L)}^{\Ha(L^{\ast})}\psi_{L},  \mathcal{W}=\cInd_{\Ha(L)}^{\Ha(L^{\ast})}\C.$$
The vector space $\mathcal{W}$ consists of the functions $f:L^{\ast}\longrightarrow \C$ such that 
$$f(l+l^{\ast})=\psi(-\tfrac{1}{2}\langle l, l^{\ast}\rangle-\tfrac{1}{2}B(l,l)) f(l^{\ast}), \quad\quad l\in  L.$$ The action is given as follows:
$$[\sigma(l^{\ast}) f](l_1^{\ast})= \psi(\tfrac{1}{2}\langle l_1^{\ast}, l\rangle ) f(l^{\ast}+ l_1^{\ast}).$$
\begin{lemma}\label{Bll}
For $l\in L$,  $[\sigma(l) f](l^{\ast})=\psi(-\tfrac{1}{2}B(l,l))  f(l^{\ast})$. 
\end{lemma}
\begin{proof}
\begin{align}
&[\sigma(l) f](l^{\ast})\\
&=\psi(\tfrac{1}{2}\langle l^{\ast}, l\rangle ) f(l^{\ast}+l)\\
&=\psi(\tfrac{1}{2}\langle l^{\ast}, l\rangle -\tfrac{1}{2}\langle l, l^{\ast}\rangle-\tfrac{1}{2}B(l,l)) f(l^{\ast})\\
&=\psi(-\tfrac{1}{2}B(l,l)) f(l^{\ast}).
\end{align}
\end{proof}
We can identity  $(\sigma_{\chi}, \mathcal{W}_{\chi})$ as a sub-representation of $(\sigma, \mathcal{W})$ such that $ \mathcal{W}_{\chi}$ consists of  elements 
$f$ satisfying  $f(l+l^{\ast})=\chi(l) \psi(-\tfrac{1}{2}\langle l, l^{\ast}\rangle-\tfrac{1}{2}B(l,l))f(l^{\ast})$ for $l\in L_1$. 
\begin{lemma}
For $l\in L_1$, $f\in \mathcal{W}_{\chi}$,  $[\sigma(l) f](l^{\ast})=\psi(\langle l^{\ast}, l\rangle -\tfrac{1}{2}B(l,l)) \chi(l)  f(l^{\ast})$. 
\end{lemma}
\begin{proof}
\begin{align}
&[\sigma(l) f](l^{\ast})\\
&=\psi(\tfrac{1}{2}\langle l^{\ast}, l\rangle ) f(l^{\ast}+l)\\
&=\psi(\tfrac{1}{2}\langle l^{\ast}, l\rangle -\tfrac{1}{2}\langle l, l^{\ast}\rangle-\tfrac{1}{2}B(l,l)) \chi(l) f(l^{\ast})\\
&=\psi(\langle l^{\ast}, l\rangle -\tfrac{1}{2}B(l,l)) \chi(l)  f(l^{\ast}).
\end{align}
\end{proof}
Note that these  $\mathcal{W}_{\chi}$ are different  irreducible representations  of $\Ha(L^{\ast})$ as $\chi$ runs through all characters of $L_1/L$. Hence 
$$\mathcal{W}=\oplus_{\chi\in \Irr(L_1/L)} \mathcal{W}_{\chi}.$$
\subsection{} Let us denote  $\mathcal{V}_{L,\psi}= \cInd_{\Ha(L^{\ast})}^{\Ha(W)} \mathcal{W}$. Then $\pi_{L,\psi}$ can be realized on $\mathcal{V}_{L,\psi}$. The vector space  $\mathcal{V}_{L,\psi}$ consists  of locally constant, compactly supported functions $f: W \longrightarrow \mathcal{W}$  such that
 $$f(l^{\ast}_1+w)=\psi(-\tfrac{1}{2}\langle  l_1^{\ast}, w\rangle )\sigma(l_1^{\ast})f(w),$$
  for $l^{\ast}_1\in L^{\ast}$,  $w\in W$.   For $h=(w',t)\in \Ha(W)$,
\begin{equation}\label{BB'2}
\pi_{L,\psi}(h) f(w)=f([w,0]\cdot[w',t])=f([w+w', t+\tfrac{1}{2}\langle w, w'\rangle])=\psi( t+\tfrac{1}{2}\langle w, w'\rangle ) f(w+w').
\end{equation}
Let us denote  $\mathcal{V}_{L_1 ,\psi_{\chi}}= \cInd_{\Ha(L^{\ast} )}^{\Ha(W)} \mathcal{W}_{\chi}$. This vector space  consists  of locally constant, compactly supported functions $f: W \longrightarrow \mathcal{W}_{\chi}$  such that
 $$f(l^{\ast}+w)=\psi(-\tfrac{1}{2}\langle l^{\ast},w\rangle)\sigma_{\chi}(l^{\ast})f(w),$$
  for $l^{\ast}\in L^{\ast}$,  $w\in W$.  Moreover, $\pi_{L_1,\psi_{\chi}}$ can be realized on $\mathcal{V}_{L_1 ,\psi_{\chi}}$.  Consequently, 
  $$\mathcal{V}_{L ,\psi} \simeq \oplus_{\chi \in \Irr(L_1/L )} \mathcal{V}_{L_1 ,\psi_{\chi}}.$$

\subsection{Non-self-dual lattice II}\label{nonselfdualtwo}
Under the basis $\{e_1, \cdots, e_m; e_1^{\ast}, \cdots, e_m^{\ast}\}$, let $d=\diag(d_1, \cdots, d_m; \tfrac{1}{d_1}, \cdots, \tfrac{1}{d_m})$. Let us define:
\begin{itemize}
\item $L''=Ld$,
\item $L''_1=L_1d$,
\item $L^{''\ast}$: the dual of $L''$.
 \end{itemize}
Then $L^{''\ast}$ is equal to $L^{ \ast}d$.  Similarly, we can replace the above $L$ by $L''$, and define the corresponding representations:
$$(\sigma''=\cInd_{\Ha(L'')}^{\Ha(L^{''\ast})}\psi_{L''},\quad  \mathcal{W}''=\cInd_{\Ha(L'')}^{\Ha(L^{''\ast})}\C)$$
 $$(\sigma''_{\chi}=\cInd_{\Ha(L_1'')}^{\Ha(L^{''\ast})} \psi_{L_1'',\chi}, \quad \mathcal{W}''_{\chi}=\cInd_{\Ha(L_1'')}^{\Ha(L^{''\ast})} \C)$$ 
$$(\pi_{L'' ,\psi}=\cInd_{\Ha(L'' )}^{\Ha(W)} \psi_{L''}, \quad  \mathcal{S}_{L'' ,\psi}=\cInd_{\Ha(L'' )}^{\Ha(W)} \C)$$
$$(\pi_{L'' ,\psi}=\cInd_{\Ha(L^{''\ast})}^{\Ha(W)} \sigma'', \quad \mathcal{V}_{L'',\psi}= \cInd_{\Ha(L^{''\ast})}^{\Ha(W)} \mathcal{W}'')$$
$$(\pi_{L'' ,\psi_{\chi}}=\cInd_{\Ha(L^{''\ast})}^{\Ha(W)} \sigma''_{\chi}, \quad \mathcal{V}_{L'',\psi_{\chi}}= \cInd_{\Ha(L^{''\ast})}^{\Ha(W)} \mathcal{W}''_{\chi}).$$
\begin{lemma}
There exists a group isomorphism: 
$$\mathfrak{i}: \Ha(W) \to \Ha(W); [w,t] \longmapsto [wd, t].$$
Moreover, $\mathfrak{i}$ sends $\Ha(L^{\ast})$ to $\Ha(L^{''\ast})$, $\Ha(L_1)$ to $\Ha(L_1^{''})$,  $\Ha(L)$ to $\Ha(L'')$,  $L_1/L$ to $L_1''/L''$, and  $L^{\ast}/L_1$ to $L^{''\ast}/L_1^{''}$.
\end{lemma}
\begin{proof}
Straightforward.
\end{proof}
 Similar to (\ref{LL_1}), 
\begin{align}\label{L''L_1''}
 \psi: L_1''/L'' \times L^{''\ast}/L_1'' \longrightarrow T; (x,y^{\ast}) \longrightarrow \psi(\langle x, y^{\ast}\rangle ),
 \end{align}
 defines a non-degenerate bilinear map.  So for any $\chi \in \Irr(L_1''/L'')$, there exists an element $y^{\ast}_{\chi} \in L^{''\ast}/L_1'' $ such that 
 $$
 \psi(\langle  x, y_{\chi}^{\ast}\rangle )=\chi(x), \quad\quad \quad x\in L_1''/L''.
 $$
 For any $\chi\in \Irr(L_1''/L'')$, $\chi \circ \mathfrak{i}\in \Irr( L_1/L)$. Then $\mathfrak{i}(y_{\chi}^{\ast}) = y_{\chi \circ \mathfrak{i}}^{\ast}$. Let us extend $\mathfrak{i}$ to the representation.  
\begin{itemize}\label{DDD}
\item[(1)] Define $\mathcal{D}: \C \to \C; v \longmapsto v$. Then the following diagram is commutative.  $$\begin{CD}\C @>\mathcal{D}>> \C   \\
      @V\psi_{L}([l,t]) VV @VV\psi_{L''}(\mathfrak{i}([l,t])) V  \\\C @>\mathcal{D}>> \C \end{CD}$$
\item[(2)]  Define $\mathcal{D}:  \mathcal{W}\to  \mathcal{W}''; f \longmapsto \mathcal{D}(f)$, where $\mathcal{D}(f)(l^{''\ast})=f(l^{''\ast} d^{-1})$. Then  the following diagram
    $$\begin{CD} \mathcal{W}  @>\mathcal{D}>> \mathcal{W}''\\
      @V\sigma([l^{\ast},t]) VV @VV\sigma''(\mathfrak{i}([l^{\ast},t])) V  \\
      \mathcal{W}  @>\mathcal{D}>>  \mathcal{W}'' \end{CD}$$
       is commutative, for $[l^{\ast},t] \in \Ha(L^{\ast})$. 
       \begin{proof}
       i) For $l''\in L^{''}$, \begin{align*}
       &\mathcal{D}(f)(l''+l^{''\ast})\\
       &=f(l'' d^{-1}+l^{''\ast} d^{-1})\\
       &=\psi(-\tfrac{1}{2}\langle l'' d^{-1}, l^{''\ast} d^{-1}\rangle-\tfrac{1}{2}B( l'' d^{-1}, l'' d^{-1})) f(l^{''\ast} d^{-1})\\
       &=\psi(-\tfrac{1}{2}\langle l'' , l^{''\ast} \rangle-\tfrac{1}{2}B( l'' , l'' )) f(l^{''\ast} d^{-1})\\
       &=\psi(-\tfrac{1}{2}\langle l'' , l^{''\ast} \rangle-\tfrac{1}{2}B( l'' , l'' )) \mathcal{D}(f)(l^{''\ast}).
       \end{align*}
       ii) \begin{align*}
       &\sigma''(\mathfrak{i}([l^{\ast},t])) \circ \mathcal{D}(f)(l^{''\ast})\\
       &= \mathcal{D}(f)(l^{''\ast}+l^{\ast}d)\psi(t+\tfrac{1}{2}\langle l^{''\ast}, l^{\ast}d\rangle)\\
       &=f(l^{''\ast}d^{-1}+l^{\ast})\psi(t+\tfrac{1}{2}\langle l^{''\ast}, l^{\ast}d\rangle).
       \end{align*}
        \begin{align*}
       &\mathcal{D} \circ\sigma([l^{\ast},t])(f)(l^{''\ast})\\
       &=\sigma([l^{\ast},t])(f)(l^{''\ast}d^{-1})\\
       &=f(l^{''\ast}d^{-1}+l^{\ast})\psi(t+\tfrac{1}{2}\langle l^{''\ast}d^{-1}, l^{\ast}\rangle).
       \end{align*}
       \end{proof}
      \item[(3)]  $\mathcal{D}$ sends $\mathcal{W}_{\chi \circ \mathfrak{i}}$ onto $\mathcal{W}''_{\chi}$.
      \begin{proof}
      For $l_1^{''}\in L_1''$, 
      \begin{align*}
      &\mathcal{D}(f)( l_1''+l^{''\ast})\\
      &=f( l_1''d^{-1}+l^{''\ast}d^{-1})\\
      &=\psi(-\tfrac{1}{2}\langle l_1''d^{-1}, l^{''\ast}d^{-1}\rangle-\tfrac{1}{2}B( l_1''d^{-1},l_1''d^{-1}))\chi \circ \mathfrak{i}( l_1''d^{-1})f(l^{''\ast}d^{-1})\\
       &=\psi(-\tfrac{1}{2}\langle l_1'', l^{''\ast}\rangle-\tfrac{1}{2}B(l_1'', l_1''))\chi( l_1'')\mathcal{D}(f)(l^{''\ast}).
       \end{align*}
         \end{proof}
       \item[(4)] The following diagram is commutative.
       $$\begin{CD} \mathcal{W}_{\chi_1 \circ \mathfrak{i}}  @>\mathcal{D}>> \mathcal{W}''_{\chi_1}\\
      @V\mathcal{A}_{\chi_1 \circ \mathfrak{i}, \chi_2 \circ \mathfrak{i}} VV @VV\mathcal{A}_{\chi_1, \chi_2} V  \\
    \mathcal{W}_{\chi_2 \circ \mathfrak{i}} @>\mathcal{D}>>  \mathcal{W}''_{\chi_2}\end{CD}$$
    \begin{proof}
    For $f\in  \mathcal{W}_{\chi_1 \circ \mathfrak{i}} $, we have:
   \begin{align*}
   &\mathcal{A}_{\chi_1, \chi_2} \circ \mathcal{D}(f)(l^{\ast})\\
   &=\mathcal{D}(f)(l^{\ast}+y^{\ast}_{\chi_1}-y^{\ast}_{\chi_2}) \psi(\tfrac{1}{2}\langle y^{\ast}_{\chi_1}-y^{\ast}_{\chi_2}, l^{\ast}\rangle)\\
   &=f(l^{\ast}d^{-1}+y^{\ast}_{\chi_1}d^{-1}-y^{\ast}_{\chi_2}d^{-1}) \psi(\tfrac{1}{2}\langle y^{\ast}_{\chi_1}-y^{\ast}_{\chi_2}, l^{\ast}\rangle);
   \end{align*}
        \begin{align*}
        &\mathcal{D}\circ \mathcal{A}_{\chi_1 \circ \mathfrak{i}, \chi_2 \circ \mathfrak{i}} (f)(l^{\ast})\\
        &= \mathcal{A}_{\chi_1 \circ \mathfrak{i}, \chi_2 \circ \mathfrak{i}} (f)(l^{\ast}d^{-1})\\
        &=f(l^{\ast}d^{-1}+y^{\ast}_{\chi_1}d^{-1}-y^{\ast}_{\chi_2}d^{-1}) \psi(\tfrac{1}{2}\langle y^{\ast}_{\chi_1}d^{-1}-y^{\ast}_{\chi_2}d^{-1}, l^{\ast}d^{-1}\rangle)\\
        &=\mathcal{A}_{\chi_1, \chi_2} \circ \mathcal{D}(f)(l^{\ast}).
        \end{align*}
    \end{proof}
\end{itemize}
\subsection{} For  $g\in \Sp(W)$, let us define a $\C$-linear map  $ \mathfrak{i}^g_{L,L''}$ from $\mathcal{V}_{L,\psi}$ to $ \mathcal{V}_{L'',\psi}$ as follows:
\begin{align}\label{interLL''}
\mathfrak{i}^g_{L,L''}(f)(w)&=\int_{a\in L^{''\ast}}\psi(\tfrac{1}{2}\langle a, w\rangle) \sigma''(-a)\mathcal{D}[f((a+w)g)] da.
\end{align}
\begin{lemma}\label{welldL1L2}
$\mathfrak{i}^g_{L,L''}$ is well-defined.
\end{lemma}
\begin{proof}
  Let  $f\in \mathcal{V}_{L,\psi}$.\\ 
1) For $l\in L^{''\ast}$, we have:
 \begin{align*}
&\mathfrak{i}^g_{L,L''}(f)(l+w)\\
&=\int_{a\in L^{''\ast}}\psi(\tfrac{1}{2}\langle a, w+l\rangle) \sigma''(-a)\mathcal{D}[f((a+w+l)g)] da\\
&=\int_{a\in L^{''\ast}}  \psi(\tfrac{1}{2}\langle a-l, w+l\rangle) \sigma''(-a+l)\mathcal{D}[f(a+w)g]da\\
&=\int_{a\in L^{''\ast}}  \psi(\tfrac{1}{2}\langle a-l, w+l\rangle)\psi( -\tfrac{1}{2}\langle a, l\rangle) \sigma''(l)\sigma''(-a)\mathcal{D}[f(a+w)g] da\\
&=\int_{a\in L^{''\ast}}  \psi(\tfrac{1}{2}\langle -l, w\rangle)\sigma''(l)\psi(\tfrac{1}{2}\langle a, w\rangle) \sigma''(-a)\mathcal{D}[f(a+w)g] da\\
&=\psi(\tfrac{1}{2}\langle -l, w\rangle)\sigma''(l)\mathfrak{i}^g_{L,L''}(f)(w).
 \end{align*}
 2) Assume $\supp(f) \subseteq C$, for some compact set $C$ of $W$. If $w\notin Cg^{-1}+L^{''\ast}$, then $f((a+w)g)=0$,  for any  $a\in L^{''\ast}$. Consequently, $\mathcal{D}[f((a+w)g)]=0$.   Hence $\supp (\mathfrak{i}^g_{L,L''}(f)) \subseteq L^{''\ast}+Cg^{-1}$.\\
 3)  If $f(a+w)=f(w)$, for $a$ belonging to  some open compact subgroup $K$ of $W$.  Then   $\mathcal{D}[f(a+w)]=\mathcal{D}[f(w)]$. Let $K_0$ be an open compact subgroup of $2(L''\cap Kg^{-1})$.  For any $l\in K_0$, we have: 
 \begin{align*}
&\mathfrak{i}^g_{L,L''}(f)(l+w)\\
&=\int_{a\in L^{''\ast}}  \psi(\tfrac{1}{2}\langle a, w+l\rangle)\sigma''(-a)\mathcal{D}[f((a+w+l)g) ]da\\
 &=\int_{a\in L^{''\ast}}  \psi(\tfrac{1}{2}\langle a, l\rangle)\psi(\tfrac{1}{2}\langle a, w\rangle)\sigma''(-a)\mathcal{D}[f((a+w)g) ]da\\
 &=\mathfrak{i}^g_{L,L''}(f)(w).
 \end{align*}
 
\end{proof}

\begin{lemma}\label{coomLL''}
Let $h=(w',t)\in \Ha(W)$. The following diagram is commutative:
\begin{equation}\label{eq1}
\begin{CD}
 \mathcal{V}_{L,\psi}@>\mathfrak{i}^g_{L,L''}>> \mathcal{V}_{L'',\psi}  \\
       @V\pi_{L,\psi}(h) VV @VV\pi_{L'',\psi}(h^{g^{-1}}) V  \\
 \mathcal{V}_{L,\psi}@>\mathfrak{i}^g_{L,L''}>> \mathcal{V}_{L'',\psi} 
\end{CD}
\end{equation} 
\end{lemma}
\begin{proof}
  \begin{align*}
 & \pi_{L'',\psi}(h^{g^{-1}})\mathfrak{i}^g_{L,L''}(f)(w)\\
 &=\pi_{L'',\psi}(w'g^{-1},t)\mathfrak{i}^g_{L,L''}(f)(w)\\
  &=\psi( t+\tfrac{1}{2}\langle w, w'g^{-1}\rangle)\mathfrak{i}^g_{L,L''}(f)(w+w'g^{-1})\\
 &=\psi( t+\tfrac{1}{2}\langle w, w'g^{-1}\rangle)\int_{a\in L^{''\ast}}\psi(\tfrac{1}{2}\langle a, w+w'g^{-1}\rangle)\sigma''(-a)\mathcal{D}[f((a+w+w'g^{-1})g)] da\\
 &=\int_{a\in L^{''\ast}}\psi( t+\tfrac{1}{2}\langle w, w'g^{-1}\rangle+\tfrac{1}{2}\langle a, w+w'g^{-1}\rangle)\sigma''(-a)\mathcal{D}[f(ag+wg+w')]da;
  \end{align*}
  \begin{align*}
 &\mathfrak{i}^g_{L,L''}[\pi_{L,\psi}(h)f](w)\\
 &=\int_{a\in L^{''\ast}}\psi(\tfrac{1}{2}\langle a, w\rangle)\sigma''(-a)\mathcal{D}[\pi_{L,\psi}(h)f((a+w)g)] da \\
 &=\int_{a\in L^{''\ast}}\psi(\tfrac{1}{2}\langle a, w\rangle)\psi( t+\tfrac{1}{2}\langle (a+w)g, w'\rangle)\sigma''(-a)\mathcal{D}[f((a+w)g+w')]da\\
  &=\int_{a\in L^{''\ast}}\psi( t+\tfrac{1}{2}\langle a, w\rangle+\tfrac{1}{2}\langle a+w, w'g^{-1}\rangle)\sigma''(-a)\mathcal{D}[f((a+w)g+w')]da\\
  &=\pi_{L'',\psi}(h^{g^{-1}})\mathfrak{i}^g_{L,L''}(f)(w).
  \end{align*}
\end{proof}
\begin{lemma}\label{chiLL''}
$\mathfrak{i}^g_{L,L''}$ sends $\mathcal{V}_{L_1,\psi_{\chi\circ \mathfrak{i}}}$ onto $\mathcal{V}_{L_1'',\psi_{\chi}}$.
\end{lemma}
\begin{proof}
Accord to \ref{DDD}(3), $\mathcal{D}$ sends $\mathcal{W}_{\chi\circ \mathfrak{i}}$ to $\mathcal{W}_{\chi}$.
 For $f\in \mathcal{V}_{L_1,\psi_{\chi\circ \mathfrak{i}}}$,  $l\in L^{''\ast}$, we have:
 \begin{align*}
&\mathfrak{i}^g_{L,L''}(f)(l+w)\\
&=\int_{a\in L^{''\ast}}\psi(\tfrac{1}{2}\langle a, w+l\rangle) \sigma''(-a)\mathcal{D}[f((a+w+l)g)] da\\
&=\int_{a\in L^{''\ast}}  \psi(\tfrac{1}{2}\langle a-l, w+l\rangle) \sigma''(-a+l)\mathcal{D}[f(a+w)g]da\\
&=\int_{a\in L^{''\ast}}  \psi(\tfrac{1}{2}\langle a-l, w+l\rangle)\psi( -\tfrac{1}{2}\langle a, l\rangle) \sigma''_{\chi}(l)\sigma''_{\chi}(-a)\mathcal{D}[f(a+w)g] da\\
&=\int_{a\in L^{''\ast}}  \psi(\tfrac{1}{2}\langle -l, w\rangle)\sigma''_{\chi}(l)\psi(\tfrac{1}{2}\langle a, w\rangle) \sigma''_{\chi}(-a)\mathcal{D}[f(a+w)g] da\\
&=\psi(\tfrac{1}{2}\langle -l, w\rangle)\sigma''_{\chi}(l)\mathfrak{i}^g_{L,L''}(f)(w).
 \end{align*}
\end{proof}
\begin{lemma}\label{Achi}
The following diagram is commutative:
\begin{equation}\label{eq1}
\begin{CD}
 \mathcal{V}_{L_1,\psi_{\chi_1\circ \mathfrak{i}}}@>\mathfrak{i}^g_{L,L''}>> \mathcal{V}_{L_1'',\psi_{\chi_1}}\\
       @V\mathcal{A}_{\chi_1\circ \mathfrak{i},\chi_2\circ \mathfrak{i}}VV @VV\mathcal{A}_{\chi_1, \chi_2} V  \\
\mathcal{V}_{L_1,\psi_{\chi_2\circ \mathfrak{i}}}@>\mathfrak{i}^g_{L,L''}>> \mathcal{V}_{L_1'',\psi_{\chi_2}} 
\end{CD}
\end{equation} 
\end{lemma}
\begin{proof}
 \begin{align}
&\mathfrak{i}^g_{L,L''}\mathcal{A}_{\chi_1\circ \mathfrak{i},\chi_2\circ \mathfrak{i}}(f)(w)\\
&=\int_{a\in L^{''\ast}}\psi(\tfrac{1}{2}\langle a, w\rangle) \sigma''(-a)\mathcal{D}[\mathcal{A}_{\chi_1\circ \mathfrak{i},\chi_2\circ \mathfrak{i}}(f)((a+w)g)] da \\
&=\int_{a\in L^{''\ast}}\psi(\tfrac{1}{2}\langle a, w\rangle) \sigma''_{\chi_2}(-a)\mathcal{D}[\mathcal{A}_{\chi_1\circ \mathfrak{i},\chi_2\circ \mathfrak{i}}(f)((a+w)g)] da \\
&=\int_{a\in L^{''\ast}}\psi(\tfrac{1}{2}\langle a, w\rangle) \sigma''_{\chi_2}(-a)\circ\mathcal{A}_{\chi_1,\chi_2}\mathcal{D}[f((a+w)g)] da \\
&=\int_{a\in L^{''\ast}}\psi(\tfrac{1}{2}\langle a, w\rangle) \mathcal{A}_{\chi_1,\chi_2}\circ \sigma''_{\chi_1}(-a) \mathcal{D}[f((a+w)g)] da \\
&= \mathcal{A}_{\chi_1,\chi_2}\Big(\int_{a\in L^{''\ast}}\psi(\tfrac{1}{2}\langle a, w\rangle) \sigma''_{\chi_1}(-a) \mathcal{D}[f((a+w)g)] da\Big) \\
&= \mathcal{A}_{\chi_1,\chi_2}\Big(\int_{a\in L^{''\ast}}\psi(\tfrac{1}{2}\langle a, w\rangle) \sigma''(-a) \mathcal{D}[f((a+w)g)] da\Big)\\
&=\mathcal{A}_{\chi_1, \chi_2}\mathfrak{i}^g_{L,L''}(f)(w).
\end{align}
\end{proof}
If $g=1$, we write  $\mathfrak{i}_{L,L''}$ for $\mathfrak{i}^g_{L,L''}$. Similarly, we can define $\mathfrak{i}_{L'',L}$.
\begin{lemma}\label{LL''}
Let $c=\mu(L^{\ast})\mu(\{a''\in L^{\ast}, a''d-a'' \in L\})$. Then  $\mathfrak{i}_{L'',L}\circ \mathfrak{i}_{L,L''}=c\id$.
\end{lemma}
\begin{proof}
Let $f\in \mathcal{V}_{L,\psi}$. Then:
\begin{align*}
 & \mathfrak{i}_{L'',L}\circ \mathfrak{i}_{L,L''}(f)(w)\\
 &=\int_{a\in L^{\ast}}\psi(\tfrac{1}{2}\langle a, w\rangle) \sigma(-a)\mathcal{D}^{-1}[\mathfrak{i}_{L'',L}(f) (a+w)] da\\
 &=\int_{a\in L^{\ast}}\psi(\tfrac{1}{2}\langle a, w\rangle) \sigma(-a)\mathcal{D}^{-1}\Big[\int_{a'\in L^{''\ast}}\psi(\tfrac{1}{2}\langle a', a+w\rangle) \sigma''(-a')\mathcal{D}[f(a'+ a+w)] da'\Big] da\\
 &=\int_{a\in L^{\ast}}\int_{a'\in L^{''\ast}} \psi(\tfrac{1}{2}\langle a, w\rangle) \psi(\tfrac{1}{2}\langle a', a+w\rangle)\sigma(-a)\sigma(-a' d^{-1}) f(a'+ a+w)] da' da\\
 &\stackrel{a''=a'd^{-1}}{=}\int_{a\in L^{\ast}}\int_{a''\in L^{\ast}} \psi(\tfrac{1}{2}\langle a, w\rangle) \psi(\tfrac{1}{2}\langle a''d, a+w\rangle)\sigma(-a)\sigma(-a'') f(a''d+ a+w) da'' da\\
 &=\int_{a\in L^{\ast}}\int_{a''\in L^{\ast}} \psi(\tfrac{1}{2}\langle a, w\rangle) \psi(\tfrac{1}{2}\langle a''d, a+w\rangle) \psi(\langle a,a''\rangle)\sigma(-a'') \sigma(-a)f(a''d+ a+w) da'' da\\
  &=\int_{a\in L^{\ast}}\int_{a''\in L^{\ast}} \psi(\langle a''d-a'', a\rangle)\psi(\tfrac{1}{2}\langle a''d, w\rangle) \sigma(-a'') f(a''d+w) da'' da
    \end{align*}
 \begin{align*}
    &=\int_{a\in L^{\ast}}\int_{a''\in L^{\ast}} \psi(\langle a''d-a'', a\rangle) \psi(\tfrac{1}{2}\langle a''d-a'', w\rangle) \psi(\tfrac{1}{2}\langle -a'', a''d\rangle) f(a''d-a''+w) da'' da\\
     &=\mu(L^{\ast})\int_{a''\in L^{\ast}, a''d-a'' \in L} \psi(\tfrac{1}{2}\langle a''d-a'', w\rangle) \psi(\tfrac{1}{2}\langle -a'', a''d\rangle) f(a''d-a''+w) da'' \\
        &=\mu(L^{\ast})\int_{a''\in L^{\ast}, a''d-a'' \in L} \psi(\tfrac{1}{2}\langle -a'', a''d\rangle) \sigma(a''d-a'')  f(w) da'' \\
         &\stackrel{\textrm{ Lemma }\ref{Bll}}{=}\mu(L^{\ast})\int_{a''\in L^{\ast}, a''d-a'' \in L} \psi(\tfrac{1}{2}\langle -a'', a''d\rangle) \psi(-\tfrac{1}{2}B(a''d-a'',a''d-a''))  f(w) da'' \\
        &=\mu(L^{\ast})\int_{a''\in L^{\ast}, a''d-a'' \in L} \psi(\tfrac{1}{2}B(a''d, a''d-a'')) \psi(\tfrac{1}{2}B(a''d-a'',a''))  f(w) da'' \\
     &=\mu(L^{\ast})\int_{a''\in L^{\ast}, a''d-a'' \in L} \psi(B(a''d-a'', a''))   f(w) da'' \\
  &=\mu(L^{\ast})\mu(\{a''\in L^{\ast}, a''d-a'' \in L\})f(w).
  \end{align*}
\end{proof}
\begin{corollary}
$\mathfrak{i}_{L,L''}$ is  bijective.
\end{corollary}
\section{The non-archimedean local field case II}\label{pn2}
Keep the above notations. We consider a special sublattice $L$ of $L_1$.  Under the basis $\{e_1, \cdots, e_m, e_1^{\ast}, \cdots, e_m^{\ast}\}$, we identity $\GSp(W)$ with $\GSp_{2m}(F)$.  Let $d_L=\begin{bmatrix}
g_1& 0\\
0& g_2\end{bmatrix}\in \GSp_{2m}(F)$, such that $(L_1\cap X)g_1 \subseteq L_1\cap X$, $(L_1\cap X^{\ast})g_2 \subseteq L_1\cap X^{\ast}$. Let $$L=L_1d_L.$$ Then $L \subseteq L_1$ and  $L=L\cap X\oplus L\cap X^{\ast}$.
 Let us define:
\begin{itemize}
\item $L'=\left\{ \begin{array}{cl}
2 L\cap X\oplus \tfrac{1}{2} L\cap X^{\ast} & \textrm{ if } 2\mid e,\\
\tfrac{1}{2} L\cap X\oplus 2 L\cap X^{\ast} & \textrm{ if } 2\nmid e.\end{array}\right.$ 
\item $L^{0}= L\cap L' $.
 \end{itemize}
 Then  $Ld_L^{-1}=L_1$, $L' d_L^{-1}= L_1'$, and $L^{0} d_L^{-1} =L_1^0$, $L^{0}\subseteq L_1^0$.
Let us  write:
\begin{itemize}
\item  $L'=L d'$.
\item $t=\lambda_{d_L}$, the similitude factor  of $d_L$.
 \end{itemize}
 \begin{lemma}
 $L^{\ast}=L_1 d_L t^{-1}$ and $L^{'\ast}=L'_1 d_L t^{-1}$.
 \end{lemma}
 \begin{proof}
 $a^{\ast} \in L^{\ast}$ iff $\psi(\langle a^{\ast}, a\rangle ) =1$ for all $a=a' d_L \in L$ iff $\psi(\langle ta^{\ast}d^{-1}_L, a'\rangle ) =1$ for all $a' \in L_1$ iff $ta^{\ast}d^{-1}_L \in L_1$ iff $a^{\ast} \in L_1 d_L t^{-1}$. The proof of Part (2) is  similar. 
 \end{proof}
 \subsection{On the Iwahori group $d_L^{-1}Id_L$}
 
We take the foregoing $L''$ to be $L'$.  Write  $g=d_L^{-1}g_0d_L\in d_L^{-1}I d_L$. Let  $\mathcal{A}_g=\{ a\in L^{\ast} \mid  ad'g-a\in L\}$.
\begin{lemma}\label{trial1}
For $a\in  \mathcal{A}_g$, $\psi(\tfrac{1}{2}\langle -a, ad'g\rangle-\tfrac{1}{2}B(ad'g-a,ad'g-a))=1$.
\end{lemma}
\begin{proof} 
 If $a\in \mathcal{A}_g$, then $a=a'd_L t^{-1}$, for some $a'\in L_1$. Hence  
$$ad'g-a\in L \Longleftrightarrow a' d_L t^{-1} d' d_L^{-1}g_0d_L -a'd_L t^{-1} \in  L_1d_L \Longleftrightarrow a' d' g_0-a' \in L_1 t .$$
Since $ L_1 t \subseteq L_1$, $a' d' g_0\in L_1$. So $a'd'\in L_1\cap L_1'=L_1^0$, $ad'\in L^{'\ast} \cap L^{\ast}$. Similarly, $a\in L^{\ast} \cap L^{ \ast} d'^{-1}$.  
\begin{itemize}
\item[(1)] If $2\mid e$, $$ad' g \subseteq  L^{'\ast} \cap L^{\ast}, \quad \psi(-\tfrac{1}{2}B(ad' g,  ad'g-a))=1,  \quad  \psi(\tfrac{1}{2}B(ad'g-a, a ))=1.$$ 
 \begin{align*}
& \tfrac{1}{2}\langle -a, ad'g\rangle-\tfrac{1}{2}B(ad'g-a,ad'g-a)\\
&=-\tfrac{1}{2}\langle a, ad'g-a\rangle-\tfrac{1}{2}B(ad'g-a,ad'g-a)\\
&=-\tfrac{1}{2}B( a, ad'g-a)+ \tfrac{1}{2}B( ad'g-a,a) -\tfrac{1}{2}B(ad'g-a,ad'g-a)\\
&=\tfrac{1}{2}B( ad'g-a,a)-\tfrac{1}{2}B(ad'g,ad'g-a).
\end{align*}
\item[(2)] If $2\nmid e$, $$ad' g \subseteq  L^{'\ast} \cap L^{\ast}, \quad \psi(-\tfrac{1}{2}B(a,  ad'g-a))=1,  \quad  \psi(-\tfrac{1}{2}B(ad'g-a, ad'g ))=1.$$
\end{itemize}
 \begin{align*}
& \tfrac{1}{2}\langle -a, ad'g\rangle-\tfrac{1}{2}B(ad'g-a,ad'g-a)\\
& =\tfrac{1}{2}\langle  ad'g, a\rangle-\tfrac{1}{2}B(ad'g-a,ad'g-a)\\
&=\tfrac{1}{2}\langle ad'g- a, a\rangle-\tfrac{1}{2}B(ad'g-a,ad'g-a)\\
&=\tfrac{1}{2}B( ad'g- a, a)-\tfrac{1}{2}B(  a, ad'g- a) -\tfrac{1}{2}B(ad'g-a,ad'g-a)\\
&=B( ad'g-a,a)- \tfrac{1}{2}B(ad'g-a,ad'g) -\tfrac{1}{2}B(  a, ad'g- a).
\end{align*}
\end{proof}
  Let $c_g=\mu(L^{\ast})\mu(\mathcal{A}_g)$. 
\begin{lemma}\label{LL'g}
 For $g\in d_L^{-1}Id_L$,  $\mathfrak{i}^{g^{-1}}_{L',L}\circ \mathfrak{i}^{g}_{L,L'}=c_g\id$.
\end{lemma}
\begin{proof}
Let $f\in \mathcal{V}_{L,\psi}$. 
\begin{align*}
 & \mathfrak{i}^{g^{-1}}_{L',L}\circ \mathfrak{i}^{g}_{L,L'}(f)(w)\\
 &=\int_{a''\in L^{\ast}}\psi(\tfrac{1}{2}\langle a'', w\rangle) \sigma(-a'')\mathcal{D}^{-1}[\mathfrak{i}^g_{L,L'}(f)( (a''+w)g^{-1})] da''\\
 &=\int_{a''\in L^{\ast}}\psi(\tfrac{1}{2}\langle a'', w\rangle) \sigma(-a'')\mathcal{D}^{-1}\Big[\int_{a'\in L^{'\ast}}\psi(\tfrac{1}{2}\langle a', (a''+w)g^{-1}\rangle) \sigma'(-a')\mathcal{D}[f(a'g+ a''+w)] da'\Big] da''\\
 &=\int_{a''\in L^{\ast}}\int_{a'\in L^{'\ast}} \psi(\tfrac{1}{2}\langle a'', w\rangle) \psi(\tfrac{1}{2}\langle a', (a''+w)g^{-1}\rangle)\sigma(-a'')\sigma(-a' d^{'-1}) f(a'g+ a''+w)] da' da''
   \end{align*}
  \begin{align*} 
 &\stackrel{a=a'd^{'-1}}{=}\int_{a''\in L^{\ast}}\int_{a\in L^{\ast}} \psi(\tfrac{1}{2}\langle a'', w\rangle) \psi(\tfrac{1}{2}\langle ad', (a''+w)g^{-1}\rangle)\sigma(-a'')\sigma(-a) f(ad'g+ a''+w) da'' da\\
 &=\int_{a''\in L^{\ast}}\int_{a\in L^{\ast}} \psi(\tfrac{1}{2}\langle a'', w\rangle) \psi(\tfrac{1}{2}\langle ad'g, a''+w\rangle) \psi(\langle a'',a\rangle)\sigma(-a) \sigma(-a'')f(ad'g+ a''+w) da'' da\\
  &=\int_{a''\in L^{\ast}}\int_{a\in L^{\ast}} \psi(\langle ad'g-a, a''\rangle)\psi(\tfrac{1}{2}\langle ad'g, w\rangle) \sigma(-a) f(ad'g+w) da'' da\\
    &=\int_{a''\in L^{\ast}}\int_{a\in L^{\ast}} \psi(\langle ad'g-a, a''\rangle) \psi(\tfrac{1}{2}\langle ad'g-a, w\rangle) \psi(\tfrac{1}{2}\langle -a, ad'g\rangle) f(ad'g-a+w) da'' da\\
     &=\mu(L^{\ast})\int_{a\in L^{\ast}, ad'g-a \in L} \psi(\tfrac{1}{2}\langle ad'g-a, w\rangle) \psi(\tfrac{1}{2}\langle -a, ad'g\rangle) f(ad'g-a+w) da\\
       &=\mu(L^{\ast})\int_{a\in L^{\ast}, ad'g-a \in L} \psi(\tfrac{1}{2}\langle -a, ad'g\rangle) \sigma(ad'g-a)f(w) da\\
     &\stackrel{ \textrm{ Lemma }\ref{Bll}}{=}\mu(L^{\ast})\int_{a\in L^{\ast}, ad'g-a \in L} \psi(\tfrac{1}{2}\langle -a, ad'g\rangle) \psi(-\tfrac{1}{2}B(ad'g-a,ad'g-a))  da  f(w)\\
     &\stackrel{ \textrm{ Lemma }\ref{trial1}}{=} c_gf(w).
  \end{align*}
\end{proof}
\begin{corollary}\label{gLL'}
$\mathfrak{i}^{g}_{L,L'}$ is bijective.
\end{corollary}
\subsubsection{} For $g\in d_L^{-1}Id_L$, let us define $M[g]\in \End(\mathcal{V}_{L,\psi})$ as follows:
\begin{align}\label{gactionsigma31}
M[g]=  \mathfrak{i}_{L',L} \circ\mathfrak{i}^g_{L,L'} .
\end{align}
\begin{lemma}
$M[g]$ is bijective. 
\end{lemma}
\begin{proof}
It follows from the above lemma.
\end{proof}
\subsection{On the Iwahori-Weyl  group $d_L^{-1}\mathfrak{W}^{\textrm{eaff}}d_L$ I} For $g=d_L^{-1} g' d_L\in d_L^{-1}\mathfrak{W}^{\textrm{eaff}}d_L$, let  $\mathcal{B}_g=\{ a\in L^{\ast} \mid  ag-a\in L\}$.
By Section \ref{decom}, for any $g'$, let us write 
$$g'=\omega_s \begin{bmatrix}
d& 0 \\0 &  d^{-1} \end{bmatrix}  \omega_S, \quad g_0=\begin{bmatrix}
d& 0 \\0 &  d^{-1} \end{bmatrix}  \omega_S,$$
for some  $S \subseteq \{1, \cdots, m\}$.  

\begin{itemize}
\item If  $a\in \mathcal{B}_g$, then $a=a'd_L t^{-1}$, for some $a'\in L_1$. Hence  
$$ag-a\in L \Longleftrightarrow a' d_L t^{-1}  d_L^{-1}g'd_L -a'd_L t^{-1} \in  L_1d_L \Longleftrightarrow a'  g'-a' \in L_1 t .$$
\end{itemize}
\begin{lemma}\label{trial1}
If $t=\lambda_{d_L}\in \mathcal{O}\setminus \mathcal{P}$ or $ \omega_S=1$,  then $\psi(\tfrac{1}{2}\langle -a, ag\rangle-\tfrac{1}{2}B(ag-a,ag-a))=1$, for $a\in \mathcal{B}_g$.
\end{lemma}
\begin{proof} 
 \begin{align*}
& \tfrac{1}{2}\langle -a, ag\rangle-\tfrac{1}{2}B(ag-a,ag-a)\\
&=-\tfrac{1}{2}B(a, ag)+\tfrac{1}{2}B(ag, a)-\tfrac{1}{2}B(ag,ag) -\tfrac{1}{2}B(a,a)+\tfrac{1}{2}B(ag,a) +\tfrac{1}{2}B(a,ag)\\
&=B(ag, a-ag)+\tfrac{1}{2}B(ag,ag) -\tfrac{1}{2}B(a,a)\\
&=B(ag, a-ag)+q_g(a).
\end{align*}
 1) In this first case,  $L^{\ast}=L$, so $\psi(q_g(a))=1$. \\
2) In the second case,  $\psi(q_g(a))=\psi(t^{-1}q_{g'}(a'))=1$. Therefore, 
 $$\psi(\tfrac{1}{2}\langle -a, ag\rangle-\tfrac{1}{2}B(ag-a,ag-a))=\psi(B(ag, a-ag)) \psi(q_g(a))=1.$$
\end{proof}
\begin{remark}
Let $g_1=d_L^{-1} \omega_s d_L$, $g_2=d_L^{-1} g_0 d_L$. Then $q_{g_1}(w)=0$ and $\psi(q_g(w))=\psi(q_{g_2}(wg_1))$. So it reduces to discussing $g_2$.
\end{remark}
In the following, we will consider the remaining cases that $\omega_S\neq 1$ and $t=\lambda_{d_L}\in \mathcal{P}$. 
\subsubsection{The case $\dim W=2$} Assume that $\dim W=2$. Write $g_0=  \begin{bmatrix}
d& 0 \\0 &  d^{-1} \end{bmatrix}\begin{bmatrix}0& -1 \\1 &  0 \end{bmatrix}  $.
For $a'=xe_1+ ye_1^{\ast}$, $a' g_0= d^{-1} ye_1- dxe_1^{\ast}$, $a' g_0-a'=(d^{-1} y-x)e_1+(- dx-y)e_1^{\ast}\in L_1t$. Then
$$d^{-1} y-x=tl_1, \quad \quad - dx-y=tl_2,$$
for some $l_1e_1+l_2e_1^{\ast}\in L_1$. Hence:
$$y=-\tfrac{1}{2} tl_2+ \tfrac{1}{2} tdl_1, \quad x=-\tfrac{1}{2d} tl_2-\tfrac{1}{2} tl_1.$$
 \begin{align}
&\psi(\tfrac{1}{2}\langle -a, ag\rangle-\tfrac{1}{2}B(ag-a,ag-a))\\
&=\psi(t^{-1}B(a'g_0, a'-a'g_0)+t^{-1}q_{g_0}(a')).
\end{align}
Note that $a'g_0-a'\in L_1t$, so $a'g_0\in L_1$. Hence $\psi(t^{-1}B(a'g_0, a'-a'g_0))=1$. By the equality (\ref{B}), 
\begin{align}\label{g0adt}
&\psi(\tfrac{1}{t}q_{g_0}(a'))=\psi(-\tfrac{1}{t}xy)\\
&=\psi([\tfrac{1}{2}(-l_2+ dl_1)] [\tfrac{t}{2d} (l_2+dl_1)])\\
&=\psi(-\tfrac{t}{4d}[l_2^2-(dl_1)^2])\label{qcal41}\\
&=\psi(-\tfrac{1}{4dt}[(tdl_1+2dx)^2-(tdl_1)^2])\\
&=\psi(-dl_1x -\tfrac{d}{t}x^2)\label{qcal51}\\
&=\psi(-\tfrac{1}{4dt}[(l_2t)^2-(2y+2tl_2)^2])\\
&=\psi(\tfrac{y^2}{dt}+\tfrac{1}{d}yl_2).\label{qcal61}
\end{align}
Let $\nu_2$ denote the $2$-adic valuation on $F$. Let $n_t=\nu_2(t)\geq 1$, $n_d=\nu_2(d)$. 
\begin{lemma}\label{2mide}
Assume $2\mid e$.
\begin{itemize}
\item[(i)]  If $n_d-n_t\geq 0$ or $n_d+n_t\leq 0$,  then $\psi(\tfrac{1}{2}\langle -a, ag\rangle-\tfrac{1}{2}B(ag-a,ag-a))=1$, for $a\in \mathcal{B}_g$.
\item[(ii)] If $n_d-n_t\leq -2$ and $n_d+n_t\geq 2$, then $\psi(\tfrac{1}{2}\langle -a, ag\rangle-\tfrac{1}{2}B(ag-a,ag-a))=1$, for $a\in \mathcal{B}_g$.
\end{itemize} 
\end{lemma}
\begin{proof}
ia) If $n_d-n_t\geq 0$, then $\tfrac{d}{t}\in \mathcal{O}$,  and  $ d\in t\mathcal{O} \subseteq \mathcal{P}$.   By (\ref{qcal51}), $\psi(\tfrac{1}{t}q_{g_0}(a'))=\psi(-dl_1x -\tfrac{d}{t}x^2)=1$.\\
ib) If $n_d+n_t\leq 0$, then $\tfrac{1}{dt}\in   \mathcal{O}$, and $\tfrac{1}{d}\in   t\mathcal{O} \subseteq \mathcal{P}$. By (\ref{qcal61}), $\psi(\tfrac{1}{t}q_{g_0}(a'))=\psi(\tfrac{y^2}{dt}+\tfrac{1}{d}yl_2)=1$.\\
ii) In this case, $\tfrac{t}{4d}\in \mathcal{O}$ and $\tfrac{td}{4}\in \mathcal{O}$.   By (\ref{qcal41}), $\psi(\tfrac{1}{t}q_{g_0}(a'))=\psi(-\tfrac{t}{4d}l_2^2) \psi(\tfrac{dt}{4}l_1^2)=1$.
\end{proof}
\begin{lemma}\label{2nmide}
Assume $2\nmid e$.
\begin{itemize}
\item[(i)]  If $n_d-n_t\geq -1$ or $n_d+n_t\leq -1$,  then $\psi(\tfrac{1}{2}\langle -a, ag\rangle-\tfrac{1}{2}B(ag-a,ag-a))=1$, for $a\in \mathcal{B}_g$.
\item[(ii)] If $n_d-n_t\leq -3$ and $n_d+n_t\geq 1$, then $\psi(\tfrac{1}{2}\langle -a, ag\rangle-\tfrac{1}{2}B(ag-a,ag-a))=1$, for $a\in \mathcal{B}_g$.
\end{itemize} 
\end{lemma}
\begin{proof}
In this case,  $L_1\cap X=\mathcal{P}^{\tfrac{e+1}{2}}e_1 $, $L_1\cap X^{\ast}=\mathcal{P}^{\tfrac{e-1}{2}}e_1^{\ast}$. Note that $x,l_1\in \mathcal{P}^{\tfrac{e+1}{2}}$ and $y,l_2\in \mathcal{P}^{\tfrac{e-1}{2}}$.\\
ia) If $n_d-n_t\geq -1$, then $\tfrac{d}{t}\in \mathcal{P}^{-1}$,  and  $ d\in t\mathcal{P}^{-1} \subseteq \mathcal{O}$.   By (\ref{qcal51}), $\psi(\tfrac{1}{t}q_{g_0}(a'))=\psi(-dl_1x -\tfrac{d}{t}x^2)=1$.\\
ib) If $n_d+n_t\leq -1$, then $\tfrac{1}{dt}\in   \mathcal{P}$, and $\tfrac{1}{d}\in   t\mathcal{P} \subseteq \mathcal{P}^2$. By (\ref{qcal61}), $\psi(\tfrac{1}{t}q_{g_0}(a'))=\psi(\tfrac{y^2}{dt}+\tfrac{1}{d}yl_2)=1$.\\
ii) In this case, $\tfrac{t}{4d}\in \mathcal{P}$ and $\tfrac{td}{4}\in \mathcal{P}^{-1}$.   By (\ref{qcal41}), $\psi(\tfrac{1}{t}q_{g_0}(a'))=\psi(-\tfrac{t}{4d}l_2^2) \psi(\tfrac{dt}{4}l_1^2)=1$.
\end{proof}
\subsubsection{The case $\dim W=2$ II} We consider the exceptional cases mentioned in Lemmas \ref{2mide} and \ref{2nmide}.  
Note that $n_t\geq 1$.  
\begin{itemize}
\item If $2\mid e$,  the exceptional case is that $\{ n_d-n_t=-1$ or $n_d+n_t=1\}$. 
\item If $2\nmid e$,  the exceptional case is that $\{ n_d-n_t=-2$ or $n_d+n_t=0\}$.  
\end{itemize}
Let:   $$\mathbbm{2}\footnote{This notation was originally dedicated to  Professor  Guy Henniart on his  birthday.}=\begin{bmatrix} 2&0\\ 0&  \tfrac{1}{2}\end{bmatrix},  $$ 
$$d''=\left\{\begin{array}{cl} 
\mathbbm{2} & \textrm{ if } 2\mid e, n_d-n_t=-1,\\
 \mathbbm{2}^{-1} &\textrm{ if } 2\mid e, n_d+n_t=1,\\
 \mathbbm{2} & \textrm{ if } 2\nmid e,  n_d-n_t=-2, \\
 \mathbbm{2}^{-1} & \textrm{ if } 2\nmid e, n_d+n_t=0.\end{array}\right.$$
$$\mathcal{C}_g=\{ a\in L^{\ast} \mid  ad''g-a\in L\}.$$
\begin{lemma}\label{1112}
\begin{itemize}
\item[(i)] Assume $2\mid e$.  If  $n_d-n_t=-1$ or $n_d+n_t=1$,  then $\psi(\tfrac{1}{2}\langle -a, ad''g\rangle-\tfrac{1}{2}B(ad''g-a,ad''g-a))=1$, for $a\in \mathcal{C}_g$.
\item[(ii)] Assume $2\nmid e$. If  $ n_d-n_t=-2$ or $n_d+n_t=0$,  then $\psi(\tfrac{1}{2}\langle -a, ad''g\rangle-\tfrac{1}{2}B(ad''g-a,ad''g-a))=1$, for $a\in \mathcal{C}_g$.
\end{itemize} 
\end{lemma}
\begin{proof}
Similarly, we have:
\begin{align*}
& \tfrac{1}{2}\langle -a, ad''g\rangle-\tfrac{1}{2}B(ad''g-a,ad''g-a)\\
&=B(ad''g, a-ad''g)+q_g(ad').
\end{align*}
Note that $ad''g-a\in L\subseteq L^{\ast}$. Hence $ad''g\in L^{\ast}$,  $\psi(B(ad''g, a-ad''g))=1$. Write  $a=a'd_L t^{-1}$, for some $a'\in L_1$. Then: 
$$ad''g-a\in L \Longleftrightarrow a' d''d_L t^{-1}  d_L^{-1}g_0d_L -a'd_L t^{-1} \in  L_1d_L \Longleftrightarrow a'd''  g_0-a' \in L_1 t .$$
$$\psi(q_g(ad'))= \psi(t^{-1}q_{g_0}(a'd'')).$$
Recall $g_0=  \begin{bmatrix}
d& 0 \\0 &  d^{-1} \end{bmatrix}\begin{bmatrix}0& -1 \\1 &  0 \end{bmatrix}  $. 
\begin{itemize}
\item Assume that $2\mid e, n_d-n_t=-1$ or $2\nmid e, n_d+n_t=0$. \\
  For $a'=xe_1+ ye_1^{\ast}$, $a' d'' g_0= (2d)^{-1} ye_1- 2dxe_1^{\ast}$, $a' g_0-a'=[(2d)^{-1} y-x]e_1+(- 2dx-y)e_1^{\ast}\in L_1t$. Then
$$(2d)^{-1} y-x=tl_1, \quad \quad - 2dx-y=tl_2,$$
for some $l_1e_1+l_2e_1^{\ast}\in L_1$. Hence:
$$y=-\tfrac{1}{2} tl_2+ tdl_1, \quad x=-\tfrac{1}{4d} tl_2-\tfrac{1}{2} tl_1.$$
\begin{align}
&\psi(\tfrac{1}{t}q_{g_0}(a'd''))=\psi(-\tfrac{1}{t}xy)\\
&=\psi([\tfrac{1}{2}(l_2- 2dl_1)] [-\tfrac{t}{4d} (l_2+2dl_1)])\\
&=\psi(-\tfrac{t}{8d}[l_2^2-4(dl_1)^2])\label{qcal71}\\
&=\psi(-\tfrac{1}{8dt}[(2tdl_1+4dx)^2-4(tdl_1)^2])\\
&=\psi(-2dl_1x -\tfrac{2d}{t}x^2)\label{qcal81}\\
&=\psi(-\tfrac{1}{8dt}[(l_2t)^2-(2y+tl_2)^2])\\
&=\psi(\tfrac{y^2}{2dt}+\tfrac{1}{2d}yl_2).\label{qcal91}
\end{align}
ia) If $2\mid e$ and  $n_d-n_t+1=0$, then $\tfrac{2d}{t}\in \mathcal{O}$ and $d\in \tfrac{t}{2} \mathcal{O} \subseteq  \mathcal{O}$.  By (\ref{qcal81}), $\psi(\tfrac{1}{t}q_{g_0}(a'))=\psi(-2dl_1x -\tfrac{2d}{t}x^2)=1$.\\
iia) If $2\nmid e$ and  $n_d-n_t=-2$, $\tfrac{2d}{t}\in \mathcal{P}^{-1}$ and $2d\in t \mathcal{P}^{-1} \subseteq  \mathcal{O}$.  By (\ref{qcal81}), $\psi(\tfrac{1}{t}q_{g_0}(a'))=\psi(-2dl_1x -\tfrac{2d}{t}x^2)=1$.\\
\item  For $a'=xe_1+ ye_1^{\ast}$, $a' d'' g_0= 2d^{-1} ye_1- \tfrac{d}{2}xe_1^{\ast}$, $a' g_0-a'=[2d^{-1} -x]e_1+(- \tfrac{d}{2}x-y)e_1^{\ast}\in L_1t$. Then
$$2d^{-1} y-x=tl_1, \quad \quad - \tfrac{d}{2}x-y=tl_2,$$
for some $l_1e_1+l_2e_1^{\ast}\in L_1$. Hence:
$$y=-\tfrac{1}{2} tl_2+ \tfrac{1}{4}tdl_1, \quad x=-\tfrac{1}{d} tl_2-\tfrac{1}{2} tl_1.$$
\begin{align}
&\psi(\tfrac{1}{t}q_{g_0}(a'd''))=\psi(-\tfrac{1}{t}xy)\\
&=\psi([\tfrac{1}{2}(l_2- \tfrac{1}{2}dl_1)] [-\tfrac{t}{d} (l_2+\tfrac{1}{2}dl_1)])\\
&=\psi(-\tfrac{t}{2d}[l_2^2-\tfrac{1}{4}(dl_1)^2])\label{qcal101}\\
&=\psi(-\tfrac{1}{2dt}[(\tfrac{1}{2}tdl_1+dx)^2-\tfrac{1}{4}(tdl_1)^2])\\
&=\psi(-\tfrac{1}{2}dl_1x -\tfrac{d}{2t}x^2)\label{qcal111}\\
&=\psi(-\tfrac{1}{2dt}[(l_2t)^2-(2y+tl_2)^2])\\
&=\psi(\tfrac{2y^2}{dt}+\tfrac{2}{d}yl_2).\label{qcal121}
\end{align}
ib) If  $2\mid e$ and $n_d+n_t=1$, then $\tfrac{2}{dt}\in \mathcal{O}$ and $\tfrac{1}{d}\in \tfrac{t}{2} \mathcal{O} \subseteq  \mathcal{O}$. By (\ref{qcal91}), 
$\psi(\tfrac{1}{t}q_{g_0}(a'))=\psi(\tfrac{2y^2}{dt}+\tfrac{2}{d}yl_2)=1$.\\
iib) If  $2\nmid e$ and $n_d+n_t=0$, then $\tfrac{1}{dt}\in \mathcal{O}$ and $\tfrac{1}{d}\in t \mathcal{O} \subseteq  \mathcal{P}$. By (\ref{qcal91}), 
$\psi(\tfrac{1}{t}q_{g_0}(a'))=\psi(\tfrac{2y^2}{dt}+\tfrac{2}{d}yl_2)=1$.
\end{itemize}
\end{proof}
\subsubsection{The general case}\label{thegeneralcase}
 Let $W_i=Fe_i+Fe_i^{\ast}$. Then there exists a group homomorphism:
$$\iota: \SL_2(W_1) \times  \cdots \times \SL_2(W_m) \longrightarrow \Sp(W).$$
Then its image contains $\mathfrak{W}_0^{\textrm{aff}}$. Recall that $g=d_L^{-1} g'  d_L$ and  $g'=\omega_s g_0$.  Assume that $\iota([g_0^{(1)}, \cdots, g_0^{(m)}])=g_0$.  Correspondingly, let $L_i=\mathcal{P}^{[\tfrac{e+1}{2}]}e_i\oplus \mathcal{P}^{[\tfrac{e}{2}]}e_i^{\ast}$. Let $$\mathcal{I}_g=\{ i\mid  g_0^{(\omega_s(i))} \textrm{ satisfies the condition of Lemma \ref{1112} }\}.$$ 
If $i\in \mathcal{I}_g$, we let 
$$L''_i=L_i d''=\left\{ \begin{array}{cl}
2\mathcal{P}^{[\tfrac{e+1}{2}]}e_i\oplus \tfrac{1}{2}\mathcal{P}^{[\tfrac{e}{2}]}e_1^{\ast} & \textrm{ if }  2\mid e, n_d-n_t=-1,\\
\tfrac{1}{2} \mathcal{P}^{[\tfrac{e+1}{2}]}e_i\oplus 2\mathcal{P}^{[\tfrac{e}{2}]}e_1^{\ast}  & \textrm{ if } 2\mid e,  n_d+n_t=1,\\
2\mathcal{P}^{[\tfrac{e+1}{2}]}e_i\oplus \tfrac{1}{2}\mathcal{P}^{[\tfrac{e}{2}]}e_1^{\ast} & \textrm{ if }  2\nmid e, n_d-n_t=-2,\\
\tfrac{1}{2} \mathcal{P}^{[\tfrac{e+1}{2}]}e_i\oplus 2\mathcal{P}^{[\tfrac{e}{2}]}e_1^{\ast}  & \textrm{ if } 2\nmid e, n_d+n_t=0.\end{array}\right.$$ 
Let $$L_g''=\sum_{i\in \mathcal{I}_g} L'_i + \sum_{i\notin \mathcal{I}_g}L_i.$$
Write $$L_g''=Ld_g'.$$
\subsubsection{Assume $\mathcal{I}_g=\emptyset$}
Let $g\in d_L^{-1}\mathfrak{W}^{\textrm{eaff}}d_L$.  
\begin{lemma}\label{LLIgin}
 If $\mathcal{I}_g=\emptyset$, $\mathfrak{i}^{g^{-1}}_{L,L}\circ \mathfrak{i}^{g}_{L,L}=c_g\id$, for $c_g=\mu(L^{\ast})\mu(\mathcal{B}_g)$.
\end{lemma}
\begin{proof}
Let $f\in \mathcal{V}_{L,\psi}$. 
\begin{align*}
 & \mathfrak{i}^{g^{-1}}_{L,L}\circ \mathfrak{i}^{g}_{L,L}(f)(w)\\
 &=\int_{a'\in L^{\ast}}\psi(\tfrac{1}{2}\langle a', w\rangle) \sigma(-a')\mathcal{D}^{-1}[\mathfrak{i}^g_{L,L}(f)( (a'+w)g^{-1})] da'\\
 &=\int_{a'\in L^{\ast}}\psi(\tfrac{1}{2}\langle a', w\rangle) \sigma(-a')\mathcal{D}^{-1}\Big[\int_{a\in L^{\ast}}\psi(\tfrac{1}{2}\langle a, (a'+w)g^{-1}\rangle) \sigma'(-a)\mathcal{D}[f(ag+ a'+w)] da\Big] da'\\
 &=\int_{a'\in L^{\ast}}\int_{a\in L^{\ast}} \psi(\tfrac{1}{2}\langle a', w\rangle) \psi(\tfrac{1}{2}\langle a, (a'+w)g^{-1}\rangle)\sigma(-a')\sigma(-a ) f(ag+ a'+w) da da'
   \end{align*}
  \begin{align*} 
 &=\int_{a'\in L^{\ast}}\int_{a\in L^{\ast}} \psi(\tfrac{1}{2}\langle a', w\rangle) \psi(\tfrac{1}{2}\langle ag, a'+w\rangle) \psi(\langle a',a\rangle)\sigma(-a) \sigma(-a')f(ag+ a'+w) da' da\\
  &=\int_{a'\in L^{\ast}}\int_{a\in L^{\ast}} \psi(\langle ag-a, a'\rangle)\psi(\tfrac{1}{2}\langle ag, w\rangle) \sigma(-a) f(ag+w) da' da\\
    &=\int_{a'\in L^{\ast}}\int_{a\in L^{\ast}} \psi(\langle ag-a, a'\rangle) \psi(\tfrac{1}{2}\langle ag-a, w\rangle) \psi(\tfrac{1}{2}\langle -a, ag\rangle) f(ag-a+w) da' da\\
     &=\mu(L^{\ast})\int_{a\in L^{\ast}, ag-a \in L} \psi(\tfrac{1}{2}\langle ag-a, w\rangle) \psi(\tfrac{1}{2}\langle -a, ag\rangle) f(ag-a+w) da \\
     &=\mu(L^{\ast})\Big[\int_{a\in L^{\ast}, ag-a \in L} \psi(\tfrac{1}{2}\langle -a, ag\rangle) \psi(-\tfrac{1}{2}B(ag-a,ag-a))  da \Big] f(w)\\
     &\stackrel{ \textrm{ Lemma }\ref{trial1}}{=} c_gf(w).
  \end{align*}
\end{proof}

\begin{corollary}\label{gLL2}
$\mathfrak{i}^{g}_{L,L}$ is bijective.
\end{corollary}
 For such $g\in d_L^{-1}\mathfrak{W}^{\textrm{eaff}}d_L$, let us define $M[g]\in \End(\mathcal{V}_{L,\psi})$ as follows:
\begin{align}\label{gactionsigmaaff}
M[g]= \mathfrak{i}^g_{L,L} .
\end{align}
\subsubsection{Assume $\mathcal{I}_g\neq \emptyset$} 
 Let $g\in d_L^{-1}\mathfrak{W}^{\textrm{eaff}}d_L$.  Let $\mathcal{C}_g=\{ a\in L^{\ast}\mid a d_g'' g-a\in L\}.$  Let $c_g=\mu(L^{\ast})\mu(\mathcal{C}_g)$. 
\begin{lemma} 
 If $\mathcal{I}_g\neq \emptyset$, $\mathfrak{i}^{g^{-1}}_{L''_g,L}\circ \mathfrak{i}^{g}_{L,L''_g}=c_g\id$.
\end{lemma}
\begin{proof}
See the proof of Lemma \ref{LL'g}
\end{proof}
As a consequence, we obtain:
\begin{corollary}\label{gLL'}
$\mathfrak{i}^{g}_{L,L''_g}$ is bijective.
\end{corollary}
For such $g$, let us define $M[g]\in \End(\mathcal{V}_{L,\psi})$ as follows:
\begin{align}\label{gactionsigma313}
M[g]=  \mathfrak{i}_{L''_g,L} \circ\mathfrak{i}^g_{L,L''_g} .
\end{align}
\begin{lemma}
$M[g]$ is bijective. 
\end{lemma}
\begin{proof}
It follows from Lemma \ref{LL''} and Corollary \ref{gLL'}.
\end{proof}
\subsection{On the whole group $\Sp(W)$}
Recall the  Iwahori decomposition: $$\Sp(W)=[d_L^{-1}Id_L]\cdot [d_L^{-1}\mathfrak{W}^{\textrm{eaff}}d_L] \cdot [d_L^{-1}Id_L].$$ For an element $g\in \Sp(W)$, let us \emph{fix } a decomposition:
$$g=i_1\omega_g i_2$$
for $i_1, i_2\in [d_L^{-1}Id_L]$ and $\omega_g\in [d_L^{-1}\mathfrak{W}^{\textrm{eaff}}d_L]$.  
For   $g=1\in \Sp(W)$,  we require that $i_1=1=i_2$. We also  choose the Haar measure $\mu$ such that $M[1]=\id$. Let us define: 
$$M[g]=M[i_1]M[\omega_g]M[i_2].$$
Then $M[g]\in \End(\mathcal{V}_{L,\psi})$ and $M[g]$ is bijective. 
\begin{lemma}\label{wholegroup23}
Let $h=(w',t)\in \Ha(W)$, $g\in \Sp(W)$. Then the following diagrams are commutative:
\begin{equation}\label{eq2}
\begin{CD}
\mathcal{V}_{L,\psi}@>M[g]>> \mathcal{V}_{L,\psi} \\
      @V\pi_{L,\psi}(h) VV @VV\pi_{L,\psi}(h^{g^{-1}}) V  \\
\mathcal{V}_{L,\psi}@>M[g]>> \mathcal{V}_{L,\psi}
\end{CD}
\qquad\qquad \qquad\qquad  \begin{CD}
\mathcal{V}_{L_1,\psi_{\chi}}@>M[g]>> \mathcal{V}_{L_1,\psi_{\chi}} \\
      @V\pi_{L_1,\psi_{\chi}}(h) VV @VV\pi_{L_1,\psi_{\chi}}(h^{g^{-1}}) V  \\
\mathcal{V}_{L_1,\psi_{\chi}}@>M[g]>> \mathcal{V}_{L_1,\psi_{\chi}}
\end{CD}
\end{equation} 
\end{lemma}
\begin{proof}
It follows from Lemmas \ref{coomLL''}, \ref{chiLL''},  for $g\in d_L^{-1}Id_L$ or $d_L^{-1}\mathfrak{W}^{\textrm{eaff}}d_L$.  So it holds for $g\in \Sp(W)$.        
\end{proof}
\begin{lemma}\label{wholegroup23}
Let $g\in \Sp(W)$. Then the following diagram is commutative:
\begin{equation}\label{eq1}
\begin{CD}
 \mathcal{V}_{L_1,\psi_{\chi_1}}@>M[g]>> \mathcal{V}_{L_1,\psi_{\chi_1}}\\
       @V\mathcal{A}_{\chi_1,\chi_2}VV @VV\mathcal{A}_{\chi_1, \chi_2} V  \\
\mathcal{V}_{L_1,\psi_{\chi_2}}@>M[g]>> \mathcal{V}_{L_1,\psi_{\chi_2}} 
\end{CD}
\end{equation} 
\end{lemma}
\begin{proof}
It follows from Lemma \ref{Achi}, for  $g\in d_L^{-1}Id_L$ or $d_L^{-1}\mathfrak{W}^{\textrm{eaff}}d_L$. Then it holds for $g\in \Sp(W)$.
\end{proof} 
Let us define                               
$$\pi_{L,\psi}([g,h])=M[g]\pi_{L,\psi}(h)$$
for $g\in \Sp(W)$, $h\in \Ha(W)$.  By Lemma \ref{wholegroup23}, it is well defined.
Moreover,  $M[g]$ sends  $\mathcal{V}_{L_1,\psi_{\chi}}$ onto itself. Hence  there exists a non-trivial $2$-cocycle $c_{\chi}$ of order $2$ in $\Ha^2(\Sp(W), T)$, such that 
\begin{align}
\pi_{L,\psi}([g,h])\pi_{L,\psi}([g',h'])f =c_{\chi}(g, g')\pi_{L,\psi}([g,h][g',h'])f.
\end{align}
for $[g,h]$,  $[g',h']\in \Sp(W)\ltimes \Ha(W)$, $f\in \mathcal{V}_{L_1,\psi_{\chi}}$. 
By Lemma \ref{wholegroup23}, we can see that all  cocycles $c_{\chi}$ are the same. Let 
$$1\longrightarrow T \longrightarrow \Mp(W) \longrightarrow \Sp(W) \longrightarrow 1$$
be the central extension of $\Sp(W)$ by  $T$  associated to $c(-,-)$.   Finally, we  conclude:
\begin{theorem}\label{theoremnon2}
 For $[g,t]\in \Mp(W)$, $h\in \Ha(W)$, $f\in \mathcal{V}_{L, \psi}$, let 
$$\pi_{L,\psi}([(g,t),h])f=tM[g]\pi_{L,\psi}(h) f.$$
Then $\pi_{L,\psi}$ defines a representation of $\Mp(W)\ltimes \Ha(W)$.  The restriction of $\pi_{L,\psi}$  to $\Ha(W)$ contains $\sqrt{|L^{\ast}/L|}$-number of irreducible components and every  component is a Heisenberg representation associated to $\psi$. 
\end{theorem}
\section{The real field case}
Let $F=\R$.  We will let  $\psi_0$ denote the fixed character of $F$ defined as: $ t \longmapsto e^{2\pi it}$, for $t\in F$. $\kappa=1$ or $-1$.  Let $\psi=\psi_0^{\kappa}$ be another character of $F$ defined as: $t\to \psi_0(\kappa t)$,  for $t\in F$.
\subsection{Self-dual lattice}\label{selfduall}
Let $X=\R e_1 \oplus \cdots \oplus\R e_m$, $X^{\ast}=\R e_1^{\ast} \oplus \cdots \oplus \R e_m^{\ast}$. Let $L_1=\Za  e_1 \oplus  \cdots \oplus \Za  e_m \oplus \Za  e_1^{\ast} \oplus \cdots \oplus \Za e_m^{\ast}$. Then  $L_1$ is  a self-dual lattice of $W$ with respect to $\psi$. Moreover, $L_1=L_1\cap X \oplus L_1\cap X^{\ast}$. 
For any $l\in L_1$, write 
$$l=x_l+x^{\ast}_l$$
for $x_l \in L_1\cap X$, $x_l^{\ast}\in L_1\cap X^{\ast}$.   For $w=x+x^{\ast}$, $w'=y+y^{\ast}\in W $,  with $x, y\in X$, $x^{\ast}, y^{\ast}\in X^{\ast}$,  let us define
   $$B(w, w')=\langle x, y^{\ast}\rangle.$$ 
Let $\chi\in \Irr(L_1)$. Let $\psi_{L_1, \chi}$ denote the extended  character of $\Ha(L_1)$  from $\psi$ of  $F$ defined as:
$$\psi_{L_1, \chi}: \Ha(L_1) \longrightarrow \C^{\times}; (l,t) \longmapsto \psi(t-\tfrac{1}{2}B(l,l))\chi(l),$$ 
for $l\in L_1$. Then $\psi_{L_1, \chi}$ is really a character. 
  Let us define the representation:
 $$(\pi_{L_1,\psi_{\chi}}=\Ind_{\Ha(L_1)}^{\Ha(W)} \psi_{L_1,\chi}, \quad  \mathcal{S}_{L_1,\psi_{\chi}}=\Ind_{\Ha(L_1)}^{\Ha(W)} \C).$$
\begin{lemma}
$\pi_{L_1,\psi_{\chi}}$ defines a Heisenberg representation of $\Ha(W)$ associated to  the center character $\psi$.
\end{lemma} 
\begin{proof}
See \cite{LiVe}.
\end{proof}
The Hilbert space $\mathcal{S}_{L_1,\psi_{\chi}}$ consists  of measurable functions $f: W \longrightarrow \C$ such that
\begin{itemize}
\item[(i)] $f(l+w)=\psi(-\tfrac{B(l,l) }{2}-\tfrac{\langle l, w\rangle}{2}) \chi(l) f(w)$, for  all $l\in L_1=(L_1\cap X)\oplus (L_1\cap X^{\ast})$, almost all $w\in W$;
\item[(ii)] $\int_{L_1\setminus W} ||f(w)||^2 dw<+\infty$.
\end{itemize}
Here, we choose a $W$-right invariant measure on $L_1\setminus W$. Then $\pi_{L_1,\psi_{\chi}}$  can be realized on $\mathcal{S}_{L_1,\psi_{\chi}}$ by the following formula:
$$\pi_{L_1,\psi_{\chi}}([w',t])f(w)=\psi(t+\tfrac{\langle w, w'\rangle}{2})f(w+w')$$
for $w,w'\in W$, $t\in F$.
\begin{remark}
If $\chi$ is the trivial character, we will  write $\pi_{L_1,\psi}$.  Then $\pi_{L_1,\psi_{\chi}}\simeq \pi_{L_1,\psi}$, for any $\chi$.
\end{remark}
\subsection{Non-self-dual lattice} Let $L$ be a sublattice of the above self-dual lattice $L_1$ with finite index  $[L_1:L]$. 
Note that $$\psi: L_1/L \times L^{\ast}/L_1 \longrightarrow T; (x,y^{\ast}) \longrightarrow \psi(\langle x, y^{\ast}\rangle ),$$
 defines a non-degenerate bilinear map. Hence $[L_1:L]=\sqrt{[L^{\ast}: L]}$. For any $\chi \in \Irr(L_1/L)$, there exists an element $y^{\ast}_{\chi} \in L^{\ast}/L_1 $ such that 
 $$\psi(\langle  x, y_{\chi}^{\ast}\rangle )=\chi(x), \quad\quad \quad x\in L_1/L.$$
Let $\Ha(L)=L\times F$ denote the corresponding subgroup of $\Ha(W)$.  Let $\psi_{L}$ denote the extended  character of $\Ha(L)$  from $\psi$ of  $F$ defined as:
$$\psi_{L}: \Ha(L) \longrightarrow \C^{\times}; (l,t) \longmapsto \psi(t-\tfrac{1}{2}B(l,l)),$$ 
for $l=x_l+x_l^{\ast}\in L$.   Let us define the representation:
 $$(\pi_{L,\psi}=\Ind_{\Ha(L)}^{\Ha(W)} \psi_{L}, \quad  \mathcal{S}_{L,\psi}=\Ind_{\Ha(L)}^{\Ha(W)} \C).$$
The Hilbert space $\mathcal{S}_{L,\psi}$ consists  of  measurable functions $f: W \longrightarrow \C$ such that
\begin{itemize}
\item[(i)] $f(l+w)=\psi(-\tfrac{B(l,l) }{2}-\tfrac{\langle l, w\rangle}{2}) f(w)$, for all  $l\in L$, almost all $w\in W$;
\item[(ii)] $\int_{L\setminus W} ||f(w)||^2 dw<+\infty$.
\end{itemize}
Here, we choose a $W$-right invariant measure on $L\setminus W$. Then $\pi_{L,\psi}$  can be realized on $\mathcal{S}_{L,\psi}$ by the following formulas:
$$\pi_{L,\psi}([w',t])f(w)=\psi(t+\tfrac{\langle w, w'\rangle}{2})f(w+w')$$
for $w,w'\in W$, $t\in F$.

 Note that $\Ha(L) \subseteq \Ha(L_1)$.  Let $\chi$ be a character of $L_1/L$. Then $\psi_{L_1, \chi}|_{\Ha(L)}= \psi_{L}$. 
By Clifford theory, $\Ind_{\Ha(L) }^{\Ha(L_1) } \psi_{L}\simeq  \oplus_{\chi \in \Irr(L_1/L)} \psi_{L_1, \chi}$. As a consequence, we obtain:
\begin{lemma}
$\pi_{L, \psi}\simeq \sqrt{[L^{\ast}: L]} \pi_{L_1, \psi}.$
\end{lemma}
\begin{proof}
$\pi_{L,\psi}=\Ind_{\Ha(L)}^{\Ha(W)} \psi_{L}\simeq \Ind_{\Ha(L_1)}^{\Ha(W)}  (\Ind_{\Ha(L) }^{\Ha(L_1) } \psi_{L})$
$\simeq \Ind_{\Ha(L_1)}^{\Ha(W)}( \oplus_{\chi \in \Irr(L_1/L)} \psi_{L_1, \chi}) \simeq \oplus_{\chi \in \Irr(L_1/L)}  \pi_{L_1, \psi_{\chi}}$
$\simeq \sqrt{[L^{\ast}: L]}\pi_{L_1, \psi}$.
\end{proof}
Let $\mathcal{S}'_{L_1,\psi_{\chi}}$  denote the subspace of  elements $f$ in $\mathcal{S}_{L,\psi}$ such that 
$$f(l+w)=\psi(-\tfrac{1}{2}B(l,l)-\tfrac{1}{2}\langle l, w\rangle)\chi(l)f(w),$$
for all $l\in L_1$, almost all $w\in W$. Since $L_1/L$ is a finite group, 
$$\int_{L\setminus W} ||f(w)||^2 dw<+\infty \Longleftrightarrow \int_{L_1\setminus W} ||f(w)||^2 dw<+\infty $$
for $f\in \mathcal{S}'_{L_1,\psi_{\chi}}$. So we can identity $\mathcal{S}'_{L_1,\psi_{\chi}}$ with the vector space $\mathcal{S}_{L_1,\psi_{\chi}}$, as described in Section \ref{selfduall}.  As a consequence, we can see that $\mathcal{S}'_{L_1,\psi_{\chi}}$ is an irreducible $\Ha(W)$-module associated to $\psi$. From now on we will use $\mathcal{S}_{L_1,\psi_{\chi}}$ instead of $\mathcal{S}'_{L_1,\psi_{\chi}}$.  
For two different $\chi_1$ and $\chi_2$, let $f_i \in \mathcal{S}_{L_1,\psi_{\chi_i}}$. Then:
\begin{align}
\langle f_1, f_2\rangle &=\int_{L\setminus W}  f_1(w)\overline{f_2(w)}dw\\
&=\int_{L_1\setminus W} \sum_{l\in L_1/L} f_1(w+l)\overline{f_2(w+l)}dw\\
&=\int_{L_1\setminus W}  f_1(w)\overline{f_2(w)} \sum_{l\in L_1/L} \chi_1(l)\overline{\chi_2(l)}dw\\
&=0.
\end{align}
Hence $\mathcal{S}_{L_1,\psi_{\chi_1}}$ is orthogonal to $\mathcal{S}_{L_1,\psi_{\chi_2}}$ in $\mathcal{S}_{L,\psi}$. Taking into account the number  of irreducible components, we deduce that
$$\mathcal{S}_{L,\psi}=\oplus_{\chi\in \Irr(L_1/L)} \mathcal{S}_{L_1,\psi_{\chi}}.$$
\subsubsection{}  For any $\chi \in \Irr(L_1/L)$,  $ (\pi_{L_1,\psi_{\chi}}, \mathcal{S}_{L_1,\psi_{\chi}}) $ is an irreducible representation of $\Ha(W)$. 
 Hence  
 $$(\sigma_{\chi}=\Ind_{\Ha(L_1)}^{\Ha(L^{\ast})} \psi_{L_1,\chi}, \mathcal{W}_{\chi}=\Ind_{\Ha(L_1)}^{\Ha(L^{\ast})} \C)$$ is an irreducible representation of $\Ha(L^{\ast})$.  The vector space $\mathcal{W}_{\chi}$ consists of the  functions $f$ on $L^{\ast}$ such that 
 $$f(l+l^{\ast})= \chi(l) \psi(-\tfrac{1}{2}B(l,l)-\tfrac{1}{2} \langle l, l^{\ast}\rangle)f(l^{\ast}).$$
 Note that for any $\chi \in \Irr(L_1/L)$, there exists an element $y^{\ast}_{\chi} \in L^{\ast}/L_1 $ such that 
 $$\psi(\langle  x, y_{\chi}^{\ast}\rangle )=\chi(x), \quad\quad \quad x\in L_1/L.$$
 Let us define a function:
  $$\mathcal{A}_{\chi_1, \chi_2}: \mathcal{W}_{\chi_1} \longrightarrow \mathcal{W}_{\chi_2},$$ 
  given by 
 \begin{align}
 \mathcal{A}_{\chi_1, \chi_2}(f)(l^{\ast})=f(l^{\ast}+y^{\ast}_{\chi_1}-y^{\ast}_{\chi_2}) \psi(-\tfrac{1}{2}\langle l^{\ast}, y^{\ast}_{\chi_1}-y^{\ast}_{\chi_2}\rangle).
 \end{align}
  For $f\in  \mathcal{W}_{\chi_1}$, $l\in L_1$, we have:
 \begin{align}
 &\mathcal{A}_{\chi_1, \chi_2}(f)(l+l^{\ast})\\
 &=f(l+l^{\ast}+y^{\ast}_{\chi_1}-y^{\ast}_{\chi_2})\psi(-\tfrac{1}{2}\langle l+l^{\ast}, y^{\ast}_{\chi_1}-y^{\ast}_{\chi_2}\rangle)\\
 &=\chi_1(l) \psi(-\tfrac{1}{2}B(l,l)-\tfrac{1}{2} \langle l, l^{\ast}+y^{\ast}_{\chi_1}-y^{\ast}_{\chi_2}\rangle)f(l^{\ast}+y^{\ast}_{\chi_1}-y^{\ast}_{\chi_2})\psi(-\tfrac{1}{2}\langle l+l^{\ast}, y^{\ast}_{\chi_1}-y^{\ast}_{\chi_2}\rangle)\\
 &=\chi_2(l) \psi(-\tfrac{1}{2}B(l,l)-\tfrac{1}{2} \langle l, l^{\ast}\rangle)\mathcal{A}_{\chi_1, \chi_2}(f)(l^{\ast}).
 \end{align}
 Hence $ \mathcal{A}_{\chi_1, \chi_2}(f) \in\mathcal{W}_{\chi_2}$.  So $\mathcal{A}_{\chi_1, \chi_2}$ is well-defined.
 \begin{lemma}
 $\mathcal{A}_{\chi_1, \chi_2}$ defines an intertwining operator from $\sigma_{\chi_1}$ to $\sigma_{\chi_2}$.
 \end{lemma}
  \begin{proof}
See the proof of Lemma \ref{iter}.
 \end{proof}
 \subsubsection{}
Let $$\sigma=\Ind_{\Ha(L)}^{\Ha(L^{\ast})}\psi_{L},  \mathcal{W}=\Ind_{\Ha(L)}^{\Ha(L^{\ast})}\C.$$
The vector space $\mathcal{W}$ consists of the functions $f:L^{\ast} \longrightarrow \C$ such that 
$$f(l+l^{\ast})=\psi(-\tfrac{1}{2}B(l,l)-\tfrac{1}{2}\langle l, l^{\ast}\rangle) f(l^{\ast}),$$ 
for $l\in L$, $l^{\ast}\in L^{\ast}$. The action is given as follows:
$$[\sigma(l^{\ast}) f](l_1^{\ast})= \psi(\tfrac{1}{2} \langle l_1^{\ast}, l^{\ast}\rangle) f(l^{\ast}+ l_1^{\ast}).$$
 Let us denote  $\mathcal{V}_{L,\psi}= \Ind_{\Ha(L^{\ast})}^{\Ha(W)} \mathcal{W}$. Then $\pi_{L,\psi}$ can be realized on the Hilbert space $\mathcal{V}_{L,\psi}$. The vector space  $\mathcal{V}_{L,\psi}$ consists  of measurable  functions $f: W \longrightarrow \mathcal{W}$  such that
 \begin{itemize}
 \item[(i)] $f(l^{\ast}+w)=\psi(\tfrac{1}{2}\langle w, l^{\ast}\rangle)\sigma(l_1^{\ast})f(w)$, for all $l^{\ast}\in L^{\ast}$, almost all $w\in W$,
\item[(ii)] $\int_{L\setminus W} ||f(w)||^2 dw<+\infty$.
\end{itemize}
   For $h=(w',t)\in \Ha(W)$,
\begin{equation}\label{BB'2}
\pi_{L_1,\psi}(h) f(w)=\psi( t+\tfrac{1}{2}\langle w, w'\rangle) f(w+w').
\end{equation}
Let us denote  $\mathcal{V}_{L_1 ,\psi_{\chi}}= \Ind_{\Ha(L^{\ast} )}^{\Ha(W)} \mathcal{W}_{\chi}$. This vector space  consists  of measurable  functions $f: W \longrightarrow \mathcal{W}_{\chi}$  such that
\begin{itemize}
 \item[(i)] $f(l^{\ast}+w)=\psi(\tfrac{1}{2}\langle w, l^{\ast}\rangle)\sigma_{\chi}(l^{\ast})f(w)$,
  for all  $l^{\ast}\in L^{\ast}$, almost all  $w\in W$;
\item[(ii)] $\int_{L\setminus W} ||f(w)||^2 dw<+\infty$.
\end{itemize}
 Moreover, $\pi_{L_1,\psi_{\chi}}$ can be realized on $\mathcal{V}_{L_1 ,\psi_{\chi}}$.  Consequently, 
  $$\mathcal{V}_{L ,\psi} \simeq \oplus_{\chi \in \Irr(L_1/L )} \mathcal{V}_{L_1 ,\psi_{\chi}}.$$
\subsection{Non-self-dual lattice II}
In analogy with Section \ref{nonselfdualtwo},  let $d=\diag(d_1, \cdots, d_m; \tfrac{1}{d_1}, \cdots, \tfrac{1}{d_m})\in \Sp_{2m}(\Q)$. Let:
\begin{itemize}
\item $L''=Ld$,
\item $L''_1=L_1d$,
\item $L^{''\ast}=L^{ \ast}d$: the dual of $L''$.
 \end{itemize}
We proceed by replacing the original $L$ by $L''$ and proceed to define the associated representations:
$$(\sigma''=\Ind_{\Ha(L'')}^{\Ha(L^{''\ast})}\psi_{L''},\quad  \mathcal{W}''=\Ind_{\Ha(L'')}^{\Ha(L^{''\ast})}\C)$$
 $$(\sigma''_{\chi}=\Ind_{\Ha(L_1'')}^{\Ha(L^{''\ast})} \psi_{L_1'',\chi}, \quad \mathcal{W}''_{\chi}=\Ind_{\Ha(L_1'')}^{\Ha(L^{''\ast})} \C)$$ 
$$(\pi_{L'' ,\psi}=\Ind_{\Ha(L'' )}^{\Ha(W)} \psi_{L''}, \quad  \mathcal{S}_{L'' ,\psi}=\Ind_{\Ha(L'' )}^{\Ha(W)} \C)$$
$$(\pi_{L'' ,\psi}=\Ind_{\Ha(L^{''\ast})}^{\Ha(W)} \sigma'', \quad \mathcal{V}_{L'',\psi}= \Ind_{\Ha(L^{''\ast})}^{\Ha(W)} \mathcal{W}'')$$
$$(\pi_{L'' ,\psi_{\chi}}=\Ind_{\Ha(L^{''\ast})}^{\Ha(W)} \sigma''_{\chi}, \quad \mathcal{V}_{L'',\psi_{\chi}}= \Ind_{\Ha(L^{''\ast})}^{\Ha(W)} \mathcal{W}''_{\chi}).$$
Clearly, there exists a group isomorphism: 
$$\mathfrak{i}: \Ha(W) \to \Ha(W); [w,t] \longmapsto [wd, t].$$
which  sends $\Ha(L^{\ast})$ to $\Ha(L^{''\ast})$, $\Ha(L_1)$ to $\Ha(L_1^{''})$,  $\Ha(L)$ to $\Ha(L'')$,  $L_1/L$ to $L_1''/L''$, and  $L^{\ast}/L_1$ to $L^{''\ast}/L_1^{''}$.
 Similar to (\ref{LL_1}), 
\begin{align}\label{L''L_1''}
 \psi: L_1''/L'' \times L^{''\ast}/L_1'' \longrightarrow T; (x,y^{\ast}) \longrightarrow \psi(\langle x, y^{\ast}\rangle ),
 \end{align}
 defines a non-degenerate bilinear map.  So for any $\chi \in \Irr(L_1''/L'')$, there exists an element $y^{\ast}_{\chi} \in L^{''\ast}/L_1'' $ such that 
 $$
 \psi(\langle  x, y_{\chi}^{\ast}\rangle )=\chi(x), \quad\quad \quad x\in L_1''/L''.
 $$
 For any $\chi\in \Irr(L_1''/L'')$, $\chi \circ \mathfrak{i}\in \Irr( L_1/L)$. Then $\mathfrak{i}(y_{\chi}^{\ast}) = y_{\chi \circ \mathfrak{i}}^{\ast}$. Let us extend $\mathfrak{i}$ to the representations.  Define $\mathcal{D}:  \mathcal{W}\to  \mathcal{W}''; f \longmapsto \mathcal{D}(f)$, where $\mathcal{D}(f)(l^{''\ast})=f(l^{''\ast} d^{-1})$. Then:
 \begin{itemize}
 \item[(1)]  $\sigma''(\mathfrak{i}([l^{\ast},t])) \circ \mathcal{D}=\mathcal{D} \circ \sigma([l^{\ast},t]) $, for $[l^{\ast},t] \in \Ha(L^{\ast})$. 
      \item[(2)]  $\mathcal{D}$ sends $\mathcal{W}_{\chi \circ \mathfrak{i}}$ onto $\mathcal{W}''_{\chi}$.
       \item[(3)] $\mathcal{A}_{\chi_1, \chi_2}\circ \mathcal{D}=\mathcal{D}\circ \mathcal{A}_{\chi_1 \circ \mathfrak{i}, \chi_2 \circ \mathfrak{i}} $.
\end{itemize}
\subsection{}\label{inte} For  $g\in \Sp_{2m}(\Q)$, let us define a $\C$-linear map  $ \mathfrak{i}^g_{L,L''}$ from $\mathcal{V}_{L,\psi}$ to $ \mathcal{V}_{L'',\psi}$ as follows:
\begin{align}\label{interLL''}
\mathfrak{i}^g_{L,L''}(f)(w)&=\int_{a\in L^{''\ast}}\psi(\tfrac{1}{2}\langle a, w\rangle) \sigma''(-a)\mathcal{D}[f((a+w)g)] da.
\end{align}
Similar to the non-archimeadean field case, we also have:
\begin{itemize}
\item[(1)] $\mathfrak{i}^g_{L,L''}$ is well-defined; 
\item[(2)] For $h=(w',t)\in \Ha(W)$,  $hg^{-1}=(w'g^{-1},t)$. Then $\pi_{L'',\psi}(h^{g^{-1}}) \mathfrak{i}^g_{L,L''}= \mathfrak{i}^g_{L,L''}\pi_{L,\psi}(h) $;
\item[(3)] $\mathfrak{i}^g_{L,L''}$ sends $\mathcal{V}_{L_1,\psi_{\chi\circ \mathfrak{i}}}$ onto $\mathcal{V}_{L_1'',\psi_{\chi}}$;
\item[(4)] If $g=1$, write $\mathfrak{i}_{L,L''}$ for $\mathfrak{i}^g_{L,L''}$. Then $\mathfrak{i}_{L,L''}$ is bijective.
\end{itemize}  
\subsection{Non-self-dual lattice III} Similar to the non-archimedean case, we also  consider a special sublattice $L$ of $L_1$.   Let $d_L=\begin{bmatrix}
g_1& 0\\
0& g_2\end{bmatrix}\in \GSp_{2m}(\Q)$, such that $(L_1\cap X)g_1 \subseteq L_1\cap X$, $(L_1\cap X^{\ast})g_2 \subseteq L_1\cap X^{\ast}$. Let $$L=L_1d_L.$$ Then $L \subseteq L_1$ and  $L=L\cap X\oplus L\cap X^{\ast}$.
 Let us define:
 $$L'=2 L\cap X\oplus \tfrac{1}{2} L\cap X^{\ast}.$$
Then  $Ld_L^{-1}=L_1$, $L' d_L^{-1}= L_1'$, and $L^{0} d_L^{-1} =L_1^0$, $L^{0}\subseteq L_1^0$.
Let us  write:
\begin{itemize}
\item  $L'=L d'$.
\item $t=\lambda_{d_L}\in \Z$, the similitude factor  of $d_L$.
 \end{itemize}
 \subsection{On the Iwahori group $d_L^{-1}Id_L$}
We take the foregoing $L''$ to be $L'$.  Write  $g=d_L^{-1}g_0d_L\in d_L^{-1}I d_L$. Let  $\mathcal{A}_g=\{ a\in L^{\ast} \mid  ad'g-a\in L\}$.
\begin{lemma}\label{trial123}
For $a\in  \mathcal{A}_g$, $\psi(\tfrac{1}{2}\langle -a, ad'g\rangle-\tfrac{1}{2}B(ad'g-a,ad'g-a))=1$.
\end{lemma}
\begin{proof} 
See the proof of Lemma \ref{trial1}.
\end{proof}
  Let $c_g=\mu(L^{\ast})\mu(\mathcal{A}_g)$. 
\begin{lemma}\label{LL'g3}
 For $g\in d_L^{-1}Id_L$,  $\mathfrak{i}^{g^{-1}}_{L',L}\circ \mathfrak{i}^{g}_{L,L'}=c_g\id$.
\end{lemma}
\begin{proof}
See the proof of Lemma \ref{LL'g}.
\end{proof}
\begin{corollary}\label{gLL'43}
$\mathfrak{i}^{g}_{L,L'}$ is bijective.
\end{corollary}
\subsubsection{} For $g\in d_L^{-1}Id_L$, let us define $M[g]\in \End(\mathcal{V}_{L,\psi})$ as follows:
\begin{align}\label{gactionsigma31}
M[g]=  \mathfrak{i}_{L',L} \circ\mathfrak{i}^g_{L,L'} .
\end{align}
\begin{lemma}
$M[g]$ is bijective. 
\end{lemma}
\begin{proof}
It follows from the above lemma.
\end{proof}
  \subsection{On the  Iwahori-Weyl  group $d_L^{-1}\mathfrak{W}^{\textrm{eaff}}d_L$ I} For $g=d_L^{-1} g' d_L\in d_L^{-1}\mathfrak{W}^{\textrm{eaff}}d_L$, let  $\mathcal{B}_g=\{ a\in L^{\ast} \mid  ag-a\in L\}$.
By Section \ref{APP}, for any $g'$, let us write 
$$g'=\omega_s  \begin{bmatrix}
d& 0 \\0 &  d^{-1} \end{bmatrix} \omega_S, \quad g_0=\begin{bmatrix}
d& 0 \\0 &  d^{-1} \end{bmatrix} \omega_S, $$
for some $S\subseteq \{1, \cdots, m\}$.  

\begin{itemize}
\item If  $a\in \mathcal{B}_g$, then $a=a'd_L t^{-1}$, for some $a'\in L_1$. Hence  
$$ag-a\in L \Longleftrightarrow a' d_L t^{-1}  d_L^{-1}g'd_L -a'd_L t^{-1} \in  L_1d_L \Longleftrightarrow a'  g'-a' \in L_1 t .$$
\end{itemize}
\begin{lemma}\label{trial151}
If $t=\lambda_{d_L}\in \Z\setminus 2\Z$ or $ \omega_S=1$,  then $\psi(\tfrac{1}{2}\langle -a, ag\rangle-\tfrac{1}{2}B(ag-a,ag-a))=1$, for $a\in \mathcal{B}_g$.
\end{lemma}
\begin{proof} 
 Note that $\tfrac{1}{2}\langle -a, ag\rangle-\tfrac{1}{2}B(ag-a,ag-a) \in \tfrac{1}{2}\Z$.  By the proof of Lemma \ref{trial1}, we have:
  \begin{align*}
& \tfrac{1}{2}\langle -a, ag\rangle-\tfrac{1}{2}B(ag-a,ag-a)\\
&=B(ag, a-ag)+q_g(a).
\end{align*}
As $B(ag, a-ag)\in \Z$, $q_g(a)\in \tfrac{1}{2}\Z$. Note that $q_g(a)=t^{-1}q_{g'}(a')$. \\
 1) In this first case,   $a' \in L_1$,  and $q_{g'}(a')\in \Z$. Since $\nu_2(t^{-1}q_{g'}(a'))=\nu_2(q_{g'}(a'))\geq 0$, $\psi(q_g(a))=1$. \\
2) In the second case,  $\psi(q_g(a))=\psi(t^{-1}q_{g'}(a'))=1$. 
\end{proof}
\begin{remark}
Let $g_1=d_L^{-1} \omega_s d_L$, $g_2=d_L^{-1} g_0 d_L$. Then $q_{g_1}(w)=0$ and $\psi(q_g(w))=\psi(q_{g_2}(wg_1))$. So it reduces to discussing $g_2$.
\end{remark}
In the following, we will consider the remaining cases that $\omega_S\neq 1$ and $t=\lambda_{d_L}\in 2\Z$. 
\subsubsection{The case $\dim W=2$} Write $g_0=  \begin{bmatrix}
d& 0 \\0 &  d^{-1} \end{bmatrix}\begin{bmatrix}0& -1 \\1 &  0 \end{bmatrix}  $.
For $a'=xe_1+ ye_1^{\ast}$, $a' g_0= d^{-1} ye_1- dxe_1^{\ast}$, $a' g_0-a'=(d^{-1} y-x)e_1+(- dx-y)e_1^{\ast}\in L_1t$. Then
$$d^{-1} y-x=tl_1, \quad \quad - dx-y=tl_2,$$
for some $l_1e_1+l_2e_1^{\ast}\in L_1$. Hence:
$$y=-\tfrac{1}{2} tl_2+ \tfrac{1}{2} tdl_1, \quad x=-\tfrac{1}{2d} tl_2-\tfrac{1}{2} tl_1.$$
 \begin{align*}
&\tfrac{1}{2}\langle -a, ag\rangle-\tfrac{1}{2}B(ag-a,ag-a)\\
&=t^{-1}B(a'g_0, a'-a'g_0)+t^{-1}q_{g_0}(a').
\end{align*}
Note that $a'g_0-a'\in L_1t$, so $a'g_0\in L_1$. Hence $t^{-1}B(a'g_0, a'-a'g_0)\in \Z$ and $\psi(t^{-1}B(a'g_0, a'-a'g_0))=1$. As a consequence, $t^{-1}q_{g_0}(a')\in \tfrac{1}{2}\Z$. By (\ref{g0adt}), we have:
\begin{align}\label{3qcal}
\tfrac{1}{t}q_{g_0}(a')
=-\tfrac{t}{4d}[l_2^2-(dl_1)^2]=-dl_1x -\tfrac{d}{t}x^2=\tfrac{y^2}{dt}+\tfrac{1}{d}yl_2.
\end{align}
 Let $n_t=\nu_2(t)\geq 1$, $n_d=\nu_2(d)$. 
\begin{lemma}\label{3mide}
\begin{itemize}
\item[(i)]   If $n_d-n_t\geq 0$ or $n_d+n_t\leq 0$,  then $\psi(\tfrac{1}{2}\langle -a, ag\rangle-\tfrac{1}{2}B(ag-a,ag-a))=1$, for $a\in \mathcal{B}_g$.
\item[(ii)]  If $n_d-n_t\leq -2$ and $n_d+n_t\geq 2$, then $\psi(\tfrac{1}{2}\langle -a, ag\rangle-\tfrac{1}{2}B(ag-a,ag-a))=1$, for $a\in \mathcal{B}_g$.
\end{itemize} 
\end{lemma}
\begin{proof}
ia) If $n_d-n_t\geq 0$, then $n_d\geq n_t\geq 1$.   By (\ref{3qcal}), $\tfrac{1}{t}q_{g_0}(a')=-dl_1x -\tfrac{d}{t}x^2$. Hence $$\nu_2(\tfrac{1}{t}q_{g_0}(a'))=\nu_2(-dl_1x -\tfrac{d}{t}x^2) \geq \min\big\{ \nu_2(-dl_1x),  \nu_2(-\tfrac{d}{t}x^2)\big\} \geq 0.$$ So $\tfrac{1}{t}q_{g_0}(a')\in \Z$ and $\psi(\tfrac{1}{t}q_{g_0}(a'))=1$. \\
ib) If $n_d+n_t\leq 0$, then $\nu_2(\tfrac{1}{dt})\geq 0$, and $\nu_2(\tfrac{1}{d})\geq 1$. By (\ref{3mide}), $\tfrac{1}{t}q_{g_0}(a')=\tfrac{y^2}{dt}+\tfrac{1}{d}yl_2$. Hence $$\nu_2(\tfrac{1}{t}q_{g_0}(a'))=\nu_2(\tfrac{y^2}{dt}+\tfrac{1}{d}yl_2) \geq \min\big\{ \nu_2(\tfrac{y^2}{dt}),  \nu_2(\tfrac{1}{d}yl_2)\big\} \geq 0.$$ 
ii) In this case, $\nu_2(\tfrac{t}{4d})\geq 0$ and $\nu_2(\tfrac{td}{4})\geq 0$.   By (\ref{3mide}), $\tfrac{1}{t}q_{g_0}(a')=-\tfrac{t}{4d}l_2^2+\tfrac{dt}{4}l_1^2$.
$$\nu_2(\tfrac{1}{t}q_{g_0}(a'))=\nu_2(-\tfrac{t}{4d}l_2^2+\tfrac{dt}{4}l_1^2) \geq \min\big\{ \nu_2(-\tfrac{t}{4d}l_2^2),  \nu_2(\tfrac{dt}{4}l_1^2)\big\} \geq 0.$$
\end{proof}
We consider the exceptional cases mentioned in Lemmas \ref{3mide}.  The exceptional case is that $ n_d-n_t=-1$ or $n_d+n_t=1$.  
Recall that   $\mathbbm{2}=\begin{bmatrix} 2&0\\ 0&  \tfrac{1}{2}\end{bmatrix}$. Let:
$$d''=\left\{\begin{array}{cl} 
\mathbbm{2} & \textrm{ if } n_d-n_t=-1,\\
 \mathbbm{2}^{-1} &\textrm{ if }  n_d+n_t=1.\end{array}\right.$$
$$\mathcal{C}_g=\{ a\in L^{\ast} \mid  ad''g-a\in L\}.$$
\begin{lemma}\label{3mide233}
  If  $n_d-n_t=-1$ or $n_d+n_t=1$,  then $\psi(\tfrac{1}{2}\langle -a, ad''g\rangle-\tfrac{1}{2}B(ad''g-a,ad''g-a))=1$, for $a\in \mathcal{C}_g$.
\end{lemma}
\begin{proof}
Similarly, we have:
\begin{align*}
& \tfrac{1}{2}\langle -a, ad''g\rangle-\tfrac{1}{2}B(ad''g-a,ad''g-a)\\
&=B(ad''g, a-ad''g)+q_g(ad'').
\end{align*}
Note that $ad''g-a\in L\subseteq L^{\ast}$. Hence $ad''g\in L^{\ast}$, $B(ad''g, a-ad''g)\in \Z$ and  $\psi(B(ad''g, a-ad''g))=1$. Write  $a=a'd_L t^{-1}$, for some $a'\in L_1$. Then: 
$$ad''g-a\in L \Longleftrightarrow a' d''d_L t^{-1}  d_L^{-1}g_0d_L -a'd_L t^{-1} \in  L_1d_L \Longleftrightarrow a'd''  g_0-a' \in L_1 t .$$
$$q_g(ad'')= t^{-1}q_{g_0}(a'd'').$$
Recall $g_0=  \begin{bmatrix}
d& 0 \\0 &  d^{-1} \end{bmatrix}\begin{bmatrix}0& -1 \\1 &  0 \end{bmatrix}  $. 
\begin{itemize}
\item Assume that $ n_d-n_t=-1$. For $a'=xe_1+ ye_1^{\ast}$, $a' d'' g_0= (2d)^{-1} ye_1- 2dxe_1^{\ast}$, $a' g_0-a'=[(2d)^{-1} y-x]e_1+(- 2dx-y)e_1^{\ast}\in L_1t$. Then
$$(2d)^{-1} y-x=tl_1, \quad \quad - 2dx-y=tl_2,$$
for some $l_1e_1+l_2e_1^{\ast}\in L_1$. Hence:
$$y=-\tfrac{1}{2} tl_2+ tdl_1, \quad x=-\tfrac{1}{4d} tl_2-\tfrac{1}{2} tl_1.$$
\begin{align}\label{344qcal}
\tfrac{1}{t}q_{g_0}(a'd'')=-\tfrac{t}{8d}[l_2^2-4(dl_1)^2]=-2dl_1x -\tfrac{2d}{t}x^2=\tfrac{y^2}{2dt}+\tfrac{1}{2d}yl_2.
\end{align}
As   $n_d+n_t=1$,  $\nu_2(\tfrac{2d}{t})\geq 0$ and $\nu_2(d)\geq 0$.  By (\ref{344qcal}), $\tfrac{1}{t}q_{g_0}(a'd'')=-2dl_1x -\tfrac{2d}{t}x^2$. Hence:
$$\nu_2(\tfrac{1}{t}q_{g_0}(a'd''))=\nu_2(-2dl_1x -\tfrac{2d}{t}x^2) \geq \min\big\{ \nu_2(-2dl_1x),  \nu_2(-\tfrac{2d}{t}x^2)\big\} \geq 0.$$
\item  Assume that $n_d+n_t=1$. For $a'=xe_1+ ye_1^{\ast}$, $a' d'' g_0= 2d^{-1} ye_1- \tfrac{d}{2}xe_1^{\ast}$, $a' g_0-a'=[2d^{-1} -x]e_1+(- \tfrac{d}{2}x-y)e_1^{\ast}\in L_1t$. Then
$$2d^{-1} y-x=tl_1, \quad \quad - \tfrac{d}{2}x-y=tl_2,$$
for some $l_1e_1+l_2e_1^{\ast}\in L_1$. Hence:
$$y=-\tfrac{1}{2} tl_2+ \tfrac{1}{4}tdl_1, \quad x=-\tfrac{1}{d} tl_2-\tfrac{1}{2} tl_1.$$
\begin{align}\label{3445qcal}
\tfrac{1}{t}q_{g_0}(a'd'')=-\tfrac{t}{2d}[l_2^2-\tfrac{1}{4}(dl_1)^2]=-\tfrac{1}{2}dl_1x -\tfrac{d}{2t}x^2=\tfrac{2y^2}{dt}+\tfrac{2}{d}yl_2.
\end{align}
As $n_d+n_t=1$, $\nu_2(\tfrac{2}{dt})\geq 0$ and $\nu_2(\tfrac{1}{d})\geq 0$. By (\ref{3445qcal}), $\tfrac{1}{t}q_{g_0}(a')=\tfrac{2y^2}{dt}+\tfrac{2}{d}yl_2$.Hence:
$$\nu_2(\tfrac{1}{t}q_{g_0}(a'd''))=\nu_2(\tfrac{2y^2}{dt}+\tfrac{2}{d}yl_2) \geq \min\big\{ \nu_2(\tfrac{2y^2}{dt}),  \nu_2(\tfrac{2}{d}yl_2)\big\} \geq 0.$$
\end{itemize}
\end{proof}
\subsubsection{The general case} Let $W_i=\R e_i+\R e_i^{\ast}$. There exist group homomorphisms:
$$\iota: \SL_2(W_1) \times  \cdots \times \SL_2(W_m) \longrightarrow \Sp(W);$$
$$   \iota_0:   \SL_2(\Q) \times  \cdots \times \SL_2(\Q) \longrightarrow \Sp_{2m}(\Q).$$
Then the  image of $ \iota_0$ contains $\mathfrak{W}_0^{\textrm{aff}}$. Assume that $\iota_0([g_0^{(1)}, \cdots, g_0^{(m)}])=g_0$.  Correspondingly, let $L_i=\Z e_i\oplus \Z e_i^{\ast}$. Let $$\mathcal{I}_g=\{ i\mid  g_0^{(\omega_s(i))} \textrm{ satisfies the condition of Lemma \ref{3mide233} }\}.$$ 
\begin{itemize}
\item If $i\in \mathcal{I}_g$, we let $L''_i=L_i d''=\left\{ \begin{array}{cl}
2\Z e_i\oplus \tfrac{1}{2}\Z e_1^{\ast} & \textrm{ if }  n_d-n_t=-1,\\
\tfrac{1}{2}\Z e_i\oplus 2\Z e_1^{\ast}  & \textrm{ if }  n_d+n_t=1.\end{array}\right.$
\item Let $L_g''=\sum_{i\in \mathcal{I}_g} L'_i + \sum_{i\notin \mathcal{I}_g}L_i$.
\item Write $L_g''=Ld_g'$.
\end{itemize}
\subsubsection{Assume $\mathcal{I}_g=\emptyset$}
Let $g\in d_L^{-1}\mathfrak{W}^{\textrm{eaff}}d_L$.  
\begin{lemma} 
 If $\mathcal{I}_g=\emptyset$, $\mathfrak{i}^{g^{-1}}_{L,L}\circ \mathfrak{i}^{g}_{L,L}=c_g\id$, for $c_g=\mu(L^{\ast})\mu(\mathcal{B}_g)$.
\end{lemma}
\begin{proof}
See the proof of Lemma \ref{LLIgin}.
\end{proof}
\begin{corollary}\label{gLL2}
$\mathfrak{i}^{g}_{L,L}$ is bijective.
\end{corollary}
 For such $g\in d_L^{-1}\mathfrak{W}^{\textrm{eaff}}d_L$, let us define $M[g]=  \mathfrak{i}^g_{L,L}$.  Clearly,  $M[g]\in \End(\mathcal{V}_{L,\psi})$.
\subsubsection{Assume $\mathcal{I}_g\neq \emptyset$} 
 Let $g\in d_L^{-1}\mathfrak{W}^{\textrm{eaff}}d_L$.  Let $\mathcal{C}_g=\{ a\in L^{\ast}\mid a d_g'' g-a\in L\}.$  Let $c_g=\mu(L^{\ast})\mu(\mathcal{C}_g)$. 
\begin{lemma} 
 If $\mathcal{I}_g\neq \emptyset$, $\mathfrak{i}^{g^{-1}}_{L''_g,L}\circ \mathfrak{i}^{g}_{L,L''_g}=c_g\id$.
\end{lemma}
\begin{proof}
See the proof of Lemma \ref{LL'g}
\end{proof}
\begin{corollary}\label{gLL'3}
$\mathfrak{i}^{g}_{L,L''_g}$ is bijective.
\end{corollary}
For such $g$, let us define $M[g]= \mathfrak{i}_{L''_g,L} \circ\mathfrak{i}^g_{L,L''_g}$. Clearly,  $M[g]\in \End(\mathcal{V}_{L,\psi})$.
\begin{lemma}
$M[g]$ is bijective. 
\end{lemma}
\begin{proof}
It follows from Section \ref{inte}(4), and Corollary \ref{gLL'3}.
\end{proof}
\subsection{On the whole group $\Sp_{2m}(\Q)$}
Recall the  Iwahori decomposition: $$\Sp_{2m}(\Q)=[d_L^{-1}Id_L]\cdot [d_L^{-1}\mathfrak{W}^{\textrm{eaff}}d_L] \cdot [d_L^{-1}Id_L].$$ For an element $g\in \Sp_{2m}(\Q)$, let us \emph{fix } a decomposition:
$$g=i_1\omega_g i_2$$
for $i_1, i_2\in [d_L^{-1}Id_L]$ and $\omega_g\in [d_L^{-1}\mathfrak{W}^{\textrm{eaff}}d_L]$.  
For   $g=1\in \Sp_{2m}(\Q)$,  we require that $i_1=1=i_2$. We   choose the Haar measure $\mu$ such that $M[1]=\id$. Let us define: 
$$M[g]=M[i_1]M[\omega_g]M[i_2].$$
Then $M[g]\in \End(\mathcal{V}_{L,\psi})$ and $M[g]$ is bijective.  Let $h=(w',t)\in \Ha(W)$, $g\in \Sp_{2m}(\Q)$. Similar to the non-archimedean case, we also have:
\begin{lemma}\label{3wholegroup23}
\begin{itemize}
\item[(1)] $\pi_{L,\psi}(h^{g^{-1}})\circ M[g] (f)=M[g] \circ\pi_{L,\psi}(h) (f)$, for $f\in \mathcal{V}_{L_1,\psi}$.
\item[(2)] $\pi_{L_1,\psi_{\chi}}(h^{g^{-1}}) \circ M[g] (f)=M[g]\circ \pi_{L_1,\psi_{\chi}}(h) (f)$, for $ f\in \mathcal{V}_{L_1,\psi_{\chi}}$.
\item[(3)] $M[g] \circ \mathcal{A}_{\chi_1, \chi_2} (f)=\mathcal{A}_{\chi_1,\chi_2} \circ M[g](f)$, for $f\in \mathcal{V}_{L_1,\psi_{\chi_1}}$.
\end{itemize}
\end{lemma}
Let us define:                               
$$\pi_{L,\psi}([g,h])=M[g]\pi_{L,\psi}(h).$$ By the above lemma \ref{3wholegroup23}, it is well defined and   $M[g]$ sends  $\mathcal{V}_{L_1,\psi_{\chi}}$ onto itself. Hence  there exists a non-trivial $2$-cocycle $c_{\chi}$ of order $2$ in $\Ha^2(\Sp_{2m}(\Q), T)$, such that 
\begin{align}
\pi_{L,\psi}([g,h])\pi_{L,\psi}([g',h'])f =c_{\chi}(g, g')\pi_{L,\psi}([g,h][g',h'])f.
\end{align}
for $[g,h]$,  $[g',h']\in \Sp_{2m}(\Q)\ltimes \Ha(W)$, $f\in \mathcal{V}_{L_1,\psi_{\chi}}$. 
By Lemma \ref{3wholegroup23}(3), we can see that all  cocycles $c_{\chi}$ are the same. Let 
$$1\longrightarrow T \longrightarrow \Mp_{2m}(\Q) \longrightarrow \Sp_{2m}(\Q) \longrightarrow 1$$
be the central extension of $\Sp_{2m}(\Q)$ by  $T$  associated to $c(-,-)$.   Finally, we  conclude:
\begin{theorem}\label{main3}
 For $[g,t]\in \Mp_{2m}(\Q)$, $h\in \Ha(W)$, $f\in \mathcal{V}_{L, \psi}$, let 
$$\pi_{L,\psi}([(g,t),h])f=tM[g]\pi_{L,\psi}(h) f.$$
Then $\pi_{L,\psi}$ defines a representation of $\Mp_{2m}(\Q) \ltimes \Ha(W)$.  The restriction of $\pi_{L,\psi}$  to $\Ha(W)$ contains $\sqrt{|L^{\ast}/L|}$-number of irreducible components and every  component is a Heisenberg representation associated to $\psi$. 
\end{theorem}

\subsection{Appendix}\label{APP}
The objective of this section is to provide a straightforward demonstration of the Iwahori decomposition theorem specifically for the rational symplectic group at the prime number  $2$.  Let $K= \mathbb{F}_2, \Z, \Q$, or $\R$.  Let $\M_{m\times k}(K)$ denote the Matrix ring of rank $m\times k$ over $K$. Let $\GL_{m}(K)$ denote the group of all  invertible elements in $\M_{m\times m}(K)$.   For a matrix $A\in \M_{m\times k}(K)$, let $A^{T}$ denote its transpose matrix. Let $J=\begin{bmatrix}
0 & 1_m\\
-1_m & 0\end{bmatrix}\in\GL_{2m}(K)$, with $1_m$ being the identity element of $\GL_m(K)$. Let $(U, \langle, \rangle)$ be a free symplectic module over $K$ with a  symplectic basis $\{e_1, \cdots, e_m; e_1^{\ast}, \cdots, e_m^{\ast}\}$. For an element $w\in U$, let us write 
$$x=(x_1, \cdots, x_{2m}), \quad e=(e_1, \cdots, e_m, e_1^{\ast}, \cdots, e_m^{\ast}), \quad w=xe^T.$$
Then  we will identity $U$ with $K^{2m}$, and endow $K^{2m}$ with the symplectic form $\langle x, y\rangle=xJy^T$, for $x, y \in K^{2m}$.  For $g\in \Sp(U)$, let us write $e^T g=A_g e^T$, for some $A_g\in \GL_{2m}(K)$. Then the matrix $A_g$ satisfies $A_g JA_g^T=J$.
Let $\Sp_{2m}(K)=\{ A\in \GL_{2m}(K) \mid A_g JA_g^T=J\}$. We identity $\Sp(U)$ with $\Sp_{2m}(K)$ through $g \to A_g$.  For a subset $S\subseteq \{1, \cdots, m\}$, let us define an element   $\omega_S$ of  $ \Sp_{2m}(K)$ as follows: $ (e_i)\omega_S=\left\{\begin{array}{lr}
-e_i^{\ast}& i\in S,\\
 e_i & i\notin S,
 \end{array}\right.$ and $ (e_i^{\ast})\omega_S=\left\{\begin{array}{lr}
e_i^{\ast}& i\notin S,\\
 e_i & i\in S.
 \end{array}\right.$  Let $S_m$ denote the permutation group on $m$ elements. For $s\in S_m$, let $\omega_s$ be an element of $\Sp_{2m}(K)$ defined as $(e_i)\omega_s=e_{s(i)}$ and $(e_j^{\ast}) \omega_s=e_{s(j)}^{\ast}$. Let:
\begin{itemize}
\item  $X=\Span_{K}\{e_1, \cdots, e_m\}$, $X^{\ast}=\Span_{K}\{e^{\ast}_1, \cdots, e^{\ast}_m\}$. 
\item  $M_{X^{\ast}}(K)=\{ g= \begin{pmatrix}
\alpha & 0\\
 0 & (\alpha^{T})^{-1}\end{pmatrix} \mid \alpha\in \GL_m(K)  \} \subseteq \Sp_{2m}(K) $.
 \item  $N_{X^{\ast}}(K)=\{ g= \begin{pmatrix}
1 & \beta\\
 0 & 1\end{pmatrix}  \mid  \beta^T=\beta \} \subseteq \Sp_{2m}(K)$.
 \item $N_{X^{\ast}}^-(K)=\{ g= \begin{pmatrix}
1 &  0\\
 \beta & 1\end{pmatrix}  \mid  \beta^T=\beta \} \subseteq \Sp_{2m}(K)$
 \item  $ P_{X^{\ast}}(K)=\{ g_1g_2 \mid  g_1\in M_{X^{\ast}}(K), g_2\in N_{X^{\ast}}(K) \}$. 
 \item  $T(K)=\{ g= \begin{pmatrix}
\alpha & \\
 & \alpha^{-1}\end{pmatrix}\mid  \alpha=\diag(d_1, \cdots, d_m)\} \subseteq \Sp_{2m}(K)$.
 \item  $N_0(K)=\{ g=\begin{pmatrix}
\alpha & \\
 & (\alpha^{T})^{-1}\end{pmatrix} \mid \alpha= \begin{bmatrix}
1& \ast & \ast\\
 & 1     &  \ast\\
  &      &1
\end{bmatrix}\in\GL_m(K)\} \subseteq \Sp_{2m}(K)$.
\item $N_0^-(K)=\{ g=\begin{pmatrix}
\alpha & \\
 & (\alpha^{T})^{-1}\end{pmatrix} \mid g^T\in N_0(K)\} \subseteq \Sp_{2m}(K)$.
\item $N(K)=N_0(K)N_{X^{\ast}}(K)$.
\item $N^-(K)=N^-_0(K)N^-_{X^{\ast}}(K)$.
\item $B(K)=T(K)N(K)$.
\item $D_0=\{\begin{pmatrix} d &0\\ 0 & d^{-1} \end{pmatrix}  \mid d = \diag(d_1, \cdots, d_m), d_i \in \Z_+,  d_m|d_{m-1}| \cdots |d_1 \}$.
\item $D(2)=\{\begin{pmatrix} d &0\\ 0 & d^{-1} \end{pmatrix} \in D \mid d = \diag(2^{k_1}, \cdots, 2^{k_m}), k_i\in \Z\}$.
 \item   $\mathfrak{W}_{P_{X^{\ast}}}=\{ \omega_S \mid S=\emptyset \textrm{  or } S= \{j, j+1, \cdots, m\} \}$.
 \item   $\mathfrak{W}=\{ \omega_s\omega_S \mid S \subseteq   \{1, \cdots, m\}, s\in S_m\}$.
 \item  $\widetilde{\mathfrak{W}}=\{  \omega_s \omega^{k}_S  \mid S \subseteq   \{1, \cdots, m\}, s\in S_m\}$.(A group)
 \item $D=\{ d^{\omega}\mid \omega\in \mathfrak{W}, d\in D_0\}$.
 \item $\mathfrak{W}^{\textrm{eaff}}=\{ \omega_1 d\omega_2 \mid d\in D_0, \omega_1, \omega_2\in \mathfrak{W}\}=D\rtimes \mathfrak{W}$.
  \item $\mathfrak{W}^{\textrm{eaff}}(2)=D(2) \rtimes \mathfrak{W}$.
 \item $\Gamma(2)=\{ g\in \Sp_{2m}(\Z)\mid g\equiv 1(\bmod 2)\}$.
 \end{itemize}
For $\omega\in \mathfrak{W}$, let $N_{\omega}^-(K)=[\omega N^-(K) \omega^{-1}]\cap N(K)$, $N_{\omega}^{--}(K)=[\omega^{-1} N^-(K) \omega] \cap N(K)$.
\begin{theorem}[Bruhat decomposition]\label{Buhde}
If $K=\F_2$, then:
\begin{itemize}
\item $\Sp_{2m}(K)=\sqcup_{\omega\in \mathfrak{W}_{P_{X^{\ast}}}} P_{X^{\ast}}(K)\omega P_{X^{\ast}}(K)$.
\item $\Sp_{2m}(K)=\sqcup_{\omega\in \mathfrak{W}} B(K)\omega B(K)=\sqcup_{\omega\in \mathfrak{W}} N(K) \omega B(K)=\sqcup_{\omega\in \mathfrak{W}} N_{\omega}^-(K) \omega B(K)=\sqcup_{\omega\in \mathfrak{W}} B(K)\omega N_{\omega}^{--}(K)$.
\end{itemize}
\end{theorem}
\begin{proof}
See  \cite[Sections 2.5-2.9]{Ca1}, and \cite{BoTi}, or \cite[p.60]{DiMi}. 
\end{proof}
\begin{theorem}\label{SPQdecom}
$\Sp_{2m}(\Q)=\sqcup_{\omega_d  \in D_0 } \Sp_{2m}(\Z) \omega_d \Sp_{2m}(\Z)$.
\end{theorem}
\begin{proof}
See \cite[Theorem 2.1]{Be}.
\end{proof}
\subsection{Iwahori decomposition for $\Sp_{2m}(\Q)$} 
Let $\Z_2$ be  the ring of  $2$-adic integers in $\Q_2$. By \cite[Lemma 25]{Ig}, there exists an exact sequence:
$$1\longrightarrow \Gamma(2\Z_2) \longrightarrow \Sp_{2m}(\Z_2) \to \Sp_{2m}(\F_2)\to 1.$$
 Let $\Gamma(2^m\Z_2) =\ker(\Sp_{2m}(\Z_2) \to \Sp_{2m}(\Z_2/2^m\Z_2))$.  Let $I(\Z_2)$, $N(\Z_2)$ denote  the inverse images of $B(\F_2)$ and $N(\F_2)$ in $\Sp_{2m}(\Z_2)$, respectively. Then $I(\Z_2)$ is  an Iwahori  subgroup of $\Sp_{2m}(\Z_2)$.  It is clear that $N(\Z_2)\to N(\F_2)$, $B(\Z_2)\to B(\F_2)$, $T(\Z_2)\to T(\F_2)$ all are surjective.  Then $I(\Z_2)=B(\Z_2)\Gamma(2\Z_2)$.
\begin{theorem}\label{SPQdecomQ23}
$\Sp_{2m}(\Q_2)=\sqcup_{\omega  \in \mathfrak{W}^{\textrm{eaff}}(2) } I(\Z_2) \omega  I(\Z_2)$.
\end{theorem}
\begin{proof}
See \cite{Vi}.
\end{proof}
It is known that there exists an exact sequence:
$$1\longrightarrow \Gamma(2) \longrightarrow \Sp_{2m}(\Z) \to \Sp_{2m}(\mathbb{F}_2)\to 1.$$
Let  $I$ denote the inverse image of $ B(\mathbb{F}_2)$ in $\Sp_{2m}(\Z)$. Let $N=N(\Z)$, $B=B(\Z)$, $T=T(\Z)$. It is clear that $N\to N(\F_2)$, $B\to B(\F_2)$, $T\to T(\F_2)$ all are surjective. By Theorem \ref{Buhde},   there exists the following decomposition:
\begin{equation}\label{sqspZ}
\Sp_{2m}(\Z)=\sqcup_{\omega\in \mathfrak{W}} I\omega I=\sqcup_{\omega\in \mathfrak{W}} I\omega N.
\end{equation}
Note that  $I=\Gamma(2)N$, and  $\Gamma(2)$(resp. $N$) is dense in $\Gamma(2\Z_2)$(resp. $N(\Z_2)$). So $I$ is dense in $I(\Z_2)$, and $I(\Z_2) =I \cdot \Gamma(2^k\Z_2)$, for any $k>0$.
\begin{proposition}[Iwahori decomposition at $2$]\label{Iwah}
$\Sp_{2m}(\Q)=\cup_{\omega \in \mathfrak{W}^{\textrm{eaff}}} I \omega I.$
\end{proposition}

 We will demonstrate the existence of this decomposition through induction on $m$ in the following sections   \ref{sec:induction1}---\ref{sec:induction3}. In Section \ref{sec:induction1},we will prove the case for  $m=1$.  In Section \ref{sec:induction2}, we will prove the case for  $m=2$ and in Section \ref{sec:induction3}, we will prove the case that $m\geq 3$.

 Before giving the proof, we remark that this decomposition is weakly unique with respect to $\omega$. This means that if we write $\omega=\omega_1d_0\omega_2=d\omega_3$, where $\omega_i\in  \mathfrak{W}$,  $d_0\in D_0$, and  $d\in D$, then $d_0$ is uniquely determined by $g$ according to Theorem  \ref{SPQdecom}, and $\omega_3$ is  uniquely determined by $g$ according to Theorem \ref{SPQdecomQ23}.

\subsection{The proof of Prop.\ref{Iwah}: $m=1$}\label{sec:induction1}  Let $\omega=\begin{bmatrix} 0& -1\\ 1& 0\end{bmatrix}$, $n_1=\begin{bmatrix} 1& 1\\ 0& 1\end{bmatrix}$. Then:
\begin{align}
\SL_2(\Z)&=I \cup I \omega N= I \cup I \omega \cup I \omega n_1\\
&=I \cup N \omega I = I \cup \omega I \cup  n_1 \omega I. 
\end{align}
 By Theorem \ref{SPQdecom},  
$$ \SL_2(\Q)= \cup_{d\in D_0} \Big[(I d I)\cup (I d\omega I ) \cup (I d n_1  \omega I) \cup (I \omega d I) \cup (I \omega d \omega I)$$
$$\cup  (I \omega d  n_1 \omega I) \cup (I \omega n_1 d I)\cup (I \omega n_1 d \omega I) \cup (I \omega n_1 d  n_1 \omega I)\Big].$$
Let $d=\diag(t, t^{-1})\in D_0$, with $t\in \Z_+$. 
\begin{itemize}
\item $d n_1  \omega=\begin{bmatrix} 1& t^2\\ 0& 1\end{bmatrix} d \omega \in Id \omega I$.\\
\item  $ \omega d  n_1 \omega=\begin{bmatrix} 0& -1\\ 1& 0\end{bmatrix} \begin{bmatrix}t& 0\\ 0&  t^{-1}\end{bmatrix} \begin{bmatrix} 1& 1\\ 0& 1\end{bmatrix}\begin{bmatrix} 0& -1\\ 1& 0\end{bmatrix} =\begin{bmatrix} -t^{-1}& 0\\t& -t\end{bmatrix}=\begin{bmatrix} 1& 0\\-t^2& 1\end{bmatrix} \begin{bmatrix} -t^{-1}& 0\\0& -t\end{bmatrix}$\\
    $=\begin{bmatrix} 1& 0\\-t^2-1& 1\end{bmatrix} \begin{bmatrix} 1& 1\\0& 1\end{bmatrix}  \begin{bmatrix} 0& -1\\ 1& 0\end{bmatrix} \begin{bmatrix} -t^{-1}& 0\\0& -t\end{bmatrix}\begin{bmatrix}1& t^2\\0& 1\end{bmatrix}$.
    \begin{itemize}
\item If $2\mid t$,  $\begin{bmatrix} 1& 0\\-t^2& 1\end{bmatrix} \in I$. Then $\omega d  n_1 \omega\in I d^{-1} I$.
\item If $ 2\nmid t$, $\omega d  n_1 \omega\in I \omega d^{-1} I.$
     \end{itemize}
\item $\omega n_1 d= \begin{bmatrix} 1& -1\\0& 1\end{bmatrix}\begin{bmatrix} t& 0\\0& t^{-1}\end{bmatrix} \begin{bmatrix}1& 0\\t^2& 1\end{bmatrix}$ $= \begin{bmatrix} 1& t^2-1\\0& 1\end{bmatrix}\begin{bmatrix} t& 0\\0& t^{-1}\end{bmatrix}\begin{bmatrix} 0& -1\\ 1& 0\end{bmatrix} \begin{bmatrix}1& 1\\ 0& 1\end{bmatrix} \begin{bmatrix}1& 0\\t^2-1& 1\end{bmatrix}$.
    \begin{itemize}
\item If $2\mid t$,  $\omega n_1 d\in I d I$.
\item If $ 2\nmid t$, $\omega d  n_1 \omega\in I d \omega  I.$
     \end{itemize}
\item  $\omega n_1 \omega= \begin{bmatrix}
-1& 0\\
1& -1
\end{bmatrix}= \begin{bmatrix}
1& -1\\
0& 1
\end{bmatrix} \begin{bmatrix}
0& -1\\
1& 0
\end{bmatrix} \begin{bmatrix}
1& -1\\
0& 1
\end{bmatrix}\in I\omega I$.\\
\item $\omega n_1 d  n_1 \omega =\begin{bmatrix}1& 0 \\-t^2-1& 1\end{bmatrix}  \begin{bmatrix} -t^{-1}& 0\\0& -t\end{bmatrix}=\begin{bmatrix}1& 0\\ -t^2-2 & 1\end{bmatrix}\begin{bmatrix}
1& 1\\
0& 1
\end{bmatrix} \begin{bmatrix}
0& -1\\
1& 0
\end{bmatrix} \begin{bmatrix} -t^{-1}& 0\\0& -t\end{bmatrix} \begin{bmatrix} 1& t^2\\0& 1\end{bmatrix}  $. 
\begin{itemize}
\item If $2\nmid t$, $\omega n_1 d  n_1 \omega\in I d^{-1} I$.
\item If $ 2\mid t$, $ \omega n_1 d  n_1 \omega \in I \omega d^{-1} I$.
     \end{itemize}
  \end{itemize} 
\subsection{The proof of Prop.\ref{Iwah}: $m=2$}\label{sec:induction2} 
Recall that $\mathfrak{W}=\{ \omega_s \omega_S\mid s=1, s=(12), S=\emptyset, \{2\},\{1\}, \{1,2\}\}.$ 
$$\omega_{\emptyset}= \begin{pmatrix}
1& 0& 0 & 0\\
 0&  1 & 0  &  0\\
0 & 0 & 1   &0 \\
0 & 0  &0    &  1
\end{pmatrix}, \omega_{\{ 1\}}= \begin{pmatrix}
0& 0& -1 & 0\\
 0&  1 & 0  &  0\\
1 & 0 &  0    &0 \\
0 & 0  &0    &  1
\end{pmatrix}, \quad\omega_{\{ 2\}}= \begin{pmatrix}
1& 0& 0 & 0\\
 0&  0 & 0  & -1\\
0& 0 &  1    &0 \\
0 & 1  &0    &  0
\end{pmatrix}, \quad\omega_{\{ 1,2\}}= \begin{pmatrix}
0& 0&-1 & 0\\
 0&  0& 0  & -1\\
1& 0 &  0    &0 \\
0 &1 &0   &  0
\end{pmatrix},$$ $$\omega_{(12)}= \begin{pmatrix}
0& 1& 0 & 0\\
 1&  0& 0  &  0\\
0 & 0 &  0    &1 \\
0 & 0  &1    &  0
\end{pmatrix},
 \omega_{(12)}\omega_{\{ 1\}}=\begin{pmatrix}
0& 1& 0 & 0\\
 0&  0 & -1  & 0\\
0& 0 &  0   &1\\
1& 0  &0    &  0
\end{pmatrix}, \omega_{(12)}\omega_{\{ 2\}} =\begin{pmatrix}
0& 0& 0 & -1\\
 1&  0 & 0  & 0\\
0&1 &  0   &0\\
0& 0  &1  &  0
\end{pmatrix}, \omega_{(12)}\omega_{\{ 1, 2\}} =\begin{pmatrix}
0& 0& 0 & -1\\
 0&  0 & -1  & 0\\
0&1 &  0   &0\\
1& 0  &0 &  0
\end{pmatrix}$$
%
Let 
$$X_{\{1\}}= \{\begin{pmatrix}
1& 0& t & 0\\
 0&  1 & 0  &  0\\
0 & 0 & 1   &0 \\
0 & 0  &0    &  1
\end{pmatrix}\mid t\in \Z\}, \qquad X_{\{2\}}= \{\begin{pmatrix}
1& 0& 0 & 0\\
 0&  1 & 0  & t\\
0 & 0 & 1   &0 \\
0 & 0  &0    &  1
\end{pmatrix}\mid t\in \Z\}$$
$$   X_{(12)\{ 1, 2\}}= \{\begin{pmatrix}
1& 0& 0 & t\\
 0&  1 & t  & 0\\
0 & 0 & 1   &0 \\
0 & 0  &0    &  1
\end{pmatrix}\mid t\in \Z\}, \qquad X_{(12)}= \{\begin{pmatrix}
1& t& 0 & 0\\
 0&  1 & 0  & 0\\
0 & 0 & 1   &0 \\
0 & 0  &-t    &  1
\end{pmatrix}\mid t\in \Z\}.$$
 According to \cite{Ka}, the above four sets correspond to the  positive root system of $C_2$ with respect to the above  Bruhat decomposition. For $K=\F_2$ or $\Z$,  let:
 $$\mathfrak{X}^{\ast}: 0 \subseteq \Span_{K}\{ e^{\ast}_1\} \subseteq  \Span_{K}\{e_2; e_1^{\ast}, e_2^{\ast}\} \subseteq \Span_{K}\{e_1, e_2; e_1^{\ast}, e_2^{\ast}\}.$$
 Let $P_{\mathfrak{X}^{\ast}}(K)$  be defined as the stabilizer of the flag  $\mathfrak{X}^{\ast}$. Moreover, $P_{\mathfrak{X}^{\ast}}(K)$ is characterized by the set of matrices given by
 $$P_{\mathfrak{X}^{\ast}}(K)=\left\{ \begin{pmatrix}
x^{-1}& \alpha_1& b & \alpha_2\\
 0&  a_1 &  \beta_{1}  & b_1 \\
0 & 0 &  x  & 0\\
0 &  c_1 & \beta_{2}  & d_1  
\end{pmatrix}\,\middle\vert\,
\begin{array}{l} g=\begin{bmatrix}
a_1& b_1\\
c_1& d_1
\end{bmatrix} \in \SL_{2}(K),\\
 (\beta_1, \beta_2)^T=xgJ(\alpha_1, \alpha_2)^T, \\
 \textrm{ for } J= \begin{bmatrix}
0& 1\\
-1& 0
\end{bmatrix}\in \SL_{2}(K)
\end{array}\right\}.$$
We let:
\begin{itemize}
\item $N_{\mathfrak{X}^{\ast}}(K)=\{    \begin{pmatrix}
1& \alpha_1& b & \alpha_2\\
 0& 1 &  \beta_{1}  & 0 \\
0 & 0 &  1    & 0\\
0 &  0 & \beta_{2}  &1 
\end{pmatrix} \mid  \beta_1=\alpha_2, \beta_2=-\alpha_1\}$.
\item $ M_{\mathfrak{X}^{\ast}}(K)=\{   \begin{pmatrix}
x^{-1}& 0& 0 & 0\\
 0&  a_1 &  0  & b_1 \\
0 & 0 &  x   & 0\\
0 &  c_1 & 0 & d_1  
\end{pmatrix}  \mid  g=\begin{bmatrix}
a_1& b_1\\
c_1& d_1
\end{bmatrix} \in \SL_{2}(K)\}$.
\end{itemize}
Then 
\begin{itemize}
\item $N_{\mathfrak{X}^{\ast}}(\Z)= X_{(12)} X_{\omega_{(12)}\omega_{\{ 1, 2\}}}X_{\{1\}}$ and $M_{\mathfrak{X}^{\ast}}(\Z)\cap N=X_{\{2\}}$.
\end{itemize}
 Let $P_{\mathfrak{X}^{\ast}}=P_{\mathfrak{X}^{\ast}}(\Z)\Gamma(2)$, $N_{\mathfrak{X}^{\ast}}=N_{\mathfrak{X}^{\ast}}(\Z)\Gamma(2)$, $M_{\mathfrak{X}^{\ast}}=M_{\mathfrak{X}^{\ast}}(\Z) \Gamma(2)$.  It can be seen that 
$P_{\mathfrak{X}^{\ast}}(\Z)\to P_{\mathfrak{X}^{\ast}}(\F_2)$, $M_{\mathfrak{X}^{\ast}}(\Z) \to M_{\mathfrak{X}^{\ast}}(\F_2)$, $N_{\mathfrak{X}^{\ast}}(\Z) \to N_{\mathfrak{X}^{\ast}}(\F_2)$ all are surjective.  Let $\mathfrak{W}'_{\{2\}}=\{\omega_{(12)},\omega_{\{1\}}, \omega_{\emptyset}\}, \mathfrak{W}_{\{2\}}=\{ \omega_{\{2\}}, \omega_{\emptyset}\}$. Then $\mathfrak{W}= \mathfrak{W}_{\{2\}}\mathfrak{W}'_{\{2\}}\mathfrak{W}_{\{2\}}$.
 Hence:
$$\Sp_4(\Z)=\cup_{\omega\in \mathfrak{W}'_{\{2\}}} P_{\mathfrak{X}^{\ast}} \omega P_{\mathfrak{X}^{\ast}}=\cup_{\omega\in\mathfrak{W}'_{\{2\}}} \Gamma(2) P_{\mathfrak{X}^{\ast}}(\Z) \omega P_{\mathfrak{X}^{\ast}}(\Z).$$
\subsubsection{}
  \begin{lemma}
   Let $n\in N_{\mathfrak{X}^{\ast}}(\Z)$, $n_1\in X_{\{2\}} $, $d\in D$. Then:
  \begin{itemize}
  \item[(1)] $\omega_{\{2\}} n \omega_{\{2\}}^{-1}\in N_{\mathfrak{X}^{\ast}}(\Z) \subseteq I$.
   \item[(2)] Let  $ g=\omega_{\{2\}} n_1d$, $ d n_1 \omega_{\{2\}}$, $\omega_{\{2\}} n_1 \omega_{\{2\}}^{-1}$. Then $g\in \Gamma(2)X_{\{2\}} \omega d' X_{\{2\}} \Gamma(2) $, for some $d'\in D$, $\omega=\omega_{\{2\}}$ or $\omega_{\emptyset}$.
  \end{itemize}
  \end{lemma}
   \begin{proof}
  1) By the knowledge of linear algebraic group,  $\omega_{\{2\}} n \omega_{\{2\}}^{-1}\in  N_{\mathfrak{X}^{\ast}}(\Q_2) \cap \Sp_{2m}(\Z)=N_{\mathfrak{X}^{\ast}}(\Z) $. 
2) It is essential follows from the case $m=1$. 
\end{proof}
 \begin{lemma}\label{Nwd}
   Let $\omega \in  \mathfrak{W}'_{\{2\}}$, $n\in N_{\mathfrak{X}^{\ast}}(\Z)$, $d\in D$. Then $$\omega n d\in I d'\omega'  N_{\mathfrak{X}^{\ast}}(\Z) \Gamma(2),$$ for some $d'\in D$, $\omega'=\omega$ or $\omega'=\omega_{\emptyset}$.
  \end{lemma}
  \begin{proof}
  \begin{itemize}
  \item If $\omega=1$, then  $nd$ belongs to $ N_{\mathfrak{X}^{\ast}}(\Z) d$.
  \item If  $\omega=\omega_{(12)}$, we let   $X'_{(12)}=X_{\omega_{(12)}\omega_{\{ 1, 2\}}}X_{\{1\}} $ and write $n=n_1n_2$, for $n_2\in X_{(12)}$ and $n_1\in X'_{(12)}$. Then:
  \begin{align*}
\omega_{(12)} nd & =\omega_{(12)} n_1n_2d = [\omega_{(12)} n_1\omega_{(12)}^{-1}]\cdot  [ \omega_{(1,2)}n_2d]\in N \Gamma(2)  d'\omega'  X_{(12)}  \Gamma(2)
\end{align*}
for some $d'\in D$ and $\omega'= \omega_{(1,2)}$ or $ \omega_{\emptyset}$.
  \item If  $\omega=\omega_{\{1\}}$, the proof is similar as above.
 \end{itemize}
  \end{proof}

  \begin{lemma}\label{Nwdg2}
   Let $\omega \in  \mathfrak{W}$, $n\in N$, $d\in D$. Then $$\omega n d\in I d'\omega'I,$$ for some $d'\in D$, $\omega'\in  \mathfrak{W}$.
  \end{lemma}
  \begin{proof}
  Assume $\omega \in  \mathfrak{W}_{\{2\}}\omega_0  \mathfrak{W}_{\{2\}}$, for some  $\omega_0  \in \mathfrak{W}'_{\{2\}}$. Write $n=n_1n_2$, for $n_2\in X_{\{2\}}$, $n_1\in N_{\mathfrak{X}^{\ast}}(\Z)$. We denote the following:
  \begin{itemize}
  \item $  \omega_{\{1\}}'$ as as either $\omega_{\{1\}}$ or $\omega_{\emptyset}$, 
  \item  $  \omega_{\{2\}}'$ as as either $\omega_{\{2\}}$ or $\omega_{\emptyset}$, 
  \item $  \omega_{(12)}'$ as as either $\omega_{(12)}$ or $\omega_{\emptyset}$.
  \end{itemize} 
  \begin{itemize}
  \item[(1)] Assume $\omega_0 =\omega_{\emptyset}$. Then $\omega=\omega_{\{2\}}$, or $\omega=\omega_{\emptyset}$. If $\omega=\omega_{\emptyset}$, then $nd\in IdI$. 
  If $\omega=\omega_{\{2\}}$, then 
  $$\omega nd=\omega_{\{2\}} n_1n_2 d=\omega_{\{2\}} n_1\omega_{\{2\}}^{-1}\omega_{\{2\}} n_2 d \in N \Gamma(2) X_{\{2\}} \omega' d'  X_{\{2\}} \Gamma(2) \subseteq I \omega' d'I.$$
   \item[(2)]  Assume  $\omega_0 =\omega_{(12)}$.  
  \begin{itemize}
\item If  $\omega= \omega_{(12)} $, by the above lemma,  $\omega n d\in I\omega_{(12)}' d' I$.
\item If $\omega= \omega_{(12)}\omega_{\{2\}}$, then
\begin{align*}
\omega n d&= \omega_{(12)} \omega_{\{2\}} n_1n_2d\\
&=\omega_{(12)} [\omega_{\{2\}} n_1 \omega_{\{2\}}^{-1} ]\omega_{\{2\}} n_2d  \\
&\stackrel{n_1'= \omega_{\{2\}} n_1 \omega_{\{2\}}^{-1} \in N_{\mathfrak{X}^{\ast}}(\Z)}{=} \omega_{(12)}n_1' \omega_{\{2\}} n_2d  \\
&\in  I\omega_{(12)}n_1' X_{\{2\}} d_1\omega_{\{2\}}'  I \\
&\subseteq  I\omega_{(12)}n_1'd_1\omega_{\{2\}}'  I \qquad\qquad( \omega_{(12)} X_{\{2\}} \omega_{(12)}^{-1} =X_{\{1\}})\\
&\subseteq  I\omega'_{(12)}d_1'N_{\mathfrak{X}^{\ast}}(\Z)\omega_{\{2\}}'  I\\
&\subseteq I\omega'_{(12)}d'\omega_{\{2\}}' I.
\end{align*}
\item If $\omega=  \omega_{\{2\}} \omega_{(12)}$, then 
\begin{align*}
\omega n d&=\omega_{\{2\}} \omega_{(12)} n_2n_1d\\
&\xlongequal[n_1'\in X_{\{1\}}]{\omega_{(12)} n_2=n_1'\omega_{(12)}} \omega_{\{2\}} n_1'  \omega_{(12)}n_1d \\
& \in \Gamma(2)  N_{\mathfrak{X}^{\ast}}(\Z) \omega_{\{2\}}  N d_1   \omega_{(12)}' I\\
&=\Gamma(2)  N_{\mathfrak{X}^{\ast}}(\Z) \omega_{\{2\}} N_{\mathfrak{X}^{\ast}}(\Z) X_{\{2\}}   d_1  \omega_{(12)}' I\\
&\subseteq I\omega_{\{2\}} X_{\{2\}}   d_1    \omega_{(12)}' I \qquad\qquad  ( \omega_{\{2\}} X_{\{2\}}   d_1   \subseteq  \Gamma(2) X_{\{2\}}   \omega_{\{2\}}' d_2 X_{\{2\}} \Gamma(2)) \\
&\subseteq I  \omega'_{\{2\}} d_2   X_{\{2\}}    \omega_{(12)}' I \qquad\qquad  ( X_{\{2\}}    \omega_{(12)}' \subseteq  \omega_{(12)}' N) \\
&  \subseteq I  \omega'_{\{2\}}  \omega_{(12)}' d'     I.
\end{align*}
\item If $\omega=  \omega_{\{2\}} \omega_{(12)}\omega_{\{2\}}$, then 
\begin{align*}
&\omega n d\\
&=\omega_{\{2\}} \omega_{(12)}\omega_{\{2\}} n_1n_2d\\
&\xlongequal{n_1'=\omega_{\{2\}} n_1\omega_{\{2\}}^{-1}}\omega_{\{2\}} \omega_{(12)} n_1'\omega_{\{2\}}n_2d \qquad\qquad\qquad\qquad\qquad\qquad(n_1'\in N_{\mathfrak{X}^{\ast}}(\Z)) \\
&\xlongequal{\omega_{\{2\}}n_2d=\gamma_1 n_2' d_1\omega_{\{2\}}' i_1  }\omega_{\{2\}} \omega_{(12)} n_1'\gamma_1 n_2' d_1\omega_{\{2\}}' i_1 \qquad\qquad  ( \gamma_1\in \Gamma(2), n_2'\in X_{\{2\}}, i_1\in I)\\
&=\gamma_1'\omega_{\{2\}}n_1'' \omega_{(12)} n_1'  d_1\omega_{\{2\}}' i_1  \qquad\qquad \qquad\qquad (\omega_{(12)} n_2'=n_1'' \omega_{(12)},   n_1''\in X_{\{1\}} ) \\
&\xlongequal[n \in N, n_1'''\in  N_{\mathfrak{X}^{\ast}}(\Z)]{\omega_{(12)}n_1'  d_1=\gamma_3 n d_2\omega_{(12)}' n_1''' \gamma_4 }
\gamma_1'\omega_{\{2\}}n_1'' \gamma_3 n d_2\omega_{(12)}' n_1''' \gamma_4 \omega_{\{2\}}' i_1\\
&\in I \omega_{\{2\}} X_{\{2\}}d_2\omega_{(12)}' N_{\mathfrak{X}^{\ast}}(\Z)\omega_{\{2\}}'I\\
&\subseteq I \omega_{\{2\}} X_{\{2\}}d_2\omega_{(12)}' \omega_{\{2\}}'I\\
&\subseteq I d' \omega_{\{2\}}' X_{\{2\}} \omega_{(12)}' \omega_{\{2\}}'I  \\
& \subseteq I d' \omega_{\{2\}}' X_{\{2\}}  \omega_{\{2\}}'I \cup I d' \omega_{\{2\}}' \omega_{(12)} X_{\{1\}} \omega_{\{2\}}'I\\
&\subseteq I d' \omega_{\{2\}}' X_{\{2\}}  \omega_{\{2\}}'I \cup I d' \omega_{\{2\}}' \omega_{(12)} \omega_{\{2\}}'  I\\
&\subseteq I d'' \omega_{\{2\}}'I \cup I d' \omega_{\{2\}}' \omega_{(12)} \omega_{\{2\}}'  I.
\end{align*}
\end{itemize} 
  \item[(3)]  Assume  $\omega_0 =\omega_{\{1\}}$.  
  \begin{itemize}
\item If  $\omega= \omega_{\{1\}}= \omega_{\{2\}} \omega_{\{1\}}\omega_{\{2\}} $, by the above lemma,  $\omega_1 n d\in I\omega' d' I$, for some $\omega'=\omega_{\{1\}} $ or $\omega_{\emptyset} $.
\item If $\omega= \omega_{\{1\}} \omega_{\{2\}}=\omega_{\{2\}} \omega_{\{1\}}$, then
\begin{align*}
&\omega n d\\
&= \omega_{\{1\}} \omega_{\{2\}} n_1n_2d\\
&=\omega_{\{1\}} [\omega_{\{2\}} n_1 \omega_{\{2\}}^{-1} ]\omega_{\{2\}} n_2d  \\
&\stackrel{n_1'= \omega_{\{2\}} n_1 \omega_{\{2\}}^{-1} \in N_{\mathfrak{X}^{\ast}}(\Z)}{=} \omega_{\{1\}}n_1' \omega_{\{2\}} n_2d  \\
&\in  I\omega_{\{1\}}n_1' X_{\{2\}} d'\omega_{\{2\}}'  I \\
&\subseteq  I\omega_{\{1\}}n_1'd'\omega_{\{2\}}'  I\\
&\subseteq  I\omega'_{\{1\}}d'N_{\mathfrak{X}^{\ast}}(\Z)\omega_{\{2\}}'  I\\
&\subseteq I\omega'_{\{1\}}d'\omega_{\{2\}}' I.
\end{align*}
 \end{itemize} 
  \end{itemize}
  \end{proof}
  
\subsubsection{The final proof of Lemma \ref{Iwah}: $m=2$}\label{thewholeproof}
According to Theorem\ref{SPQdecom} and (\ref{sqspZ}),  for $A\in \Sp_{4}(\Q)$, we write:
$$A=g\mathfrak{d}g', g=\mathfrak{i} \omega_g \mathfrak{n}, g'=\mathfrak{n}' \omega_{g'} \mathfrak{i}',$$
for some $g,g'\in  \Sp_{4}(\Z)$,  $ \mathfrak{i},  \mathfrak{i}'\in I$, $ \mathfrak{n}, \mathfrak{n}'\in N$, $ \mathfrak{d}\in D_0$, and  $\omega_g,\omega_{g'}\in \mathfrak{W}$.   Hence 
$$A= \mathfrak{i} \omega_g \mathfrak{n}\mathfrak{d}\mathfrak{n}' \omega_{g'}  \mathfrak{i}'= \mathfrak{i} \omega_g \mathfrak{n}[\mathfrak{d}\mathfrak{n}' \mathfrak{d}^{-1}] \omega_{g'}[(\omega_{g'}^{-1})\mathfrak{d}\omega_{g'}]  \mathfrak{i}'.$$
Note that $\mathfrak{d}\mathfrak{n}'\mathfrak{d}^{-1}\in N$.  So $ \mathfrak{i} \omega_g \mathfrak{n}[\mathfrak{d}\mathfrak{n}' \mathfrak{d}^{-1}] \omega_{g'} \in \Sp_{4}(\Z)$. Let us write it as: 
$$\mathfrak{i}\omega_g \mathfrak{n}[\mathfrak{d}\mathfrak{n}' \mathfrak{d}^{-1}]\omega_{g'} =i\omega n,$$
for some $i\in I$, $\omega\in \mathfrak{W}$, $n\in N$. Write $d=(\omega_{g'}^{-1})\mathfrak{d}\omega_{g'}\in D$. Then:
$$A=i\omega n d \mathfrak{i}'.$$
So it suffices to show $\omega nd$ has the decomposition, which has been proved in Lemma \ref{Nwdg2}.

\subsection{The proof of Prop.\ref{Iwah}: $m\geq 3$}  \label{sec:induction3} 
The methodology applied here closely follows that of the preceding case.
\subsubsection{} For $K=\F_2$ or $\Z$,  let:
 $$\mathfrak{X}^{\ast}: 0 \subseteq \Span_{K}\{ e^{\ast}_1\} \subseteq  \Span_{K}\{e_2, \cdots, e_m; e_1^{\ast}, e_2^{\ast} \cdots, e^{\ast}_m\} \subseteq U= \Span_{K}\{e_1, \cdots, e_m; e_1^{\ast}, \cdots, e^{\ast}_m\} .$$ 
  Let $P_{\mathfrak{X}^{\ast}}(K)$  be defined as the stabilizer of the flag  $\mathfrak{X}^{\ast}$ in $\Sp(U)$. Under the basis $\{ e_1, e_2, \cdots, e_m; e_2^{\ast}, \cdots, e_m^{\ast}, e_1^{\ast}\}$ of $U$, 
$$P_{\mathfrak{X}^{\ast}}(K)=\{  \begin{pmatrix}
x^{-1} & \alpha & b\\
 & g      &  \beta^T\\
  &      &x
\end{pmatrix} \mid  g\in \Sp_{2m-2}(K), \beta^T=x gJ\alpha^T, \textrm{ for } J= \begin{bmatrix}
0& 1_{m-1}\\
-1_{m-1}& 0
\end{bmatrix}\in \Sp_{2m-2}(K)\};$$
As $K=\F_2$ or $\Z$, the above $x=1$. Let:
\begin{itemize}
\item $N_{\mathfrak{X}^{\ast}}(K)=\{ \begin{pmatrix}
1 & \alpha & b\\
 &1    &  \beta^T\\
  &      &1
\end{pmatrix} \mid   \beta^T= J\alpha^T\}$.
\item $ M_{\mathfrak{X}^{\ast}}(K)=\{  \begin{pmatrix}
1 & &\\
 & g      &  \\
  &      &1
\end{pmatrix}\mid   g\in \Sp_{2m-2}(K)\}$.
\item $N'_{\mathfrak{X}^{\ast}}(K)$: the unipotent subgroup of $M_{\mathfrak{X}^{\ast}}(K)$. 
\item $\mathfrak{W}_{\mathfrak{X}^{\ast}}$: the corresponding subset of the extended affine Wely group of $M_{\mathfrak{X}^{\ast}}(K)$.
\item $\mathfrak{W}'_{\mathfrak{X}^{\ast}}=\{\omega_{(1,2)}, \omega_{\{ 1\}}, \omega_{\emptyset}\}$.
 \item $P_{\mathfrak{X}^{\ast}}=\Gamma(2) P_{\mathfrak{X}^{\ast}}(\Z)$.
\end{itemize}
 Then:
$$\Sp_{2m}(\Z)=\cup_{\omega\in \mathfrak{W}'_{\mathfrak{X}^{\ast}}} P_{\mathfrak{X}^{\ast}} \omega P_{\mathfrak{X}^{\ast}}=\cup_{\omega\in \mathfrak{W}'_{\mathfrak{X}^{\ast}}} \Gamma(2) P_{\mathfrak{X}^{\ast}}(\Z)  \omega P_{\mathfrak{X}^{\ast}}(\Z) .$$
Under the basis $\{ e_1, e_2, \cdots, e_m; e_2^{\ast}, \cdots, e_m^{\ast}, e_1^{\ast}\}$, we have:
\[\omega_{\{ 1\}}=\begin{pmatrix}
&  & -1\\
 & 1_{2m-2}      &  \\
  1&      &
\end{pmatrix}, \omega_{(12)}=\begin{pmatrix}
\begin{bmatrix}
 0           &             1 &           \\
  1           &            0 &            \\
              &             &1_{m-2}
              \end{bmatrix}                         &  \\
             & \begin{bmatrix}
                             &           &  1 \\
                                 &1_{m-2}   &  \\
                           1      &           & \end{bmatrix}
\end{pmatrix}.  \] 
Let:
\begin{itemize}
\item $X_{\{1\}}=\{\begin{pmatrix}
1 & 0& t\\
 &1_{2m-2}    &  0\\
  &      &1
  \end{pmatrix} \mid t\in \Z\}$, $X_{(12)}=\{\begin{pmatrix}
\begin{bmatrix}
 1         &             t &           \\
  0           &            1 &            \\
              &             &1_{m-2}
              \end{bmatrix}                         &  \\
             & \begin{bmatrix}
                            1 &           &  -t \\
                                 &1_{m-2}   &  \\
                            &           & 1\end{bmatrix}
\end{pmatrix}
 \mid t\in \Z\}$
\end{itemize}
\begin{lemma}
   Let $n\in N_{\mathfrak{X}^{\ast}}(\Z)$, $n_1\in N'_{\mathfrak{X}^{\ast}}(\Z) $, $d\in D$, $\omega\in \mathfrak{W}_{\mathfrak{X}^{\ast}}$. Then:
  \begin{itemize}
  \item[(1)] $\omega n \omega^{-1}\in N_{\mathfrak{X}^{\ast}}(\Z) \subseteq I$.
   \item[(2)] Let  $ g=\omega n_1d$, $ d n_1 \omega$, $\omega n_1 \omega^{-1}$. Then $g\in \Gamma(2) N'_{\mathfrak{X}^{\ast}}(\Z)  \omega' d' N'_{\mathfrak{X}^{\ast}}(\Z)  \Gamma(2) $, for some $d'\in D$, $\omega' \in \mathfrak{W}_{\mathfrak{X}^{\ast}}$.
  \end{itemize}
  \end{lemma}
   \begin{proof}
  Part (1) follows from the Bruhat decomposition. Part (2) follows from the induction.
\end{proof}
\begin{lemma}\label{Nwd2m}
   Let $\omega \in   \mathfrak{W}'_{\mathfrak{X}^{\ast}}$, $n\in N_{\mathfrak{X}^{\ast}}(\Z)$, $d\in D$. Then $$\omega n d\in I d'\omega'  N_{\mathfrak{X}^{\ast}}(\Z) \Gamma(2),$$ for some $d'\in D$, $\omega'=\omega$ or $\omega'=\omega_{\emptyset}$.
  \end{lemma}
  \begin{proof}
  \begin{itemize}
  \item If $\omega=1$, then  $nd$ belongs to $ N_{\mathfrak{X}^{\ast}}(\Z) d$.
  \item If  $\omega=\omega_{(12)}$, we let   $X'_{(12)}$ be the complement of $X_{(12)}$ in $N_{\mathfrak{X}^{\ast}}(\Z) $. Write $n=n_1n_2$, for $n_2\in X_{(12)}$ and $n_1\in X'_{(12)}$. Then:
  \begin{align*}
\omega_{(12)} nd & =\omega_{(12)} n_1n_2d = [\omega_{(12)} n_1\omega_{(12)}^{-1}]\cdot  [ \omega_{(1,2)}n_2d]\in N \Gamma(2)  d'\omega'  X_{(12)}  \Gamma(2)
\end{align*}
for some $d'\in D$ and $\omega'= \omega_{(1,2)}$ or $ \omega_{\emptyset}$.
  \item If  $\omega=\omega_{\{1\}}$, the proof is similar as above.
 \end{itemize}
  \end{proof}
 \begin{lemma}\label{Nwdg2m1}
   Let $\omega \in  \mathfrak{W}$, $n\in N$, $d\in D$. Then $$\omega n d\in I d'\omega'I,$$ for some $d'\in D$, $\omega'\in  \mathfrak{W}$.
  \end{lemma}
  \begin{proof}
  Assume $\omega \in   \mathfrak{W}_{\mathfrak{X}^{\ast}}\omega_0   \mathfrak{W}_{\mathfrak{X}^{\ast}}$, for some  $\omega_0  \in  \mathfrak{W}'_{\mathfrak{X}^{\ast}}$. Write $n=n_1n_2$, for $n_2\in N'_{\mathfrak{X}^{\ast}}(\Z)$, $n_1\in N_{\mathfrak{X}^{\ast}}(\Z)$. We denote the following:
  \begin{itemize}
  \item $  \omega_{\{1\}}'$ as as either $\omega_{\{1\}}$ or $\omega_{\emptyset}$, 
  \item $  \omega_{(12)}'$ as as either $\omega_{(12)}$ or $\omega_{\emptyset}$.
  \end{itemize} 
  \begin{itemize}
  \item[(1)] Assume $\omega_0 =\omega_{\emptyset}$. Then $\omega\in   \mathfrak{W}_{\mathfrak{X}^{\ast}} $.  By induction, we have:
  $$\omega nd\in I \omega' d'I.$$
   \item[(2)]  Assume  $\omega_0 =\omega_{(12)}$.  
  \begin{itemize}
\item If  $\omega= \omega_{(12)} $, by the above lemma,  $\omega n d\in I\omega_{(12)}' d' I$.
\item If $\omega= \omega_{(12)}\omega_{2}$, for $\omega_{2}\in \mathfrak{W}_{\mathfrak{X}^{\ast}}$,   then
\begin{align*}
\omega n d&= \omega_{(12)} \omega_{2} n_1n_2d\\
&=\omega_{(12)} [\omega_{2} n_1 \omega_{2}^{-1}  ]\omega_{2} n_2d  \\
&\stackrel{n_1'= \omega_{2} n_1 \omega_{2}^{-1}  \in N_{\mathfrak{X}^{\ast}}(\Z)}{=} \omega_{(12)}n_1' \omega_{2} n_2d  \\
&\in  I\omega_{(12)}n_1' N'_{\mathfrak{X}^{\ast}}(\Z) d_1\omega_{2}'  I \qquad\qquad( \omega_{2}'\in \mathfrak{W}_{\mathfrak{X}^{\ast}})  \\
&\subseteq  I\omega_{(12)}n_1'd_1\omega_{2}'  I \qquad\qquad\qquad( \omega_{(12)} N'_{\mathfrak{X}^{\ast}}(\Z) \omega_{(12)}^{-1} \subseteq  N)\\
&\subseteq  I\omega'_{(12)}d'N_{\mathfrak{X}^{\ast}}(\Z)\omega_{2}'  I \qquad\qquad\qquad( \textrm{ Lemma \ref{Nwd2m}})\\
&\subseteq I\omega'_{(12)}d'\omega_{2}' I.
\end{align*}
\item If $\omega=  \omega_{2} \omega_{(12)}$,  for $\omega_{2}\in \mathfrak{W}_{\mathfrak{X}^{\ast}}$,  then 
\begin{align*}
\omega n d&=\omega_{2} \omega_{(12)} n_2n_1d\\
&\xlongequal[n'\in N]{\omega_{(12)} n_2=n'\omega_{(12)}} \omega_{2} n'  \omega_{(12)}n_1d \\
& \in \Gamma(2)  \omega_{2}  N d_1   \omega_{(12)}' I\\
& = \Gamma(2)  N_{\mathfrak{X}^{\ast}}(\Z) \omega_{2}  N'_{\mathfrak{X}^{\ast}}(\Z) d_1   \omega_{(12)}' I\\
&\subseteq I\omega_{2} d_2N'_{\mathfrak{X}^{\ast}}(\Z)      \omega_{(12)}' I \qquad\qquad\qquad\qquad  ( \omega_{2} N'_{\mathfrak{X}^{\ast}}(\Z)  d_1   \subseteq  \Gamma(2) N'_{\mathfrak{X}^{\ast}}(\Z)  \omega_{2}' d_2 N'_{\mathfrak{X}^{\ast}}(\Z) \Gamma(2)) \\
&  \subseteq I  \omega_{2}' d_2      \omega_{(12)}' I \qquad\qquad  \qquad\qquad\qquad ( N'_{\mathfrak{X}^{\ast}}(\Z)   \omega_{(12)}' \subseteq  \omega_{(12)}' N) \\
&=I  \omega_{2}' \omega_{(12)}' d' I. 
\end{align*}
\item If $\omega=  \omega'_{2} \omega_{(12)}\omega_{2}$, for $\omega_{2}, \omega_{2}'\in \mathfrak{W}_{\mathfrak{X}^{\ast}}$,  then 
\begin{align*}
&\omega n d\\
&=\omega_{2}' \omega_{(12)}\omega_{2}n d\\
&\subseteq  \omega_{2}' I  \omega_{2}'' d_1      \omega_{(12)}' I \qquad\qquad  ( \textrm{ the above case})\\
&\subseteq I \omega_{2}'  N'_{\mathfrak{X}^{\ast}}(\Z)   \omega_{2}'' d_1      \omega_{(12)}' I  \qquad\qquad  (   \omega_{2}' N_{\mathfrak{X}^{\ast}}(\Z) (\omega_{2}')^{-1} \subseteq N_{\mathfrak{X}^{\ast}}(\Z) )\\
&\subseteq  I \omega_{2}'''   d'_1    N'_{\mathfrak{X}^{\ast}}(\Z)   \omega_{(12)}' I \qquad\qquad  ( \textrm{ the induction})\\
&\subseteq  I \omega_{2}'''   d'_1      \omega_{(12)}' I \qquad\qquad ( (\omega_{(12)}')^{-1}  N'_{\mathfrak{X}^{\ast}}(\Z)  \omega_{(12)}' \subseteq N )\\
&\subseteq  I \omega'   d'   I.
\end{align*}
\end{itemize} 
  \item[(3)]  Assume  $\omega_0 =\omega_{\{1\}}$. The proof is similar as the case $\omega_{(12)}$.  
  \end{itemize}
  \end{proof}
\subsubsection{} According to Theorem\ref{SPQdecom} and (\ref{sqspZ}),   for $A\in \Sp_{2m}(\Q)$, we write:
$$A=g\mathfrak{d}g', g=\mathfrak{i} \omega_g \mathfrak{n}, g'=\mathfrak{n}' \omega_{g'} \mathfrak{i}',$$
for some $g,g'\in  \Sp_{2m}(\Z)$,  $ \mathfrak{i},  \mathfrak{i}'\in I$, $ \mathfrak{n}, \mathfrak{n}'\in N$, $ \mathfrak{d}\in D_0$, and  $\omega_g,\omega_{g'}\in \mathfrak{W}$.   Hence 
$$A= \mathfrak{i} \omega_g \mathfrak{n}\mathfrak{d}\mathfrak{n}' \omega_{g'}  \mathfrak{i}'= \mathfrak{i} \omega_g \mathfrak{n}[\mathfrak{d}\mathfrak{n}' \mathfrak{d}^{-1}] \omega_{g'}[(\omega_{g'}^{-1})\mathfrak{d}\omega_{g'}]  \mathfrak{i}'.$$
Note that $\mathfrak{d}\mathfrak{n}'\mathfrak{d}^{-1}\in N$.  So $ \mathfrak{i} \omega_g \mathfrak{n}[\mathfrak{d}\mathfrak{n}' \mathfrak{d}^{-1}] \omega_{g'} \in \Sp_{2m}(\Z)$. Let us write it as: 
$$\mathfrak{i}\omega_g \mathfrak{n}[\mathfrak{d}\mathfrak{n}' \mathfrak{d}^{-1}]\omega_{g'} =i\omega n,$$
for some $i\in I$, $\omega\in \mathfrak{W}$, $n\in N$. Write $d=(\omega_{g'}^{-1})\mathfrak{d}\omega_{g'}\in D$. Then:
$$A=i\omega n d \mathfrak{i}'.$$ 
So it suffices to show that  $\omega nd$ has the Iwahori decomposition, which has been proved in Lemma \ref{Nwdg2m1}.


\begin{thebibliography}{99}

\bibitem{AuPr} A.-M.Aubert, T.Przebinda,
 {\it A reverse engineering approach to the Weil representation}, Central Eur. J. Math. 12, 1500-1585 (2014).

\bibitem{Be} Y. Benoist,
 {\it On the rational symplectic group}, 	arXiv:2208.02549.
 
 
\bibitem{Bl} L.Blasco,
 {\it Paires duales r\'eductives en caract\'eristique $2$}, M\'em. Soc. Math. France (N.S.) No. 52, 1-73(1993).
 
 \bibitem{BoTi}  A.Borel, J. Tits,
 {\it Groupes r\'eductifs}, Publ. Math. IHES 27 , 55-150(1965). 
 
 
\bibitem{BushH} C.J.Bushnell, G.Henniart,
 {\it The local langlands conjecture for $GL(2)$},
 Grundlehren Math. Wiss. 335, Springer-Verlag, Berlin, 2006.
 
\bibitem{HeOi} G.Henniart, M.Oi,
 {\it Simple supercuspidal L-packets of symplectic groups over dyadic fields}, arXiv:2307.15248.



 \bibitem{Ho}  R.Howe,
 {\it Another look at the local -correspondence for an unramified dual pair},
Festschrift in honor of I. I.Piatetski-Shapiro on the occasion of his sixtieth birthday, Part I, 93124, Israel Math. Conf. Proc., 2, Weizmann,
Jerusalem, 1990.

\bibitem[Ca1]{Ca1} R.W.Carter,
 {\it Finite Groups of Lie Type: Conjugacy classes and complex characters}, John Wiley \& Sons, New York, 1985.
 
\bibitem{Ca} P.Cartier,
 {\it Quantum Mechanical Commutation Relations and Theta Functions}, Proc. of
Symp. Pure Mathematics, 9, Amer. Math. Soc., , pp. 361-38 (1966).

\bibitem{DiMi}  F.Digne, J.Michel,  
 {\it Representations of Finite Groups of Lie Type}, London Mathematical Society Student Texts, Vol. 95. (2020). 


\bibitem{GeLy} A.Genestier, S.Lysenko,
{\it Geometric Weil representation in characteristic two}, J. Inst. Math. Jussieu 11, 221-271(2012).

\bibitem{GuHa} S.Gurevich, R.Hadani,
{\it The Weil representation in characteristic two}, Adv. Math. 230 , 894-926 (2012).

\bibitem{Ig} J.-I.Igusa, 
{\it Theta Functions}, Grundlehren der Mathematischen Wissenschaften, Vol. 194, 1972. 


\bibitem{KaTa} E.Kaniuth, K.F.Taylor,
 {\it  Induced Representations of Locally Compact Groups}, Cambridge Tracts in Math., vol.197, Cambridge University Press, Cambridge, 2013.
 
\bibitem{Ka} C. Kassel,
{\it A braid-like presentation of the integral Steinberg group of type $C_2$},  J. Algebra 607 (2022), 344-359.  


\bibitem{Kud2} S.Kudla,
 {\it  Notes on the local theta correspondence(lectures at the European School in Group Theory)}, preprint, available at http://www.math.utotonto.ca/~skudla/castle.pdf, 1996.

 \bibitem{LiVe} G.Lion, M.Vergne,
 {\it The Weil representation, Maslov index and Theta series},
 Progress in Math., vol.6, Birkhauser, Boston, 1980.

\bibitem{MVW} C.Moeglin, M.-F.Vign\'eras, and J.-L.Waldspurger,
{\it Correspondances de Howe sur un corps $p$-adique},
Lecture Notes in Math. 1291, Springer-Verlag, Berlin, 1987.

\bibitem{Mu} D.Mumford,
{\it Tata Lectures on Theta III}, With the Collaboration of Madhav Nori and Peter Norman, in: Progress in Mathematics, vol. 97, Birkhäuser Boston, Inc., Boston, MA, 1991.  


\bibitem{Per} P.Perrin,
{\it Repr\'esentations de Schr\" odinger. Indice de Maslov et groupe metaplectique},
in Non Commutative Harmonic Analysis and Lie Groups, Proc.( Marseille-Luming 1980), Lecture Notes in Math. 880,
 Springer-Verlag, Berlin, 1981.


\bibitem{Rao} R.R.Rao,
{\it On some explicit formulas in the theory of Weil representation},
 Pacific J. Math.  157 , 335-371 (1993).



\bibitem{Ta} S.Takeda,
{\it On the lattice model of the Weil representation and the Howe duality conjecture}, J.Ramanujan Math. Soc. 29, no.3, 321-378(2014).

\bibitem{TaWo} S.Takeda,  A.Wood,
{\it Hecke algebra correspondences for the metaplectic group}, Trans. Amer. Math. Soc. 370, 1101-1121 (2018). 


\bibitem{Vi} M.-F.Vign\'eras,
{\it  The pro-p-Iwahori Hecke algebra of a reductive p-adic group I}, Compos. Math. 152, no. 4, 693-753 (2016).


\bibitem{Wa} J.-L.Waldspurger, 
{\it D\'emonstration d'une conjecture de dualit\'e de Howe dans le cas $p$-adique, $p\neq 2$, in: Festschrift in Honor of I.I. Piatetski-Shapiro on the Occasion of His Sixtieth Birthday, Part I}, Israel Math. Conf. Proc., vol. 2, Weizmann, Jerusalem, 1990, pp. 267-324. 


\bibitem{We} A.Weil,
{\it  Sur certains groupes d'op\'erateurs unitaires}, Acta Math. 111 , 143-211 (1964).

\bibitem{Wo} A.Wood,
{\it  A minimal even type of the $2$-adic Weil representation}, Math. Z. 277, 257-283 (2014).

\bibitem{Ya} J.-H.Yang,
{\it Lattice Representations of the Heisenberg Group $\mathbb{H}_{\R}^{(g,h)}$}, Math. Annalen, 317,309-323(2000).


\end{thebibliography}
\end{document}